\documentclass[leqno,12pt]{article}
\usepackage{amsthm,amsfonts,amssymb,amsmath}
\usepackage[all]{xy}
\usepackage{graphicx}

\newcommand{\BC}{{\mathbb {C}}}

\newcommand{\BF}{{\mathbb {F}}}

\newcommand{\BN}{{\mathbb {N}}}

\newcommand{\BP}{{\mathbb {P}}}
\newcommand{\BQ}{{\mathbb {Q}}}
\newcommand{\BR}{{\mathbb {R}}}

\newcommand{\BZ}{{\mathbb {Z}}}

\newcommand{\CC  }{{\mathcal {C}}}
\newcommand{\CD}{{\mathcal {D}}}

\newcommand{\CF}{{\mathcal {F}}}

\newcommand{\CL}{{\mathcal {L}}}
\newcommand{\CM}{{\mathcal {M}}}

\newcommand{\CO}{{\mathcal {O}}}
\newcommand{\CP}{{\mathcal {P}}}

\newcommand{\CR }{{\mathcal {R}}}
\newcommand{\CS}{{\mathcal {S}}}
\newcommand{\CT}{{\mathcal {T}}}

\newcommand{\CX}{{\mathcal {X}}}
\newcommand{\CY}{{\mathcal {Y}}}
\newcommand{\CZ}{{\mathcal {Z}}}

\newcommand{\RN}{{\mathrm {N}}}

\newcommand{\Ch}{{\mathrm{Ch}}}

\renewcommand{\div}{{\mathrm{div}}}

\newcommand{\End}{{\mathrm{End}}}

\newcommand{\Frob}{{\mathrm{Frob}}}

\newcommand{\Gal}{{\mathrm{Gal}}}

\newcommand{\Jac}{{\mathrm{Jac}}}

\newcommand{\NS}{{\mathrm{NS}}}

\newcommand{\ord}{{\mathrm{ord}}}

\newcommand{\rank}{{\mathrm{rank}}}

\newcommand{\Pic}{\mathrm{Pic}}

\newcommand{\Proj}{\mathrm{Proj}}

\newcommand{\Spec}{{\mathrm{Spec}}}

\newcommand{\wt}{\widetilde}
\newcommand{\wh}{\widehat}
\newcommand{\pp}{\frac{\partial\bar\partial}{\pi i}}
\newcommand{\pair}[1]{\langle {#1} \rangle}

\newcommand{\lra}{\longrightarrow}

\theoremstyle{plain}
\newtheorem{theorem}{Theorem}[subsection]
\newtheorem{proposition}[theorem]{Proposition}
\newtheorem{lemma}[theorem]{Lemma}
\newtheorem{corollary}[theorem]{Corollary}
\newtheorem{conjecture}[theorem]{Conjecture}

\theoremstyle{definition}

\theoremstyle{remark}

\numberwithin{equation}{subsection}

\title{Gross--Schoen Cycles and Dualising Sheaves}
\author{Shou-Wu Zhang\\
Department of Mathematics\\
Columbia University\\
New York, NY 10027}
\date{November 17, 2007}
\begin{document}

\maketitle \tableofcontents

\section{Introduction and statements of results}
The aim of this paper is to study  the modified diagonal cycle in
the triple product of  a curve over a global field defined by Gross
and Schoen in \cite{gross-schoen}. Our main result is an identity
between the height of this cycle and the self-intersection of the
relative dualising sheaf. We have  some applications to the
following   problems in number theory and algebraic geometry:
\begin{itemize}
\item {\em Gillet--Soul\'e and Bogomolov's conjectures  for
 heights of cycles and points.} We will show that the
 Gillet--Soul\'e's arithmetic standard conjecture \cite{gillet-soule2}
 gives a
 lower bound for the admissible self-dualising sheaf for arithmetic surfaces in
 term of local integrations. This gives an approach toward an effective version of
 Bogomolov conjecture \cite{ullmo, equidistribution}. By applying Noether's formula, this will
 also give an alternative approach toward
 a slope inequality for Hodge bundles (or Faltings heights)
 on moduli space of curves, other than using stability in geometric
 invariant theory \cite{cornalba-harris, xiao}.
\item  {\em Beilinson--Bloch's conjectures for special
values of L-series and cycles.} By Beilinson--Block and Tate's
conjectures \cite{beilinson1, beilinson2, bloch1}, the
non-triviality of Gross--Schoen cycles will imply the vanishing of
the  L-series for the triple product motive of a curve. We have a
Northcott property for vanishing of L-series on moduli space of
curves. In the case of function field, these are unconditional.
Moreover, for non-isotrivial curve over function field of with good
reduction, the Arakelov-Szpiro theorem implies the vanishing of the
L-series of order $\ge 2$.
\item {\em Non-triviality of tautological classes in Jacobians.}
We will show that the heights of the canonical Gross--Schoen cycles
$\Delta _\xi$ have the Northcott type property on the moduli spaces
of curves.  We will give an expression of this height in terms of
the cycles $X_1$ and  $\CF (X_1)$ in  Beauville's Fourier--Mukai
transform \cite {beauville1, beauville2, beauville3} and
K\"unnemann's height pairing \cite{kunnemann2}. This implies in
particular the Northcott property holds for Ceresa \cite{ceresa}
cycles $X-[-1]^*X.$  For a non-isogeny curve over function field
with good reduction, these cycles are non-trivial by using a theorem
of Arakelov-Szpiro's theorem \cite{szpiro}.
\end{itemize}
In the following, we will describe in details the main results and
applications, and a plan of proof.

\subsection{Gross-Schoen cycles}
Let us first review Gross and Shoen's construction of the modified
diagonal cycles in \cite{gross-schoen} and definitions of heights of
Bloch \cite{bloch1}, Beilinson \cite{beilinson1, beilinson2}, and
Gillet--Soul\'e \cite{gillet-soule2}. Let $k$ be a field and let $X$
be a smooth, projective, and geometrically connected curve over $k$.
Let $Y=X^3$ be the triple product of $X$ over $k$ and let $e=\sum
a_i p_i$ be a divisor  of degree $\sum a_i\deg p_i=1$ such that some
positive multiple $ne$ is  defined over $k$. Define the diagonal and
the partially diagonal cycles with respect to base $e$ as follows:
$$\Delta _{123}=\{(x,x,x): x\in X\},$$
$$\Delta _{12}=\sum _i a_i\{(x,x, p_i): x\in X\},$$
$$\Delta _{23}=\sum a_i \{(p_i,x, x): x\in X\},$$
$$\Delta _{31}=\sum a_i\{(x,p_i, x): x\in X\}$$
$$\Delta _{1}=\sum _{i, j}a_ia_j\{(x,p_i, p_j): x\in X\},$$
$$\Delta _{2}=\sum _{i, j}a_ia_j\{(p_i,x, p_j): x\in X\},$$
$$\Delta _{3}=\sum _{i, j}a_ia_j\{(p_i,p_j, x): x\in X\}.$$
Then define the Gross-Schoen cycle  associated to $e$  to be
$$\Delta _e=\Delta _{123}-\Delta _{12}-\Delta _{23}-\Delta _{31}
+\Delta _1+\Delta _2+\Delta _3\in \Ch ^2(X^3)_\BQ.$$  Gross and
Schoen has shown that  $\Delta _e$  is homologous to $0$ in general,
and that $\Delta _e$  it is rationally equivalent to $0$ if $X$ is
rational, or elliptic, or hyperelliptic when $e$ is a Weierstrass
point. A natural question  is:
 {\em When is  $\Delta _e$  non-zero in $\Ch ^2(X^3)_\BQ$
in non-hyperelliptic case?}

Over a global field $k$, a natural invariant of $\Delta _e$ to
measure the  non-triviality of a homologically  trivial cycle is the
height of $\Delta _e$ which was conditionally constructed by
Beilinson--Bloch \cite{beilinson1, beilinson2, bloch1} and
unconditionally by Gross--Schoen \cite{gross-schoen} for $\Delta
_e$. More precisely, assume that $k$ is the fractional field of a
discrete valuation ring $R$ and that $X$ has a regular, semi-stable
model $\CX$ over $S:=\Spec R$. Then Gross-Schoen construct a regular
model $\CY$ over $S$ of $Y=X^3$  and show that the modified diagonal
cycle $\Delta _e$ on $Y$ can be extended to a codimension $2$ cycle
on $\CY$ which is numerically equivalent to $0$ in the special fiber
$\CY _s$.

If $k$ is a function field of a smooth and projective curve $B$ over
a field, then Gross and Schoen's construction gives a cycle $\wh
\Delta _e$ with rational coefficients  on a model $\CY $ of $Y=X^3$
over $B$. We can define the height of $\Delta _e$ as
$$\pair {\Delta _e, \Delta _e}=\wh \Delta _e\cdot \wh \Delta _e.$$
 The right hand here
is the intersection of cycles on $\CY$. This pairing does not depend
on the choice of $\CY$ and the extension $\wh \Delta _e$ of
$\Delta_e$.

If $k$ is a number field, then we use the same formula to define the
height for the arithmetical cycle $$\wh \Delta =(\wt \Delta _e, g)$$
Gillet--Soul\'e's arithmetic intersection theory
\cite{gillet-soule1} where \begin{itemize}
\item $\wt \Delta_e$ is the
Gross--Schoen extension of $\Delta _e$ over a model $\CY$ over
$\Spec \CO _k$; \item  $g$ is a Green's current on the complex
manifold $Y(\BC)$ of the complex variety $Y\otimes _\BQ\BC$ for the
cycle $\Delta _e$: $g$ is a current on $Y(\BC)$ of degree $(1, 1)$
with singularity supported on  $\Delta _e(\BC)$ such that the
curvature equation holds:
$$\pp g=\delta _{\Delta _e(\BC)}.$$
Here the right hand side denotes the Dirac distribution on the cycle
$\Delta _e(\BC)$ when integrate with forms of degree $(2,2)$ on
$Y(\BC)$. \end{itemize}
 Notice that this height can be
also defined using K\"unnemann's results in \cite{kunnemann2}. As
the non-triviality of $\Delta _e$ follows from the nonvanishing of
its height, a natural question is: {\em When is $\pair{\Delta _e,
\Delta _e}$ non-zero}?

\subsection{Admissible dualising sheaves}

Our main result of this paper is an expression of the height $\pair
{\Delta _e, \Delta _e}$ in terms of the self-intersection
$\omega_a^2 $ of the relative dualising sheaf  defined  in our early
paper \cite{admissible} which we recall as follows. Let $X$ be a
curve over a field $k$ of positive genus. We assume that $k$ is
either the fraction field of a smooth and projective curve $B$ or a
number field where we still set $B=\Spec \CO_k$, and that $X$ has a
semistable model $\CX$ over $B$.

When $B$ is a projective curve, then one has a usual intersection
pairing of divisors on $\CX$  and a usual relative dualising sheaf
$\omega _{\CX/B}$ which gives an adjunction formula for
self-intersections of sections.

In number field case, Arakelov theory gives intersections on the
arithmetic divisors of form $\wh D=(D, G)$ formed by a divisor $D$
on $\CX$ and an {\em admissible green's function} $G$ on $X(\BC)$ in
the sense that  its curvature satisfies , $$\delta _{D_\BC}-\pp
G=\deg D \cdot d\mu
$$ where $d\mu$ is the Arakelov measure on $X(\BC)$: on each
connected component $X_v(\BC)$ corresponding to archimedean place
$v$,
$$d\mu_v=\frac i{2g}\sum _{n=1}^g\omega _n\wedge \bar \omega _n$$
where $g$ is the genus of $X$ and $\omega _n$ are base of $\Gamma
(X_v, \Omega _{X_v})$ normalized such that
$$\frac i2\int \omega _m\wedge\bar \omega _n=\delta _{m,n}.$$
Arakelov shows that there is a  unique metric such that an
adjunction formula is true for a dualising sheaf with admissible
metric.
 By Faltings \cite{faltings}, we have a Hodge index theorem.

In \cite{admissible}, we construct an intersection theory (for
function field case or number field case) on divisors of the form
$(D, G)$ formed by a divisor $D$ of $X$ and $G$ an {\em adelic}
green's function with {\em adelic} curvature $d\mu _a$. More
precisely, $G$ has a component
 $G_v$ as a continuous function  on the reduction graph $R(X_v)$
 of $X\otimes k_v$
\cite{chinburg-rumely} for each closed point $v$ of $B$, and as a
usual  green's function on $X_v (\BC)$ for each archimedean place
$v$ in number field case.  We show that in this intersection theory,
we still have an adjunction formula and Hodge index theorem. We
called such intersection pairing an {\em adelic} admissible pairing.

We have proved in \cite{admissible} the following inequalities:
$$\wh\omega _{\CX/B}^2\ge \omega _a^2\ge 0.$$
Moreover the difference of the first two item is given by local
integrations: \begin{equation} \omega _{\CX/B}^2=\omega _a^2+\sum
_v\epsilon (X_v) \log \RN (v),\end{equation} where $v$ runs throught
the set of non-archimedean places, and $$\epsilon (X_v):= \int
_{R(X_v)}G_v(x, x)(\delta _{K_{X_v}}+(2g-2)d\mu _v)$$ where $G_v(x,
y)$ is the admissible Green's function on $R(X_v)$ for the
admissible measure $d\mu_v$ on $R(X_v)$, and $K_{X_v}$ is the
canonical divisor on $R(X_v)$. The first inequality is strict unless
$X$ has genus $1$ or $X$ has good reductions at all non-archimedean
place.

\subsection{Main result and first consequences}

The main result of this paper proved in \S3.5 is an identity between
the two canonical invariants:

\theoremstyle{plain}
\newtheorem*{main}{Main Theorem}

\begin{theorem} Let $X$ be a curve of genus $g>1$
 over a field $k$ which is either a number field or  the fraction
field of a curve $B$.  Then
$$\pair{\Delta_e, \Delta _e}
=\frac{2g+1}{2g-2}\omega _a^2+6(g-1)\pair{x_e, x_e}-\sum _v\varphi
(X_v)\log \RN (v).$$ Here $\pair{x_e, x_e}$ is the Neron-Tate height
of the class $e-K_X/(2g-2)$ in $\Pic ^0(X)_\BQ$, and $\varphi_v$ are
some contribution from  places $v$ of $K$:
\begin{enumerate}
\item If $v$ is an archimedean place, then
$$\varphi(X_v)=\sum _{\ell, m, n}\frac 2{\lambda _\ell}\left|\int _{X_v}
\phi _\ell\omega _m(x)\bar \omega _n(x)\right|^2$$ where $\phi
_\ell$ are normalized real eigenforms of the Arakelov Laplacian:
$$\pp \phi _\ell =\lambda _\ell \cdot \phi _\ell \cdot d\mu _v,
\qquad \int \phi _k\phi_\ell d\mu =\delta _{k, \ell}, $$ and $\omega
_i$ are basis of $\Gamma (X_v, \Omega _{X_v})$ normalized by
$$\frac i2\int \omega _m\bar \omega _n=\delta _{m, n}.$$
\item If $v$ is a nonarchimedean place, then
$$\varphi (X_v)=-\frac 14 \delta(X_v)+\frac 14\int _{R(X_v)}G_v(x,
x)((10g+2)d\mu_a-\delta _{K_{X_v}})$$ where $\delta(X_v)$ is the
number of singular points on the special fiber of $X_v$, $G_v(x, y)$
is the admission Green function for the admissible metric $d\mu_v$,
and the $K_{X_v}$ is the canonical divisor on $R(X_v)$. In
particular, $\varphi_v=0$ if $X$ has good reduction at $v$.
\end{enumerate}
\end{theorem}

Replace $k$ by an extension, we may fix a class  $\xi\in \Pic
(X)(k)$ such that $(2g-2)\xi=K_X$.  By the positivity of the
Neron-Tate height pairing, $\pair{\Delta _e, \Delta _e}$ reaches its
minimal value precisely when where
$$e =\xi+\text{torsion divisor}.$$
We call the cycle $\Delta _\xi$ the {\em canonical} Gross--Schoen
cycle for $X$.

\begin{corollary}
$$\omega _a^2=\frac {2g-2}{2g+1} \left(\pair {\Delta _\xi, \Delta _\xi}
+\sum _v\varphi (X_v)\log \RN (v)\right).$$
\end{corollary}

\begin{corollary} Assume that $X$ is a hyperelliptic curve, then
$$\omega _a^2=\frac {2g-2}{2g+1}\sum _v\varphi(X_v)\log \RN (v).$$
\end{corollary}

Combining with (1.2.1), this also gives an identity for the
self-intersection of the usual relative sheaf of hyperelliptic curve
in term of bad reductions. Some explicit examples of such formulae
have been given by Bost, Mestre, and Moret-Bailly in
\cite{bost-mestre-moret-bailly}. It is an interesting question to
compare our formula with theirs.

It is a hard problem to check when the height $\pair {\Delta _\xi,
\Delta _\xi}=0$ even in the function field case. We have the
following consequence of Theorem 1.3.1 in smooth case:
\begin{corollary}
Assume that $k$ is the function field and that $X$ can be extended
to a non-isotrivial family $\CX\lra B$ of smooth and projective
curves of genus $g>1$ over a projective and smooth curve $B$. Then
$$\pair{\Delta _\xi, \Delta _\xi}
=\frac {2g+1}{2g-2}\omega _{\CX/B}^2>0.$$
\end{corollary}

\begin{proof}
The first equality follows from Theorem 1.3.1 and the discussion in
\S1.2. The second inequality is due to the ampleness of $\omega
_{\CX/B}$ by proved by Arakelov in case of characteristic $0$ and by
Szpiro in case of positive characteristic.
\end{proof}

 We will show that $\pair {\Delta_\xi,
 \Delta _\xi}$ is essentially a height function in \S4.2:

 \begin{theorem}
 Let $Y\lra T$ be a flat family of smooth and projective curves of genus $g\ge 3$
 over a projective variety $T$ over a number field $k$, or the function field
 of a curve over a finite field. Then the function
 $$t\in T(\bar k)\mapsto (2g-2)\pair{\Delta _\xi(Y_t), \Delta _\xi (Y_t)}$$
 is a height function associate to  Deligne's pairing
 $$(2g+1)\pair {\omega _{Y/T}, \omega _{Y/T}}.$$
 Moreover if the induced map $T\lra \CM_g$ from $T$ to the coarse moduli
 space of curves of genus $g$ is finite, then we have a Northcott
 property: for any positive numbers $D$ and $H$,
 $$\#\left\{t\in T(\bar k): \quad \deg t\le D,\quad
 \pair{\Delta (Y_t)_\xi,
 \Delta(Y_t)_\xi}\le H\right\}<\infty.$$
 \end{theorem}

\subsubsection*{Remarks}

We would like to give some remarks about the upper bound for
$\pair{\Delta _\xi, \Delta _\xi}$.

When $k$ is a function field of curve $B$ of genus  $\ge 2$ a field
of characteristic $0$,  the semi-stable model $\CX$ is a surface of
general type and one has  the Bogomolov--Miyaoka--Yau inequality:
$$c_1(\CX)^2\le 3c_2(\CX).$$
Equivalently,  in term of relative data,
$$\omega _{\CX/B}^2\le (2g-2)(2q-2)+3\sum _{b\in B}\delta (X_b).$$
See Moret-Bailly \cite{moret-bailly} for details. By Corollary
1.3.1, we have a  bound  for the  height of Gross--Schoen cycle:
$$\pair{\Delta _\xi, \Delta_\xi}\le (2g+1)(2q-2)+\sum _{b\in B}
\left(\frac {6g+3}{2g-2}(\delta (X_b)+\epsilon (X_b))-\varphi
(X_b)\right).$$

When $k$ is a function field of positive characteristic, then the
Bogomolov--Miyaoka--Yau inequality is not true. Instead, one has a
Szpiro (Theorem 3, \cite{szpiro}) inequality in which one needs to
add some inseparableness of $f$. So we have a similar inequality for
the height of Gross--Schoen cycle.

When  $k$ is a number field, then Parshin \cite{parshin} and
Moret-Bailly \cite{moret-bailly}  have formulated an arithmetic
Bogomolov--Miyaoka--Yau inequality. It has been proved that this
conjecture is equivalent to  the effective Mordell conjecture,
Szpiro's discriminant conjecture, and the ABC conjecture.
Conversely, Elkies \cite{elkies} has proved that $ABC$-conjecture
will imply the effective conjecture. By our main theorem, these are
equivalent to an upper bound  conjecture with $\omega^2$ replaced by
the height of Gross--Schoen cycle.

\subsection{Gillet--Soul\'e and Bogomolov  conjectures}
By the construction, the cycle $\Delta _e$ has zero intersection in
$\Ch^3(X^3)$ with  $p_i^*\Pic (X)$ via the projections $p_i: X^3\lra
X$. Thus, it is primitive with respect an ample line bundle $\CL$ on
$X^3$ of the form $\sum p_i^*\CL$ for an ample line bundle on $X$.
In case where $k$ is a function field of characteristic $0$, by
Hodge index theorem, this height is non-negative, and is vanishing
precisely when $\wh \Delta _e$ is numerically equivalent to $0$.

In function field case of positive characteristic, the Hodge index
theorem is part of the Standard Conjecture of Grothendieck and
Gillet--Soul\'e \cite{gillet-soule2}:
\begin{conjecture}[Grothendieck, Gillet-Soul\'e]
Let $k$ be a number field or a function field with positive
characteristic, then
$$\pair {\Delta _\xi, \Delta _\xi}\ge 0$$
and
  this height  vanishes precisely when $\Delta_\xi$ is rationally
equivalent to $0$.
\end{conjecture}

Granting this conjecture which is true in case of the function field
of characteristic $0$  and hyperelliptic case, we then have a lower
bound for $\omega _a^2$:
$$\omega _a^2\ge \frac {2g-2}{2g+1}\sum _v\varphi (X_v)\log \RN (v).$$
It is  proved in \cite{admissible} that $\omega_a^2>0$ is equivalent
to the Bogomolov conjecture about the finiteness of points $x\in
X(\bar k)$ with small Neron-Tate height in the map
$$ X\lra \Jac (X), \qquad x\mapsto [(2g-2)x-K_X]\in \Jac (X).$$
In  number field case, the Bogomolov conjecture  is proved by Ullmo
\cite{ullmo, equidistribution}. The conjecture of Gillet-Soul\'e
thus implies an effective version of Bogomolov conjecture as
$\varphi _v$ can be computed effectively for any given graph. In
view of  the Bogomolov conjecture, we would to make the following:

\begin{conjecture}Let $v$ be a finite place.
 Let $\delta_0(X_v), \cdots, \delta
_{[g/2]}(X_v)$ denote the numbers of singular points $x$ in the
special fiber $X_{k(v)}$ such the local normalization of $X_{k(v)}$
at $x$ is connected when $i=0$ or a disjoint union of two curves of
genus $i$ and $g-i$. Then
$$\varphi (X_v)\ge c(g)\delta_0(X_v)+\sum _{i>0}\frac
{2i(g-i)}g\delta_i(X_v)$$ where $c(g)$ is positive continuous
function of $g>1$.
\end{conjecture}

From Theorem 1.3.1, it is clear that the conjecture is true for
Archimedean places and finite places with good reductions. In \S4.3,
we will show that it suffices to show the conjecture when all
$\delta_i=0$ for $i>0$.  More precisely,  we will give an explicit
formula in \S4.4 for $\varphi_v$ for elementary graphs and prove the
following:
\begin{theorem}
Assume that the reduction graph $R(X_v)$ is elementary in the sense
that every edge is included in at most one  cycle. Then the
conjecture is true with $$c(g)=\frac {g-1}{6g}.$$ The equality is
true if and only if every circle has at most one vertex.
\end{theorem}

Recently, Xander Faber \cite{faber} has verified the conjecture for
lower genus curves. For example, he shows for genus $2$ and $3$, we
may take $c(2)=1/27$ and that $c(3)=2/81$. Thus he has a proof of
the Bogomolov for all curves of genus $3$.

The Bogomolov conjecture should hold for non-isotrivial curve over
function field. Some partial results have been obtained by Moriwaki
\cite{moriwaki}, Yamaki \cite {yamaki}, and Gubler \cite{gubler}.
The work of Moriwaki and Yamaki are effective and follows from a
slope inequality of Moriwaki for general semistable fiberation $\pi:
\CX\lra B$:
$$\lambda (\CX):=\deg \pi_*\omega _{\CX/B}
\ge \frac g{8g+4}\delta _0(X)+\sum _{i>0}\frac {i(g-i)}{2g+1}\delta
_i(X),$$ where $\delta _i(X)=\sum _v\delta _i(X_v)\log \RN (v)$ is
the intersection of $B$ with $i$-the boundary component of the
moduli space. This formula is a generalization of a work of Xiao
\cite{xiao} and Cornalba--Harris \cite{cornalba-harris}, and is
proved based on the stability of the sheaf $\pi _*\omega _{\CX/B}$
and by Noether's formula
\begin{equation}
\lambda (\CX/B)=\frac 1{12}(\omega _{\CX/B}^2+\sum
_v\delta(X_v))\end{equation} where $\delta(X_v)=\sum \delta _i(X_v)$
be the total number of singular points in the fiber over $v$. Thus,
we have an equality \begin{equation}
 \lambda
(\CX/B)=\frac {2g-2}{2g+1}\pair{\Delta _\xi, \Delta _\xi} +\sum
\lambda (X_v)\log \RN (v)\end{equation} where \begin{equation}
\lambda (X_v)=\frac {g-1}{6(2g+1)}\varphi_v +\frac 1{12}(\epsilon
(X_v)+\delta(X_v)).\end{equation} Thus the Hodge index theorem gives
\begin{theorem} If $k$ is a function field of characteristic $0$, then
$$\lambda (\CX/B)\ge \sum \lambda (X_v).$$
\end{theorem}

We believe that this is the sharpest slope inequality  for fibred
surfaces with given configuration of singular fibers. In particular,
the Moriwaki's inequality should follows from the following
\begin{conjecture} If $v$ is a non-archimedean place, then
$$\lambda (X_v)\ge \frac g{8g+4}\delta _0(X_v)+\sum _{i>0}\frac {i(g-i)}{2g+1}
\delta_i(X_v).$$
\end{conjecture}
In \S4.3, we will reduce this conjecture to the case where $\delta
_i=0$ and prove the conjecture for elementary graphs. Also Xander
Faber \cite{faber} has verified the conjecture for curves with small
genera.

In number field case, Faltings \cite{faltings} defines  a volume
form on $\lambda $ for each archimedean place $v$. The number
$\lambda (\CX/B)$ is called the Faltings height of $\CX$. He also
proves a Noether formula (1.4.1) with his $\delta _v$. Thus we still
have expression (1.4.2) with $\lambda_v$ given in (1.4.3 ) when $v$
is non-archimedean, and
$$\lambda (X_v)=\frac {g-1}{6(2g+1)}\varphi(X_v)
+\frac 1{12}\delta(X_v)$$ when $v$ is archimedean, where $\varphi_v$
is given in Theorem 1.3.1. Now Theorem 1.4.4 is a conjecture
predicted by Gillet-Soul\'e's Conjecture 1.4.1:
\begin{conjecture} If $k$ is a number  field, then
$$\lambda (\CX/B)\ge \sum \lambda (X_v)\log \RN (v).$$
\end{conjecture}

\subsection{Beilinson--Bloch conjecture and tautological classes}
Assume that $k$ is a number field or a function field of a curve
defined over a finite field. For a smooth and projective variety $Y$
defined over $k$, and an integer between $0$ and $\dim Y$, we should
have a motive $H^{2n-1}(Y)(n)$ and  a complete  $L$-series
$L(H^{2n-1}(Y)(n), s)$ with a conjectured holomorphic continuation
and a  function equation. We also have a Chow group $\Ch ^n(Y)^0$ of
 codimension $n$-cycles on $Y$ with trivial classes in
$H^{2n}(Y)(n)$. The conjecture of Beilinson \cite{beilinson1,
beilinson2} and Bloch \cite{bloch1} asserts that $\Ch^n(Y)^0$ is of
finite rank and \begin{equation} \rank \Ch^n(Y)^0=\ord
_{s=0}L(H^{2n-1}(Y)(n), s).\end{equation} If $Y$ is a curve, then
the above is the usual Birch and Swinnerton-Dyer conjecture for
$\Jac (X)$. If $k$ is a function field, then the holomorphic
continuation of the L-series and the functional equation are known.
The Beilinson--Bloch conjecture in function field case is equivalent
to Tate's conjecture.

Now we assume that $Y=X^3$ is a power of a curve over $k$, $n=2$.
Then both sides of (1.5.1) has decomposition by correspondences
defined by action of symmetric group $\CS_3$ acting on $X^3$,
projections and emdeddings between $X^i$ and $X^j$. In \S5.1, we
will show that $\Delta _\xi$ lies in the subgroup $\Ch (M)$  of $\Ch
^2(X^3)^0$ of elements $z$ satisfying the following conditions:
\begin{enumerate}
\item $z$ is symmetric with respect to permutations on $X^3$;
\item the pushforward $p_{12*}z=0$ with respect to the projection
$$p_{12}:\quad X^3\lra X^2, \qquad (x, y, z)\mapsto (x, y).$$
\item let $i: X^2\lra X^3$ be the embedding defined by $(x, y)\lra
(x, x, y)$ and $p_2: X^2\lra X$ be the second projection. Then
$$p_{2*}i^*z=0.$$
\end{enumerate}
The operations induces some correspondences on $X^3$. The
corresponding  Chow  motive $M$ can be defined to be the kernel of
$$\bigwedge ^3H^1(X)\,(2)\lra H^1(X)(1),\quad
a\wedge b\wedge c\mapsto a(b\cup c) +b(c\cup a)+ c(a\cup b).$$ The
motive $M$ is pure of weight $-1$ with an alternative pairing
$$M\otimes M\lra \BQ(1).$$
It is conjectured that the complete L-series of $M$ has a
holomorphic continuation to whole complex plane and satisfies a
functional equation
$$L(M, s)=\pm c(M)^{-s}L(M, -s)$$
where $\epsilon (M)=\pm 1$ is the root number of $M$, and $c(M)\in
\BN$ is the conductor of $M$ which is divisible only by places
ramified in $M$. See Deligne \cite{deligne2} and Tate \cite{tate}
for details. In our situation, the Beilinson and Bloch conjecture
has a refinement:
\begin{conjecture}[Beilinson--Bloch]
$$\rank \Ch (M)=\ord _{s=0}L(M, s).$$
\end{conjecture}

If $k$ is a number field, we don't know in general that $L(M, s)$
has a homomorphic continuation. But we attempt to guess that for
most curve $X$ over a field $k$, the L-series should has vanishing
order $\le 2$. In other words, for general $X$,
$$\epsilon
(M)=1\Longrightarrow L(M, 0)\ne 0,$$
$$\epsilon (M)=-1
\Longrightarrow L'(M, 0)\ne 0.$$ The following are some formulae for
computing epsilon factors proved in \S5.2:

\begin{theorem} The epsilon factor has a decomposition
$$\epsilon (M)=\prod _v\epsilon _v(M)$$
into a product of local epsilon factor give as follows.
\begin{enumerate}
\item If $v$ is a real  place,
$$\epsilon _v(M)=(-1)^{g(g-1)/2}=\begin{cases}1,&\text{if $g\equiv 0, 1\mod 4$}\\
-1,&\text{if $g=2,3 \mod 4$}\end{cases}
$$
\item If $v$ is a complex place
$$\epsilon_v(M)=(-1)^{g(g+1)(g+2)/6}=\begin{cases}1&\text {if $g\not\equiv 1\mod 4$}\\
-1&\text{if $g\equiv 1\mod 4$}
\end{cases}
$$
\item If $v$ is a non-archimedean place, then
$$\epsilon _v(M)= (-1)^{e(e-1)(e-2)/6+ge}\cdot \tau
^{(e-1)(e-2)/2+g}$$ where $e$ is the dimension of toric part $T_v$
of the reduction of N\'eron model of $\Jac (X)$ at $v$, and
$\tau=\pm 1$ is the determinant of the Frobenius $\Frob _v$ acting
on the character group $X^*(T_v)$.
\end{enumerate}
\end{theorem}

If $k$ is a function field,  we have an inequality
\begin{equation}\rank \Ch ^0(M)\le \ord _{s=0}L(M, s).\end{equation} If $X$ is
non-isotrivial and has good reduction everywhere over places of $k$,
then by Theorem 1.3.4, $\Delta _\xi$ is non-zero. On the other hand,
we can show that the sign of the functional equation is $1$. Thus we
must have
\begin{theorem}If $X/k$ is a curve of over function field of a curve $B$
over a finite field of genus $g\ge 3$. Assume that $X$ can be
extended into a non-isotrivial  smooth family of curves over $B$.
Then
 $$\ord _{s=0}L(M, s)\ge 2.$$
 \end{theorem}
 In view of Tate's conjecture, we have
 $$\rank \Ch (M)=\ord _{s=0}L(M,s)\ge 2.$$
 Thus we have a natural question: {\em how to find another cycle in
 $\Ch ^2(M)^0$ which is linear independent of $\Delta _\xi$?}

 In general it is very difficult to compute
  the special values or derivatives of $L(M,
 s)$ at $s=0$. However the following is a consequence of
Theorem 1.3.5 and  Beilinson-Bloch's conjecture, we conclude the
 following:
 \begin{conjecture}
 Let $Y\lra T$ be a flat family of smooth and projective curves of genus $g\ge 3$
 over a projective variety $T$ over a number field $k$. Assume
 the induced map $T\lra \CM_g$ from $T$ to the coarse moduli
 space of curves of genus $g$ is finite, then we have a Northcott
 property: for any positive numbers $D$,
 $$\#\left\{t\in T(\bar k): \quad \deg t\le D,\quad
 L(M(Y_t), 0)\ne 0\right\}<\infty.$$
 \end{conjecture}
 Over function field, this is a theorem induced from Theorem 1.3.5
 and formula (1.5.2).

In the following, we want to apply our result to the tautological
algebraic cycles in the Jacobian defined by Ceresa \cite {ceresa}
and Beauville \cite{beauville3}. We will use Fourier--Mukai
transform of Beauville (\cite{beauville1, beauville2}) and height
pairing of K\"unnemann (\cite{kunnemann2}).

Let $X\lra J$ be an embedding given by taking $x$ to the class of
$x-\xi$. Then we define the tautological classes  $\CR$  to the
smallest subspace of $\Ch ^*(J)$ containing $X$ closed under the
following operations:
\begin{itemize}
\item intersection pairing ``$\cdot$";
\item Pontriajan's star operator ``$*$";
$$x*y:=m_*(p_1^*x\cdot p_2^*y)$$
where $p_1, p_2, m$ are projection and addition on $J^2$;
\item Fourier--Mukai transform
$$\CF: \Ch ^*(J)_\BQ\lra \Ch ^*(J)_\BQ$$
$$x\mapsto \CF (x):=p_{2*}(p_1^*x\cdot e^\lambda)$$
where $\lambda$ is the Poincar\'e class:
$$\lambda =p_1^*\theta+p_2^*\theta -m^*\theta.$$
\end{itemize}
Using Fourier--Mukai transform, we have spectrum decomposition
$$X=\sum_{s=0}^{g-1} X_s, \qquad X_s\in \Ch ^{g-1}(J)$$
with $[k]_*X_s=k^{2+s}X_s.$ By Beauville \cite{beauville3}, the ring
$\CR$ under the intersection pairing is generated by $\CF (X_s)\in
\Ch^{1+s}(J)$. The pull-back of these cycles on $X^3$ under the
morphism $f_3: X^3\lra J$ can be computed explicitly. In particular,
we can prove the following formulae proposed by Wei Zhang
\cite{weizhang}:
\begin{theorem} Consider the addition morphism $f_{3}: X^3\lra J$.
Then
$$ f_3^*\CF (X_1)=\Delta, \qquad f_{3*}\Delta _\xi=\sum
_s(3^{2+s}-3\cdot 2^{2+s}+3)X_s,$$
$$X_s=(3^{2+s}-3\cdot 2^{2+s}+3)^{-1}\sum _{i+j+k=s-1}(X_i*X_j*X_k)\cdot \CF (X_1), \qquad s>0.$$
Moreover, the following are equivalent:
\begin{enumerate}
\item $\Delta _\xi=0$ in $\Ch ^2(X^3)_\BQ$;
\item $X-[-1]^*X=0$ in $\Ch ^{g-1}(J)_\BQ$;
\item $X_1=0$ in $\Ch ^{g-1}(J)_\BQ$;
\item $X_s=0$ in $\Ch ^{g-1}(J)_\BQ$ for all $s>0$.
\end{enumerate}
\end{theorem}

By this theorem,  under the operators $\cdot$ and $*$,  the ring $R$
is generated by $X_0$ and any one of three canonical classes $X_1$,
Gross-Schoen cycle $f_{3*}\Delta _\xi$, and Ceresa cycle
$X-[-1]^*X$.  The following give a more precise relation between the
height of $\Delta _\xi$ and the height of class $X_1$ and $\CF
(X_1)$:
\begin{theorem}
The cycle  $\CF (X_1)$ is primitive with respect to theta divisor
$\theta$, homologically trivial  in $\Ch ^2(J)_\BQ$, and
$$\pair{\CF (X_1), X_1}_J=\frac 16\pair{\Delta _\xi, \Delta _\xi}_{X^3}
=\frac 1{(g-3)!} \pair{\CF (X_1), \theta^{g-3}\CF (X_1)}_J.$$
\end{theorem}

\subsection*{Plan of proof}
The proof of Theorem 1.3.1 is proceeded in several steps in \S2-3:
\begin{enumerate}
\item Reduction from $X^3$ to $X^2$: we  express the height as a triple  product on $X\times X$ of an adelic  line bundles
with generic fiber (Theorem 2.3.5):
$$\Delta -p_1^*\xi-p_2^*\xi.$$
\item Reduction form $X^2$ to $X$: we  express the triple as the self-intersection of the
canonical sheaf plus some local triple integrations (Theorem 2.3.5).
\item Local triple pairing: we develop an intersection theory on the reduction complex of the
product $X\times X$ at a non-archimedean place (Theorem 3.4.2) and
use this to complete the proof of Theorem 1.3.1.
\end{enumerate}

The proof of other results about the estimate of the height follows
form detailed calculation of constants $\phi$ and $\lambda$ in \S4.
We first express these constants in terms of integration of
resistance on metrized graph and reduce the computation to 2-edge
connected graphs, and finitely compute everything for $1$-edge
graphs.

The last section is devoted to study the Beilinson--Bloch conjecture
and the Beauville tautological cycles. We first define a minimal
Chow motive $M$ so that its Chow group contains $\Delta _\xi$. Then
we compute the $\epsilon$-constant of its L-series. Finally, we
translate the statements to tautological cycles in the Jacobian
varieties.

\subsection*{Acknowledgement} I would like to thank Benedict Gross
for explaining to me his joint work with  Steve Kudla in
\cite{gross-kudla} and Chad Schoen \cite{gross-schoen} which
inspired this work, to Johan de Jong for explaining to me  many
facts in algebraic geometry, and to Xander Faber, Xinyi Yuan, and
Wei Zhang for many helpful conversations during  a workshop on this
work. This work has been supported by the National Science
Foundation  and the Chinese Academy of Sciences.

\section{Gross-Schoen cycles and correspondences}
The aim of this section is to prove some global formulae for the
heights of Gross--Schoen cycles in terms of the self-intersections
of the relative dualising sheaves and some local intersections:
$$\frac {2g+1}{2g-2}\omega ^2+\text {local contributions, \qquad (Theorem 2.5.1)}.$$
These local contributions will be computed in the next section. More
generally, for any correspondences $t_1, t_2, t_3$ on $X\times X$,
we compute the height pairing
$$\pair{\Delta_e,  (t_1\otimes t_2\otimes t_3)\Delta _e}.$$
This pairing is positive if $t_i$ are correspondence of positive
type by Gillet--Soul\'e's Conjectures 2.4.1 and 2.4.2. We will show
that this is equal to the intersection number $\wh t_1\cdot \wh
t_2\cdot \wh t_3$ on  $X\times X$ (Theorem 2.3.5).

\subsection{Cycles and Heights}

In this subsection, we will review intersection theory of
Gillet--Soul\'e and some adelic extensions. The basic reference are
Gillet--Soul\'e \cite{gillet-soule1, gillet-soule2}, and Deligne
\cite{deligne}, and our previous paper \cite{adelic}.

\subsubsection*{Arithmetical intersection theory}
  Let $k$ be a number field with the ring of
integers $\CO_k$. By an arithmetical variety over $\CO_k$, we mean a
flat and projective morphism $\CX\lra \Spec \CO_k$ such that $\CX_k$
is regular.  We have  a homological arithmetical Chow group
$\wh\Ch_*(\CX)$,
 formed by cycles $(Z, g)$ where $Z$ is a cycle on $\CX$ and $g$ is
a current such that $\pp g+\delta _Z$ is smooth on $\CX(\BC)$,
modulo the relations: \begin{itemize} \item  $(\div (f), -\log
|f|)=0$ for a rational function $f$ on an integral  subscheme $\CY$
of $\CX$; \item $(0, \partial \alpha +\bar\partial \beta)=0$.
\end{itemize}
We also have a cohomological arithmetical Chow group $\wh\Ch^*(\CX)$
formed by Chern classes of Hermitian vector bundles. Then the tensor
product among vector bundles gives  $\wh \Ch^*(\CX)$  a ring
structure. The intersection pairing between $\wh \Ch ^*(\CX)$ and
$\wh \Ch_*(\CX)$ gives $\wh \Ch_*(\CX)$ a module structure over $\wh
\Ch ^*(\CX)$. When  $\CX$ is regular, these two groups are
isomorphic.

By a (relative) arithmetical correspondence, we mean a pair $(C,
\gamma)$ on $\CX\otimes_{\CO_k}\CX$ formed by a (homological) cycle
$C$ of dimension equal to $\dim X$ and a current $\gamma$ such that
$h_C:=\pp\gamma +\delta _C$ is regular in the sense that for any
smooth form $\alpha$ on $X(\BC)$, the currents
$$p_{2*}(h_C\cdot p_1^*\alpha), \qquad p_{1*}(h_C\cdot p_2^*\alpha)$$
are both smooth. Let $C(\CX)$ denote the group of arithmetical
correspondences modulo the same relations as above on $\CX\otimes
\CX$. An arithmetical correspondence $c:=(C, \gamma)$ defines maps
$c_*$ and $c^*$ as usual from the group $\wh \Ch ^*(\CX)$ to $\wh
\Ch _*(\CX)$. If $\CX$ is smooth over $\CO_k$, then $\CX^2$ is
smooth and we can define composition on correspondences to make
$C(\CX)$ a ring. In this case,  the morphism $c\lra c_*$ is a ring
homomorphism form $C(X)$ to $\End (\wh\Ch^*(\CX))$.

\subsubsection*{Cycles homologous to zero}

Let $X$ be a smooth and projective variety of dimension $n$ over a
number field or a function field $k$. Let $\Ch (X)$ denote the group
of Chow cycles with coefficient in $\BQ$. Then we have a class map
to $\ell$-adic cohomology:
$$\Ch (X)\lra H^*(X)$$
where $H^*(X)=H^*(X\otimes \bar k, \BQ_\ell )$ with $\ell$ a prime
different than the characteristic of $k$. The kernel $\Ch (X)^0$ of
this map is called the group of  homologically trivial cycles.
Beilinson (\cite{beilinson1, beilinson2}) and Bloch (\cite{bloch1})
have given a conditional definition of height pairing between cycles
in $\Ch(X)^0$. We will focus on the case of number fields but all
the results hold for case where $k$ is the function field of a
smooth and projective curve $B$ over some field $k_0$, and where we
have the same height pairing with $\Spec \CO _k$ replaced by $B$ and
with the condition about green's function dropped.

\subsubsection*{Height pairing}
 One construction of this height pairing in number
field case is based on Gillet and Soul\'e's intersection theory as
follows. Assume that $X$ has a regular model $\CX$ over $\Spec
\CO_k$,  and that every cycle $z\in \Ch (X)^0$ has an extension $\wh
z=(\bar z, g_z)$ to an arithmetic cycle which has trivial
intersection to vertical arithmetic cycles of dimension $2$:
\begin{enumerate}
\item $\bar z$ is a cycle on $\CX$ extending  $z$;
\item $g_z$ has curvature $h_z=0$;
\item the restriction of $\bar z$ on each component in the special
fibers of $\CX$ is numerically trivial.
\end{enumerate}
Then for any $z'\in \Ch (X)^0$ extended to an arithmetic cycle $\wh
z'$ on $\CX$,  the height pairing is defined by
$$\pair{z, z'}:=\wh z\cdot \wh z'.$$
It is clear that this definition does not depend on the choice of
$\wh z'$, and that the pairing is linear and symmetric.

Let $C(X)=\Ch ^n (X\times X)$ denote the ring of (degree $0$)
correspondences on $X$. Then $C(X)$ acts on $\Ch (X)$ and preserve
$\Ch (X)^0$. Recall that the composition law is given by the
intersection pairing on $X\times X\times X$ and various projections
to $X\times X$:
$$t_2\circ t_1=p_{13*}(p_{12}^*t_1 \cdot p_{23}^*t_2),
\qquad t_1, t_2\in C(X).$$ For any $t\in C(X)$, $z\in \Ch (X)$, the
push-forward and pull-back of $z$ under $t$ are defined by
$$t_*(z)=p_{2*}(p_1^*z\cdot t), \qquad t^*(z)=p_{1*}(t\cdot p_2^*z).$$
Let $t\lra t^\vee$ be the involution defined by the permutation on
$X^2$ then we have $t^*=(t^\vee )_*$.
 It can be shown that the involution operator is the adjoint operator
 for the height pairing:
\begin{lemma}
$$\pair{t_*z, z'}=\pair{z, t^*z'}=\pair {z, t^\vee _*z'}.$$
\end{lemma}

\begin{proof}
 For each $t\in C(X)$, let $\wh t=(\bar t,
g_t)$ be an arithmetic model of $t$ over $\wt{\CX^2}$. Then we can
define a correspondence on arithmetical cycles by the same formula:
$$\wh z\mapsto \wh t^* (\wh z):=p_{1*}(\wh t\cdot p_2^*\wh z).$$
We claim that $\wh t(\wh z)$ is numerically trivial fiber wise. The
curvature is given by $p_{1*}(h_t\cdot p_2^*h_z)=0$. To check the
numerical triviality over non-archimedean places, we let $x$ be a
cycle in the fiber $\CX_v$ over a finite place $v$ of $k$. Then
$$\wh t^*(\wh z)\cdot x=p_2^*\wh z\cdot \wh t\cdot p_1^*x
=\wh z\cdot p_{2*}(\wh t\cdot p_1^*x)=0$$ as $p_{2*}(\wh t\cdot
p_1^*x)$ is still a vertical cycle over $v$. Thus we have
$$\pair {t^*z, z'}=\wh t^*(\wh z)\cdot \wh z'
=p_2^*\wh z \cdot \wh t \cdot p_1^* \wh z' =\wh z\cdot p_{2*} (\wh
t\cdot p_1^*\wh z').$$ It is clear that  $p_{2*} (\wh t\cdot
p_1^*\wh z')$ is an extension of
 $$p_{2*}(t\cdot p_1^* z')=t_*z'.$$ Thus we have
shown the adjoint property of $t$.
\end{proof}

\subsubsection*{Adelic metrized line bundles}
In the following, we want to review some facts about the adelic
metrized bundles developed in \cite{adelic}. For a smooth variety
$X$ defined over a number field $k$, let us consider the category of
arithmetic models  $\CX$ with generic fiber $X$, i.e. an arithmetic
variety $\CX\lra \Spec \CO_k$ and an isomorphism $\CX_k\simeq X$. As
this category is partially ordered by morphisms, we can define the
direct limit
$$\lim _{\lra}\wh\Pic (\CX)\otimes \BQ.$$
Every element in this group defines an algebraically metrized line
bundle on $X$. The group $\bar \Pic (X)$ of integral metrized line
bundles are certain limits of these algebraically metrized line
bundles. The intersection pairing
$$c_1(\CL_1)\cdots c_1(\CL_n)\cdot \alpha\cdot [\CX]\in \BR,
\qquad \CL _i\in \Pic (\CX), \alpha \in \Ch ^{\dim \CX-n}(\CX)$$ can
be extend to a pairing with $\CL_i\in \bar\Pic (X)$. The following
lemma shows that the pairing can be represented by a (homological)
element on $\CX$.

\begin{lemma}
Let $\CX$ be an arithmetical scheme and  let $\bar\CL_1, \cdots,
\bar\CL_n$ be some adelic metrized line bundles on $X_k$. Then the
functional
$$\Ch ^n(\CX)\lra \BR: \quad \alpha \mapsto \alpha \cdot c_1(\bar
\CL _1)\cdots c_1(\bar \CL _n)$$ is represented by an element in
$\wh \Ch _{\dim X-n}(\CX)\otimes\BR$ denoted by
$$c_1(\bar \CL_1)\cdots c_1(\bar\CL _n)\cdot [\CX]\in \wh\Ch _{\dim
X-n}(\CX)\otimes \BR.$$ Moreover, this element has the following
restriction on the generic fiber:
$$c_1(\CL_1)\cdots c_1(\CL_n)[\CX_k].$$
\end{lemma}

\begin{proof}
It suffices to deal with the case where  bundles are ample and are
limits of some  integral-ample models $(\CX_i, \CM_{i1}, \cdots,
\CM_{in})$ of $(X_k,\CL_k ^{e_i}, \cdots \CL _k^{e_i})$. Without
loss of generality, we may assume that $\CX_i$ dominates
 $\CX$,  that $\CL_{1k}, \cdots \CL _{nk}$
 have arithmetic modes $\CM _{01}, \cdots, \CM _{0n}$, and that the
 metrics on the archimedean places induce the same metrics on each $\CL
 _{ik}(\BC)$. Let $\pi _i$ denote the projection $\CX _i\lra \CX$.

For any cohomological arithmetical cycles $\alpha$  on $\CX$,
 we can define intersection
pairings:$$ \frac 1{e_i^n} c_1(\CM_{i1})\cdots c_1(\CM_{in}))\cdot
\pi _i^*\alpha=\pi _{i*}\left(\frac 1{e_i^n} c_1(\CM_{i1})\cdots
c_1(\CM_{in}))\right)\cdot \alpha$$ which has a limit denoted by
$$c_1(\bar\CL_1)\cdots c_1(\bar \CL _n)\cdot \alpha.$$
We claim that the cycles $$\pi _{i*}\left(\frac 1{e_i^n}
c_1(\CM_{i1})\cdots c_1(\CM_{in}\cdot [\CX_i])\right)$$ have a limit
in $\wh\Ch ^*(\CX)\otimes \BR$. Indeed, subtract them   by
$$c_1(\CM_{01})\cdots c_1(\CM_{0n})[\CX],$$
we obtain vertical cycles
$$V_i:=\pi _{i*}\left(\frac
1{e_i^n} c_1(\CM_{i1})\cdots c_1(\CM_{in})\cdot
[\CX_i]\right)-c_1(\CM_{01})\cdots c_1(\CM_{0n}\cdot [\CX]$$
supported in finitely many fibers of $\CX$ over $\CO_k$. Let $\CF$
be the union of these fibers as a closed subscheme of $\CX$, let $i:
\CF\lra \CX$ denote the embedding. Then $V_i$ can be written as
$V_i=i_*W_i$ with $W_i$ a cycle on $\CF $. In this way,
$$V_i\cdot \beta =W_i\cdot i^*\beta.$$
In other words, the intersection pairing of $V_i$ with $\wh \Ch^*
(\CX)$ can be written as intersection of $W_i$ with $i^*\wh \Ch
^*(\CX)$. So we may work on the intersections between the  quotients
$N_*(\CF)$ and  $N^*(\CF)$ of homological and cohomological cycles
of $\CF $ modulo numerical equivalence. Let $M$ be the image of
$i^*\wh \Ch^* (\CX)$ in $N^*(\CF)$. The elements $W_i$ thus defines
a sequence of convergent functionals on $M$. As $N^*(\CF)$ is finite
dimensional, this sequence will convergent to a functional
represented by an element $W$ of $\Ch_*(\CF )\otimes \BR$. Thus we
have shown that
$$\lim _i\pi _{i*}\left(\frac
1{e_i^n} c_1(\CM_{i1})\cdots c_1(\CM_{in})\cdot
[\CX_i]\right)=c_1(\CM_{01})\cdots c_1(\CM_{0n})\cdot [\CX]+V.$$ In
this way we define a correspondence
$$c_1(\bar\CL_1)\cdots c_1(\bar \CL _n)\cdot [\CX]=\lim _i\pi _{i*}
\left(\frac 1{e_i^n} c_1(\CM_{i1})\cdots c_1(\CM_{in})\cdot
[\CX_n]\right).$$ \end{proof}

\subsubsection*{Deligne pairing} In the following, we want to
construct Deligne pairing of  metrized line bundles. Let $f: X\lra
Y$ be a flat and projective morphism  of two smooth varieties over
valuation field $k$ of relative dimension $n$. Let $\bar \CL_0=(\CL
_0, \|\cdot\|_0)$, $\bar\CL_1=(\CL_1, \|\cdot \|)$, $\cdots$,
$\bar\CL_n=(\CL_n, \|\cdot \|_n)$ be $n+1$ integral metrized line
bundles over $X$. We want to define a Deligne paring
$$\pair {\bar \CL_0, \bar \CL_1, \cdots, \bar \CL _n}
=(\pair {\CL_0,  \CL_1, \cdots, \CL _n}, \quad \|\cdot \|)$$ as an
adelic metrized line bundle over $Y$.  Recall that $\bar \CL_i$ can
be approximated by models over $\CO _k$:
$$(\CX _i, \CM_{i0}, \cdots, \CM_{in})$$
of $(X, \CL _0^{e_i}, \cdots, \CL _n^{e_i})$ for some $e_i\in \BN$.
Without loss of generality,  we may assume that $\CX_i$ is flat and
projective over a model $\CY_i$ over $\CO_k$. Then we have a
Deligne's pairing:
$$\pair{\CM _{i0}, \cdots, \CM_{in}}\in \Pic (\CY_i).$$
This sequence of bundles on models $\CM _i$ defines an adelic
metrized line bundle on $Y$.

In the following, we would like to describe a formula for computing
norm of a section of Deligne's pairing. Let $\ell _0, \ell_1, \cdots,
\ell_n$ be non-zero sections of $\CL_i$ on $X$. By writing $\CL_i$
as linear combination of very ample line bundles and applying
Bertini's theorem, we may assume that any intersection of any subset
of $\div (\ell _i)$'s is a linear combination of subvarieties which
are smooth over $Y$. Then the pairing $\pair{\ell _0, \cdots, \ell
_n}$ is well-defined as a section of $\pair{\CL_0, \cdots, \CL_n}$.
The norm of this section can be defined by the following induction
formula:\begin{align} \log \|\pair{\ell_0, \cdots, \ell _n}\| =&\log
\|\pair{\ell
_0|_{\div \ell _n}, \cdots, \ell _{n-1}|_{\div \ell _n}}\|\notag\\
&+ \int _{X/Y}\log \|\ell _n\|c_1(\bar\CL_0)\cdots c_1(\bar \CL
_{n-1}).\end{align} We need to explain the integration in the above
formula in terms of models $(\CX_i, \CM_{i0}, \cdots, \CM_{in})$ as
above. In this case $\ell _n ^{e_i}$ extends to a rational section
$m$ of $\CM_n$. The divisor $\div (m)$ has a decomposition of Weil
divisor:
$$\div (m)=e_i\overline{\div (\ell _n)}+V_i$$
where $\overline{\div (\ell _n)}$ is the Zariski closure of $\div
(\ell _n)$ on $\CX_i$ and $V_i$ is  a divisor in the special fiber
of $\CX_i$ over $\Spec \CO_k$. Then the integral is defined as
$$\int _{X/Y}\log \|\ell _n\|c_1(\bar\CL_0)\cdots c_1(\bar \CL
_{n-1})=\lim _{i\to\infty} \frac 1{e_i^n}V_i\cdot c_1(\CM
_{i0}\cdots \CM _{i{n-1}}).$$

\subsection{Correspondences on a curve} In this subsection we want to construct
arithmetic classes for divisors on a product without using regular
models.

\subsubsection*{Decompositions}

\begin{lemma}
Let $X_1$ and $X_2$ be two varieties over a field $k$ with product
$Y=X_1\times_k X_2$. Let $e_1, e_2$
 be two rational points on $X_1$ and $X_2$, and let  $\Pic ^-(Y)$
 be the subgroup of line bundles which
 are trivial when restrict on $\{e_1\}\times X_2$ and $X_1\times
 \{e_2\}$. Then we have a decompositions of line bundles on $Y$:
 \begin{equation}
 \Pic (Y)\simeq p_1^*\Pic (X_1)\oplus p_2^* \Pic (X_2)
 \oplus \Pic ^-(Y).\end{equation}
 \end{lemma}
 \begin{proof} For any class $t\in \Pic (Y)$,  the equation
$$t=p_1^*\alpha_1 +p_2^*\alpha_2 +s, \qquad \alpha_i\in \Pic (X_i), \quad s\in \Pic ^-(Y)$$ is equivalent to
$$\alpha_2 =t|_{\{e_1\}\times X_2}, \quad \alpha _1=t|_{X\times \{e_2\}}.$$
\end{proof}

Now we assume that $X_i$ are curves over a number field. Consider an
embedding
 $$Y=X_1\times X_2\lra A:=\mathrm{Alb} (Y)=\Jac (X_1)\times \Jac (X_2)$$
 $$(x_1, x_2)\mapsto (x_1-e_1, x_2-e_2).$$
 This induces a homomorphism of groups of line bundles:
 $$\Pic (A)\lra \Pic (Y).$$
  We also have a decomposition for line
bundles on $A$ with respect to the base points $(0, 0)$ on $A$:
\begin{equation} \Pic (A)=p_1^*\Pic (\Jac (X_1))\oplus p_2^*\Pic
(\Jac (X_2))\oplus \Pic ^-(A).\end{equation}

\begin{lemma} The morphism $\Pic (A)\lra \Pic (Y)$ is
surjective. More precisely, it induces the following:
\begin{enumerate}
\item An isomorphism $\Pic ^0(A)\simeq \Pic^0(Y)$;
\item An isomorphism $\Pic ^-(A)\simeq \Pic ^-(Y)$.
\end{enumerate}
\end{lemma}

\begin{proof}
For $\Pic^0$-part this is clear as it is determined by the theorem
of square and by the corresponding isomorphism on $\Pic^0$'s on $X$
and $\Jac (X)$. To prove other parts in Proposition, we consider the
following  standard exact sequences:
$$0\lra \Pic ^0(Y)\lra \Pic (Y)\lra \NS(Y)\lra 0$$
$$0\lra \Pic ^0(A)\lra \Pic (A)\lra \NS(A)\lra 0.$$
Modulo $\Pic ^0$ in equations (2.2.1) and (2.2.2), we have
$$\NS (Y)\simeq p_1^*\NS (X_1)\oplus p_2^*\NS (X_1)\oplus \Pic
^-(Y), $$
$$\NS (A)\simeq p_1^*\NS (\Jac (X_1))\oplus p_2^*\NS (X_2)\oplus
\Pic ^-(A).$$ It remains to show that $\NS (\Jac (X_i))\lra \NS
(X_i)$ is surjective and $\Pic ^-(A)\lra \Pic ^-(Y)$ is isomorphic.
We need only prove these statements by working on $H^{1,1}$'s.
Consider the decompositions:
$$H^2(Y)=H^2(X_1)\oplus H^2(X_2)\oplus H^1(X_1)\otimes H^1(X_2)$$
$$H^2(A)=\wedge ^2H^1(A)=\wedge ^2H^2(X_1)\oplus
\wedge ^2H^1(X_2) \oplus H^1(X_1)\otimes H^1(X_2).$$ The map
$H^2(A)\lra H^2(Y)$ is induced by the identification on $H^1\otimes
H^1$'s parts and by the canonical alternative product on $\wedge ^2
H^1$-part:
$$\wedge ^2H^1(X_i)\lra H^2(X_i).$$
 Notice that
the projections
$$H^2(A)\lra \wedge ^2H^1(\Jac(A_1)), \qquad H^2(Y)\lra H^2(X_1)$$
correspond to the pull-back of the following maps:
$$\Jac (X_1)\lra A, \quad x\mapsto (x, 0)$$
$$X_1\lra Y, \quad x\lra (x, e_2).$$
Thus the decompositions in $H^2$'s are compatible with
decompositions in $\NS$'s. The surjectivity of $\NS (\Jac (X_1))\lra
\NS (X_i)$ follows from the fact that the alternative pairing on
$H^1(X_i)$ is perfect with values in $\BZ$.
\end{proof}

\subsubsection*{Admissible metrics}
Since the class map gives an embedding from $\Pic ^-(A)$ to
$H^2(A)$,  every line bundle in $\Pic ^-(A)$ is even under action by
$[-1]^*$ and thus have eigenvalues $n^2$ under action $[n]^*$. In
this way, we may construct admissible, integral, and adelic  metrics
on $\|\cdot\|$ on each line bundle in $\Pic ^-(A)$. In other words,
each $\CL$ in $\Pic ^-(A)$ can be extended into a integrable
metrized line bundle $\wh \CL=(\CL, \|\cdot\|)$ such that
$$[n]^*\wh \CL\simeq \wh \CL ^{n^2}.$$
See our previous paper \cite{adelic} for details.

The  abelian variety $A$ has an action by $\BZ^2$ by double
multiplications: for $m, n\in \BZ$,
$$[m, n]:\,  A=\Jac (X_1)\times \Jac (X_2)\lra A$$
$$(x, y)\mapsto (mx, ny).$$
In this notation the multiplication on $A$ by $\BZ$ is diagonally
embedded into $\BZ^2$. In particular these action are commutative.
By the uniqueness of the admissible metrics, the admissible metrics
is admissible with respect to the multiplication by $\BZ^2$:
$$[m, n]^*\wh \CL\simeq \wh \CL ^{mn}.$$
This shows that the bundle $\wh\CL$ is admissible in each fiber of
$A^2\lra A$ of two projection.  Thus we have shown the following:
\begin{lemma}For any class $t\in \Pic ^-(A)$, the restriction on $Y$ with
admissible metric gives an adelic metrized line bundle $\wh t$
satisfies the following conditions
\begin{itemize}
\item $\wh t$ has zero intersection with components in the
fibers over closed points for the two projects $ Y\lra X_i$;
\item $\wh t$ is trivial on $\{e_1\}\times X_2$ and on $X_1\times
\{e_1\}$.
\end{itemize}
Moreover, such an adelic structure over $t$ is unique.
\end{lemma}
\begin{proof} The difference of two different adelic structures on
$t$ satisfying the above conditions will give an adelic structure
$\wh t_0$ on the trivial bundle $t_0=\CO_Y$ satisfying the condition
in the lemma. This is certainly trivial by checking on the curves
$\{p\}\times X_1$ and $X_1\times \{p_2\}$ on closed points $p_i$ on
$X_i$.
\end{proof}

Our method above also shows that  the line bundles in $\Pic ^0(A)$
(which is odd) on any abelian variety $A$ also have integrable,
admissible, integrable metrics. Indeed, let $\CP$ be the Poincar\'e
universal  bundle on $A\times \Pic ^0(A)$ which trivial restriction
on $\{0\}\times \Pic ^0(A)$ and $A\times \{0\}$. Then $\CP$ is an
even line bundle thus admits an integrable metrized bundles. The
action by $\BZ^2$ shows that this admissible metric is admissible
fiber-wise. The following are some expressions for bundles on $\Pic
^-(A)$ and $\Pic (Y)$:

\begin{lemma}
\begin{enumerate}
\item Any line bundle $\CL\in \Pic ^-(A)$ is induced from a unique endomorphism
$\alpha\in \End (\Jac (X))$ by the following way: $$\CL =(\alpha,
1)^*\CP.$$ Moreover $\CL$ is symmetric if and only if $\alpha $ is
symmetric with respect to Rosatti involution;
\item a bundle $\CL $
in $\Pic ^-(Y)$ is symmetric (with respect to involution on
$Y=X\times X$) if and only if there is a symmetric line bundle $\CM$
on $\Pic (X)$ such that
$$\CL ^2\simeq s(\CM):= m^*\CM\otimes p_1^*\CM ^{-1}\otimes p_2^*\CM
^{-1}\otimes 0^*\CM.$$ Moreover such an $\CM$  is isomorphic to
$\Delta ^*\CL$ where $\Delta$ is the diagonal embedding $\Jac(X)\lra
A$.
\end{enumerate}
\end{lemma}

\begin{proof}
Indeed, any such $\CL$ induces an endomorphism $$\alpha: \Jac
(X)\lra \Pic (\Jac (X))=\Jac (X), \quad x\mapsto \CL |_{x\times \Jac
(X)}.$$ By universality of the Poincar\'e bundle we have that
$$\CL=(\alpha, 1)^*\CP.$$
The rest of statements in (1) is clear. If $\CL$ is symmetric, we
take
$$\CM=\Delta ^*\CL=\Delta^*(\alpha, 1)^*\CP.$$
Then we can show that
$$s(\CM)=\CL ^2.$$
\end{proof}

\subsubsection*{Example}
Let $e$ be a class in $\Pic ^1(X)$. The class $s(\Delta):=\Delta
-p_1^*e-p_2^*e$ on $X\times X$ is the pull-back of Poincar\'e bundle
via the embedding $X\lra A$ via $e$. It is also induced from the
theta divisor:
$$2s(\Delta)=-s(\Theta)|_{X\times X}.$$

\subsubsection*{Composition of arithmetic correspondences}
Let $X_i$ ($i=1,2,3$) be three curves over a number field. Let
$(\bar\CL, \|\cdot\|)$ and $\bar\CM=(\CM, \|\cdot\|)$ be integral
metrized line bundles on $X_1\times X_2$ and $X_2\times X_3$
respectively. We can define a composition $\bar\CL \circ \bar\CM$ by
Deligne pairing for the projection
$$\pi _{13*}: \quad X_1\times X_2\times X_3\lra X_1\times X_3$$
$$\bar\CL\circ \bar \CM:=\pair{\pi _{12}^*\CM, \pi _{23}^*\CL}.$$
It is easy to see that the composition is compatible with the
induced action on Chow groups of models of $X_i$.

If $\CL$ and $\CM$ are in $\Pic ^-(X_1\times X_2)$ and $\Pic
^-(X_2\times X_3)$ with respect to some base points $e_i$ and the
metrics are admissible then the composition is also an admissible
class in $\Pic ^-(X_1\times X_3)$.

\subsection{Gross--Schoen cycles}
In this subsection, we will study the height of Gross--Schoen
cycles. We will deduce a formula between the height of Gross--Schoen
cycles and the triple pairing of correspondences in Theorem 2.3.5.

Let  $X^3$ be a triple product of a smooth and projective curve $X$
over $k$. Let $e$ be a rational point on $X$. For each subset $T$ of
$\{1,2, 3\}$ define an embedding from $X$ to $X^3$ which takes $x$
to $(x_1, x_2, x_3)$ where $x_i=x$ if $i\in T$ and $x_i=e$
otherwise. Then we define the modified diagonal by
$$\Delta _e=\sum _{T\ne \emptyset}(-1)^{\#T-1}\Delta _T.$$
We may extend this cycle for case where $e$ is a divisor on $X$ of
degree $1$ as in Introduction:

\begin{lemma}[Gross--Schoen]
The cycle $\Delta _e$ is cohomologically trivial. In other words,
its class in $H^4(X^3)$ has zero cup product with elements in
$H^2(X^3)$. \end{lemma}

\begin{proof}
As
$$H^2(X^3)=\oplus _{i+j+k=2}H^i(X)\otimes H^j(X)\otimes
H^k(X),$$ any element in the above group is a sum of elements of the
form $p_{ij}^*\alpha$ where $p_{ij}$ is the projection to $(i, j)$
factors $X\times X$ and $\alpha \in H^2(X^2)$. For such form, the
pairing is given by
$$\pair{\Delta _e, p_{ij}^*\alpha}=\pair {{p_{ij*}\Delta _e},
\alpha}.$$ It is easy to show that $p_{ij*}\Delta _e=0$. Thus
$\Delta_e$ is homologically trivial.
\end{proof}

\subsubsection*{Arithmetical Gross--Schoen cycles and heights}
 Gross and Schoen have constructed a vertical
2-cycle $V$ in certain regular model of $X^3$ such that $\bar\Delta
_e-V$ is numerically trivial on each fiber,  where $\bar \Delta _e$
is the Zariski closure of $\Delta _e$. One may further extend this
to an arithmetic cycle $\wh \Delta _e=(\bar \Delta _e-V, g)$ by
adding a green current $g$ for $\Delta _e$ with curvature $0$. Thus
we have a well define pairing. More generally, let $t_1, t_2, t_3$
be three correspondence, then $t=t_1\otimes t_2\otimes t_3\in \Ch
^3(X^3\times X^3)$ is a correspondence of $X^3$, and we have a
pairing
\begin{equation} \pair{\Delta _e, t^*\Delta
_e}.\end{equation}

\subsubsection*{Triple pairing on correspondences}
In the following we want to sketch a process to relate this pairing
to some intersection numbers of cycles $t_i$ on $X^2$.

First let $\delta$ denote an idempotent  correspondence on $X$
defined by the cycle
$$\delta:=\Delta _{12}-p_1^*e.$$
Let $\delta ^3=\delta \otimes \delta\otimes \delta\in
\Ch^3(X^3\times X^3)$ denote the corresponding correspondence on
$X^3$. Then it is not difficult to show that \begin{equation} \Delta
_e=(\delta ^3)^*(\Delta _{123}).\end{equation} Indeed, the pull-back
of the cycle $p_1^*e\in C(X)$ takes every point to $e$ on $X$.

The projection in Lemma 2.2.1 is given  by the idempotent $\delta$:
$$t\mapsto t_e:=\delta\circ t\circ \delta^\vee=t-p_1^*(t^*e)-p_2^*(t_*e)
\in \Pic ^-(Y).$$ Since $\delta \circ \delta =\delta$, $(\delta^3)
^{*}\Delta _e=\Delta _e$,
$$\pair{\Delta_e, t^*\Delta _e}
=\pair{(\delta^3)^*\Delta_e, t^*(\delta^3) ^*\Delta_e } =\pair
{\Delta_e, (\delta _e^3)_*t^*(\delta^3) ^*\Delta_e} =\pair{\Delta_e,
t_e^*\Delta _e}.$$ If follows that in the expression (2.3.1) we may
assume that $t\in C(X)_e$.

Notice that the cycle $\delta $ in $X^2$ has degree $0$ for the
second projection. Thus we can construct arithmetic class $\wh
\delta$ extending $\delta$ as an integrable adelic metrized line
bundles so that it is numerically zero on fibers
 of $X^2$ via the second projection. In other words, for any point $p\in X$
 and vertical divisor $V$ on $X$, the intersection
 $$V\cdot i_p^*(\wh\delta)=0$$
 where $i_p$ is the embedding $x\lra (x, p)$.  For example, we may construct
 such a metric by decomposition
 $$\delta =(\Delta _{12}-p_1^*e-p_2^*e)+p_2^*e$$
 and put the admissible metric on the first class as in the
last subsection, and put any pull-back metric on $p_2^*e$. We may
further assume that $\wh \delta$ has trivial restriction on $X\times
\{e\}$.
\begin{lemma}
$$\wh \delta \circ \wh \delta =\wh\delta.$$
\end{lemma}
\begin{proof}
Let $\bar \CL$ be the adelic metrized line bundle with a section
$\ell$ and a divisor $\div (\ell)=\wh \delta$.  By definition,
$\wh\delta\circ \wh \delta$ is a divisor of a rational section
$$\pair{\pi _{12}^*\ell, \pi _{23}^*\ell}$$ of the line bundle
$\pair{\pi _{12}^*\CL, \pi _{23}^*\CL}$ for the projection $\pi
_{13*}: X^3\lra X^2$. By formula (2.1.1), the norm $\pair {\pi
_{12}^*\ell, \pi _{23}^*\ell}$ at a place $v$ can be written as
$$\log \|\pair {\pi _{12}^*\ell, \pi _{23}^*\ell}\|=
\log \|\pair{\pi _{12}^*\ell|_{\div \pi _{23}^*\ell}}\|+\pi
_{13*}(\log \|\pi _{12}^*\ell\|c_1(\pi _{23}^*\bar\CL)).$$

For the first term, notice that
$$\div \pi_{23}^*\ell=\pi _{23}^*\delta=X\times \Delta -X\times
\{e\}\times X.$$ Both term are isomorphic to $X\times X$ via
projection $\pi_{13}$. Thus the Deligne's pairing is given by
inversion of $\pi _{13}$
$$\pair{\pi _{12}^*\CL|_{\div\pi _{23}^*\ell}}
=\CL\otimes \alpha ^*\CL^{-1}=\CL$$ where $\alpha$ is the morphism
$$\alpha:\quad X^2\lra X^2, \quad (x, y)\mapsto (x, e).$$
The second equality is given by the assumption that $\CL$ has
trivial restriction on $X\times \{e\}$. It is easy to check that
this isomorphism takes $\pair{\pi _{12}^*\CL|_{\div\pi
_{23}^*\ell}}$ to $\ell$.

For the second term,  the restriction of  the integration on a point
$(q, p)\in X^2$ is given by integration
$$\int _X\log \|j_q^*\ell \|c_1(\bar i_p^*\bar\CL),$$
where $j_q: X\lra X^2$ is a morphism sending $x$ to $(q, x)$. This
integral is a limit of  intersection of $\bar \CL$ with some
vertical (adelic) divisors on $X$. Thus it is zero by assumption of
$\wh \delta$.
\end{proof}

We want to apply the above Lemma to construct an extension  $\wh
\Delta_e$ on $\Delta _e$ on some model which are numerically trivial
on special fiber. We will use regular models $X^3$ constructed in
Gross--Schoen \cite{gross-schoen}. Let $\CX\lra B$ be a good model
$\CX$ in the sense that the morphism  has only ordinary double
points as singular point, and that every component of fiber is
smooth. Then we can get a good model $\wt{\CX^3}$ of $X^3$ by
blowing up all components in fiber product $\CX^3$ in any fixed
order of components. Let $\wh \Delta _{123}$ be any arithmetical
cycle on $\wt{\CX^3}$ extending $\Delta _e$. By Lemma 2.1.2, the
divisors $\wh\delta^3$ defines a correspondence on $\wt{\CX^3}\times
\wt {\CX^3}$. Thus we have well defined arithmetical cycle
$(\wh\delta^3)^*\wh\Delta _e$.

\begin{lemma} The cycle  $(\wh\delta ^3)^*\wh \Delta_{123}$ is
numerically zero on every fiber of $\wh {\CX^3}$ over $\Spec \,\CO
_k$.
\end{lemma}
\begin{proof} In other words, we want to show that for any vertical
cycle $V$
$$0=(\wh \delta ^3)^*\wh \Delta_{123} \cdot V=\wh \Delta_{123}\cdot \wh\delta
^3_*V=0.$$ Actually we will to show the following
\begin{equation}\wh \delta ^3_*V=0.\end{equation}

First let us consider an archimedean place. The curvature of
$\wh\delta $ is zero on each fiber of $p_2$. Thus it has a class in
$$p_2^*H^2(X)+p_1^*H^1(X)\otimes p_2^*H^1(X).$$
In particular it is represented by a form $\omega (x, y)$ of degree
$2$ whose degree on $x$ is at most $1$. It follows that the
curvature of  $\wh\delta ^3$ is represented by a form $$\omega (x_1,
y_1)\omega (x_2, y_2)\omega (x_3, y_3)$$ on $X^3\times X^3$ whose
total degree in $x_i$'s is at most $3$. It follows that for any
smooth form $\phi$ on the first three variable $(x_1, x_2, x_3)$ of
degree $2$ the integral on $x$-variable
$$p_{456*}(\omega (x_1, y_1)\omega (x_2, y_2)\omega (x_3, y_3) \phi
(x_1, x_2, x_3))=0.$$

Now let us  consider finite places. Now let $V$ be an irreducible
vertical $2$-cycle on $\wt{\CX^3}$ over a prime $v$ of $\CO_k$. Then
there are three components $A_1, A_2, A_3$ of $\CX$ over $v$ such
that $V$ is included in the proper transformation $\wt {A_1 A_2A_3}$
of the product $A_1\times A_2\times A_3$ in $\CX^3$. Notice that
$\wt {A_1A_2A_3}$ is obtained from $A_1\times A_2\times A_3$ by
blowing up from some curves of the form $A_1\times \{p\}\times
\{q\}$, etc. Thus $V$ is a linear combination of exceptional divisor
and pull-back divisors from $A_1\times A_2\times A_3$.
 By the theorem of cube, $V$
is linear equivalent to a sum of pull-back of divisors $V_{i, j}$
via the $(i, j)$-projection:
$$V\equiv p_{12}^*V_{12}+p_{23}^*V_{23}+p_{31}^*V_{31}.$$
We may assume that $V$ is one of this term in the right, say
$$V=(p_{12}^*V_{12})_{A_1\times A_2\times A_3}
=(p_{12}^*V_{12})_{\CX^3}\cdot p_3^*A_3.$$ Now the intersection con
be computed as follows:
$$\wh \delta ^3_*V=(\wh \delta ^2_*V_{12})(\wh \delta _* A_3).$$
By definition,
$$\wh\delta _*A_3=p_{2*}(p_1^* A_3\cdot \wh \delta)$$
The cycle $p_1^*A_3\cdot \wh \delta $ in $\CX_v^2$ over each point
$y$ of $\CX_v$ is a divisor $A_3\times \{y\} \cdot \wh \delta $.
This is zero by assumption on $\wh \delta$. Thus we have shown
(2.3.3).
\end{proof}

Now we go back to the intersection number in (2.3.1) for $t_i\in
C(X)_e$. Let $\wh t_i$ be any arithmetic model of $t_i$.  There
product $\wh t$ is an arithmetic extension of the product $t$ of
$t_i$. By our construction, we see that
 $$\pair {\Delta _e,
t^*\Delta _e}= (\wh \delta^3)^*\wh\Delta_{123} \cdot \wh t^*(\wh
\delta ^3)^*\wh \Delta_{123} =\wh \Delta _{123}\cdot \wh \delta
^3_*\wh t^*(\wh \delta ^3)^*\wh \Delta _{123}.$$ Recall that $t\in
C(X)_e^3$, $\delta ^3\circ t\circ (\delta ^\vee)^3=t$. We may
replace $\wh t_i$ by $\wh \delta \circ \wh t_i\circ \wh\delta^\vee$
to assume that
\begin{equation}
\wh t_i=\wh\delta \circ \wh t_i=\wh t_i\circ \wh\delta ^\vee.
\end{equation}

 Under this assumption, the height pairing is given by
\begin{align*}
\pair{\Delta _e, t^*\Delta _e} =&\wh \Delta _{123}\cdot (\wh
t_1\otimes \wh t_2\otimes \wh t_3)^*\wh \Delta _{123}\\
=&p_{123}^*\wh \Delta _{123}\cdot (\wh t_1 \otimes \wh t_{2} \otimes
\wh t_{3})\cdot p_{456}^*\wh \Delta _{123}.
\end{align*}
Here the intersection is taken $X^6$.

As the product of the operators $\wh t_{i}$ annihilated any vertical
cycles, the above intersection number is equal to the following
expression on $\CX\times \CX$ via embedding
$$\CX^2\lra \CX^6, \qquad (x, y)\mapsto (x,x,x, y,y,y).$$
This is simply the intersection product of $\wh t_{i}$ since the
tensor product of cycles $\wh t_{i}$ are the pull-back via $p_{i,
3+i}$. Thus  we have shown the following identity:
\begin{equation}
\pair{\Delta _e, (t_1\otimes t_2\otimes t_3)^*\Delta _e} =\wh
t_{1}\cdot \wh t_{2}\cdot \wh t_{3}
\end{equation}
for cycle  $t_i\in C(X)_e$ and its extension satisfying  equation
(2.3.4).

In the following we describe  the arithmetic class $\wh t_{i}$
satisfying (2.3.4).

\begin{lemma}
The arithmetic divisors $\wh t$ on $\CX^2$ satisfying (2.3.4) are
exactly the arithmetic divisors $t\in \Pic ^-(X\times X)$ with
admissible metrics.
\end{lemma}

\begin{proof} By Lemma 2.2.3, we need only check conditions in Lemma 2.2.3.
Assume that $\wh t$ satisfies (2.3.4). From the definition of $\wh
\delta$ we see that for any vertical component
$$\wh \delta _*(v)=0, \qquad \wh \delta ^*(\bar e)=0.$$
From the expression $\wh t=\wh \delta \circ \wh t$ we see that
$$\wh t_*(v)=\wh \delta _*(\wh t _*v)=0,
\qquad \wh t^*(\bar e)=\wh t^*\delta ^* \bar e =0.$$ Similarly we
can prove other two equalities by expression $\wh t=\wh t\circ \wh
\delta$.

Now assume that $\wh t$ satisfies the condition in the Lemma.
Consider the divisor
$$\wh s:=\wh t-\wh \delta \circ \wh t\circ \wh \delta ^\vee.$$
By what we have proved, $\wh s$ is trivial on fibers over closed
points and divisor $\{\bar e\}$ for both projection. This divisor
must be trivial. Thus we must have $\wh t=\wh \delta \circ \wh
t\circ \wh \delta ^\vee$. Then the property (2.3.4) follows
immediately.
\end{proof}

In summary, we have shown the following:
\begin{theorem}For any correspondences $t_1, t_2, t_3$ in $\Pic ^-(X\times X)$ we
have
$$\pair{\Delta _e, (t_1\otimes t_2\otimes t_3)^*\Delta _e)}=
\wh t_1\cdot \wh t_2\cdot \wh t_3$$
 where $\wh t_i$ are arithmetic cycles on some model of $X^2$
 extending $t_i$ and
 satisfying conditions in Lemma 2.2.3.
\end{theorem}

\subsection{Gillet--Soul\'e's Conjectures}
By the standard conjecture of Gillet--Soul\'e \cite{gillet-soule2}
the pairing should be positively on the primitive cohomologically
trivial cycles. This implies the following
\begin{conjecture}
The following triple  pairing is semi-positive definite:
$$C(X)_e^{\otimes 3} \times C(X)_e^{\otimes 3} \lra \BR$$
\begin{align*} (t_1\otimes
t_2\otimes t_3, s_1\otimes s_2\otimes s_3) \mapsto &\wh {s_1\circ
t_1^\vee}\cdot \wh {s_2\circ
t_2^\vee} \cdot \wh{s_3\circ t_3^\vee}\\
&=\pair{(t_1\otimes t_2\otimes t_3)^*\Delta _e,\quad (s_1\otimes
s_2\otimes s_3)^*\Delta_e}.\end{align*}
\end{conjecture}

Notice that for any $t\in C(X)_e$, the correspondence $t\circ
t^\vee$ is a symmetric and positive correspondence in $C(X)_e$ in
the sense that there is a morphism $\phi: \, X\lra A$ from $X$ to an
abelian variety $A$ with ample and symmetric lines bundle $\CL$ such
that $-t\circ t^\vee$ (up to a positive multiple) is the restriction
on $X\times X$ of the Chern class of the following Poincar\'e bundle
on $A\times A$:
$$-t\circ t^\vee=s(\CL ):=m^*\CL \otimes p_1^*\CL^{-1}
\otimes p_2^*\CL ^{-1}\otimes 0^*\CL$$ where $m: A^2\lra A$ is the
addition map.

Based on the conjectured positivity of height pairing of zero
cycles; we make the following:

\begin{conjecture}
Let $X$ be a curve in abelian variety $A$ passing through $0$. Let
$\CL_i$ be three semipositive and symmetric line bundle on $A$ and
let $s(\CL_i)$ be the induced Poincar\'e bundles:
$$s(\CL_i):=m^*\CL _i\otimes p_1^*\CL_i^{-1}\otimes p_2^*\CL _i^{-1}
\otimes 0^*\CL _i.$$ Let $t_i$ be the correspondence induced by the
restriction of $s(\CL _i)$ in $X\times X$. Then
$$s(\wh \CL _1)\cdot s(\wh \CL_2)\cdot s(\wh \CL _3)|_{X\times X}\le
0,$$ where $\wh \CL_i $ is the admissible adelic metric on $\CL _i$.
Then this number vanishes if and only if cycle
$$(t_1\otimes t_2\otimes t_3)^*\Delta _e$$
is trivial.
\end{conjecture}

\subsection{Height pairing and relative dualising sheaf}
In this subsection we want to give a formula for the
self-intersection of $\Delta _e$ in terms of intersection theory of
admissible metrized line bundles in our previous paper
\cite{admissible}. Recall that in this theory, an adelic line bundle
$\CO(\wh \Delta):=(\CO (\Delta),\|\cdot\|)$ has been constructed for
a curve over a global field.  More precisely, for an archimedean
place $v$, $-\log \|1\|_v$  is the usual Arakelov function on the
Riemann surface $X_v(\BC)$. For non-archimedean place $v$, $$-\log
\|1\|(x, y)=i_v(x, y)+G_v(x, y),$$ where $i_v(x, y)$ is  the local
intersection index and $G_v(x, y)$ is a green's function on the
metrized graph $R(X_v)$. We will prove in \S3.5 that this adelic
metric line bundle is actually integrable in sense of \cite{adelic}.
In the following we assume this fact and try to prove a formula for
height of Gross--Schoen cycle.

Now fix a divisor $e$ on $X$ of degree $1$ and put a metric on it by
restriction of $\CO (\wh\Delta)$ on $X\times \{e\}$. Then we have
the admissible (adelic) divisor on $X\times X$ satisfying conditions
2.2.3:
$$\wh t_e=\wh \Delta -p_1^*\wh e-p_2^*\wh e+\wh e^2 \cdot F.$$
Here the last number $\wh e^2 \cdot F$ means  $\wh e^2$ multiple of
a vertical fiber $F$.

\begin{theorem} Assume that $g\ge 2$ and that the adelic metric line bundle
$\CO (\wh \Delta)$ is integrable. Then with notation as above
\begin{align*}
\pair{\Delta _e, \Delta_e}=\wh t^3_e=&\frac {2g+1}{2g-2}\wh\omega
^2+6(g-1)\|x_e\|^2 \\
&-\log \|1_\Delta\|\cdot (\wh\Delta ^2-6\wh \Delta \cdot p_1^*\wh
e+6p_1^*\wh e\cdot p_2^*\wh e).\end{align*} Here the last term is an
abbreviation for the adelic integration in (2.1.1) of $-\log
\|1_\Delta\|$ against the product of the first Chern class of
various arithmetic divisors involved.
\end{theorem}
\begin{proof}
By Theorem 2.3.5, we have  a formula \begin{align*} \pair{\Delta _e,
\Delta _e} =&\wh t_e^3=(\wh \Delta -p_1^*\wh e-p_2^*\wh e)^3+3\wh
e^2\cdot (\wh
\Delta -p_1^*\wh e-p_2^*\wh e)^2\\
=&\wh \Delta ^3-3\wh \Delta ^2\cdot (p_1^*\wh e+p_2^*\wh e) +3\wh
\Delta\cdot
(p_1^*\wh e+p_2^*\wh e)^2\\
&-(p_1^*\wh e+p_2^*\wh e)^3+3\wh e^2F\cdot (\wh \Delta -p_1^*\wh
e-p_2^*\wh e)^2.
\end{align*}

The last  four  terms can be simplified as follows:
$$-3\wh \Delta ^2\cdot (p_1^*\wh e+p_2^* \wh e)=-6\wh \Delta ^2\cdot p_1^*\wh
e,$$
$$3\wh \Delta \cdot
(p_1^*\wh e+p_2^*\wh e)^2 =3\wh \Delta \cdot (p_1^*\wh e^2+p_2\wh
e_2^2+2 p_1^*\wh e\cdot p_2^*\wh e) =6\wh e^2+6\wh \Delta \cdot
p_1^*\wh e\cdot p_2^*\wh e,
$$
$$
3\wh e^2F\cdot (\wh\Delta -p_1^*\wh e-p_2^*\wh e)^2 =3\wh
e^2(2-2g-2-2+2)=-6g\wh e^2,$$
$$
-(p_1^*\wh e +p_2^*\wh e)^3=-( p_1^*\wh e^3+p_2^*\wh e^3+3p_1^*\wh
e^2\cdot p_2^*\wh e+3p_1^*\wh e \cdot p_2^*\wh e^2) =-6\wh e^2.
$$
In this way we have the following expression:
\begin{align*}
\pair{\Delta _e, \Delta _e} =&\wh \Delta ^3-6\wh \Delta ^2\cdot
p_1^*\wh e+6\wh \Delta \cdot p_1^*\wh e\cdot p_2^*\wh e-6g \wh
e^2\\
=&-6g\wh e^2 +\wh \Delta \cdot (\wh \Delta ^2-6\wh \Delta \cdot
p_1^*\wh e+6p_1^*\wh e\cdot p_2^*\wh e). \end{align*}

The last term  can be written as a sum of the restriction on $\wh
\Delta$, and an intersection of  $-\log \|1_\Delta\|$ against other
cycles:
\begin{align*}
\pair{\Delta _e, \Delta _e} =&-6g \wh e^2+\wh \omega ^2-6\wh \omega
\cdot \wh e+6\wh e^2\\
& -\log \|1_\Delta\|\cdot (\wh \Delta ^2-6\wh \Delta \cdot p_1^*\wh
e+6p_1^*\wh e\cdot p_2^*\wh e).
\end{align*}
\end{proof}

When the genus of $X$ is one this formula gives $\wh t^3_e=0$.
Assume that $g>1$. Then we can get a formula  in terms of the class
of $x_e:=e-\frac 1{2g-2}\omega$ in $\Pic ^0(X)_\BQ$ using  the
formula for the Neron-Take height:
\begin{align*}
-\|x_e\|^2=&\left(\frac {\wh\omega} {2g-2}-\wh e\right)^2=\frac
{\wh\omega ^2}{4(g-1)^2}-
\frac {\wh\omega \wh e}{g-1}+\wh e^2 \\
=&\frac {\wh\omega ^2}{4(g-1)^2}-\left(\frac {\omega \wh e}{g-1}-\wh
e^2\right).\end{align*}

\begin{corollary} The pairing $\pair{\Delta _e, \Delta _e}$ gets its
minimum when $e=\xi$. More precisely, we have
$$
\pair{\Delta _e, \Delta_e}=\pair {\Delta _\xi, \Delta _\xi}+
6(g-1)\|x_e\|^2 $$
\end{corollary}

The last term in Theorem 2.5.1 is a sum of local contribution over
places of $k$. The contributions from archimedean place is easy to
compute:
\begin{proposition} At an archimedean place, the contribution in the
last term of Theorem 2.5.1 is given by $$-2\sum _{i, j, \ell}\frac
1{\lambda _\ell}\left|\int _X \phi _\ell (x)\bar\omega _i(x)\omega
_j (x)\right|^2.$$
\end{proposition}

\begin{proof} At an archimedean place, $-\log \|1_\Delta\|=G(x, y)$
is the usual Arakelov Green's function, and $\CO (\wh \Delta)$ has
curvature
$$h_\Delta (x, y)=d\mu (x)+d\mu (y)-\sqrt {-1}\sum_i \left(\omega _i(x)\bar
\omega_i (y)+\omega _i(y)\bar\omega_i (y)\right)$$ where $\omega _i$
is a bases of $\Gamma (X, \Omega )$ such that
$$\sqrt {-1}\int \omega _i \bar \omega _j=\delta _{i, j}.$$
It follows that
$$h_\Delta d\mu (x)=h_\Delta d\mu (y)=d\mu (x)d\mu (y)$$
also $\int G(x, y) d\mu (x)d\mu (y)=0$. Thus we get the formula
$$\int G  \cdot [h_\Delta ^2-3h_\Delta \cdot
(p_1^*d\mu+p_2^*d\mu)+3(p_1^*d\mu+p_2^*d\mu)^2]=\int G h_\Delta
^2.$$

Let $\phi_\ell $ be the real eigen function on $X$ of the Laplacian
for the Arakelov metric with eigenvalue $\lambda_\ell >0$ then
$$G (x, y)=\sum _\ell \frac {\phi _\ell (x)\phi _\ell
(y)}{\lambda _\ell}.$$ Since $\int_X G(x, y)d\mu (x)=\int _X g(x,
y)d\mu (y)=0$, it follows that
\begin{align*}
\int G h_\Delta ^2=&\int_{X^2} G(x, y)[d\mu (x)+d\mu (y)-\sqrt
{-1}\sum (\omega _i(x)\bar \omega_i (y)+\omega _i(y)\bar\omega
(x)]^2\\
=&-\int _{X^2}G(x, y)(\sum (\omega _i(x)\bar \omega_i (y)+\omega
_i(y)\bar\omega (x))^2 \\
=&-\int G(x, y)\sum _{i, j}\left[\omega _i(x)\bar\omega _j(x)\bar
\omega _i(y)\omega _j(y) +\bar \omega _i (x)\omega _j (x)\omega _i
(y)\bar \omega _j (y)\right]\\
=&-\sum _{i, j, \ell}\frac 1{\lambda _\ell} \left[\int \phi _\ell
(x)\omega _i(x)\bar \omega _j (x)\int\phi _\ell (y)\bar\omega
_i(y)\omega _j (y)\right. \\
&\left. +\int _X \phi _\ell (x)\bar\omega _i(x)\omega _j (x)\int _X
\phi _\ell (y)\omega _i(y)\bar \omega _j(y)\right].
\end{align*}
Since $\phi_\ell$ are all real, it follows that $$\int G h_\Delta ^2
=-2\sum _{i, j, \ell}\frac 1{\lambda _\ell}\left|\int _X \phi _\ell
(x)\bar\omega _i(x)\omega _j (x)\right|^2.$$
\end{proof}

\subsubsection*{Remark 1}
The quality in the Proposition is negative. Indeed it vanishes only
when $\bar \omega _i\omega _j$ is perpendicular to all $\phi _\ell$.
As $d\mu$ is the only measure satisfying this property, $\omega
_i\bar \omega _j$ are all proportional to each other. Thus we must
have $g=1$, and  thus a contradiction. This leads to a conjecture
that all local contributions at bad place are all negative.

\subsubsection*{Remark 2} When $X$ is hyperelliptic, Gross and
Schoen have shown that $\Delta _\xi$ is rationally equivalent to
$0$. It follows that $\wh t_\xi ^3=0$. Our conjecture thus gives a
formula for $\wh \omega ^2$ in terms of local contributions.

\section{Intersections on reduction complex}
The aim of this section is to describe an intersection theory on the
product $Z$ of two curves $X$ and $Y$ over a local field $k$ and use
this to  finish the proof of  Main Theorem 1.3.1.  The reduction map
on the usual curves over local field gives a reduction map
$$Z(\bar k)\lra R(Z):=R(X)\times R(Y)$$
where the right hand side is the product of the reduction graphs for
$X$ and $Y$. The semistable models $\CX$ and $\CY$ over  finite
extensions $k'$ gives a model $\CX\times _{\CO _{k'}}\CY$. Blow-up
these models at its singular points to get regular (but not
semistable) models $\CZ$ for $Z$. We will show that the vertical
divisors of $\CZ$ can be naturally identified with  piece-wise
linear functions on $R(Z)$. The intersection pairing on vertical
divisors can be extended into a pairing on functions $f_i$
($i=1,2,3$) on $R(Z)$ such that  each $f_{ix}$ (resp. $f_{iy}$) is
continuous as a function of $y$ (resp. $x$) except at some diagonals
$D$ as follows:
\begin{align*} (f_1, f_2, f_3)=&
\int _{R(X)\times R(Y)} \left(\Delta _x(f_1)(f_{2y}f_{3y})+\Delta
_x(f_2)(f_{3y}f_{1y})+\Delta _x(f_3)(f_{1y}f_{2y})
\right)dxdy\\
&+\frac 14\int _D \delta (f_1)\delta (f_2)\delta
(f_3)dx.\end{align*} where $\Delta _x$ and $\Delta _y$ are Laplacian
operators on piece-wise smooth functions, and $\delta (f_i)$ are
some invariants of $f_i$ on the diagonal to measure the difference
of two limits of first derivatives.

\subsection{Regular models}
In this subsection, we will study local intersection theory on a
product of two curves with semi-stable reduction. We first blow-up
the singular points in the special fiber to get a regular model.
This model has non-reduced exceptional divisors isomorphic to $\BP
^1\times \BP ^1$ but is canonical in the sense that it does not
depend on the order of blowing-ups. Also one can get semistable
models by blowing-down exceptional divisor to one of two factors
$\BP^1$. Then we give an explicit description of the  intersections
of curves and surfaces in this threefold. Finally, we show that the
inverse of the relative dualising sheaf on a semistable model can be
written as in a similar way as the restriction of the ideal sheaf on
the proper transformation of the diagonal.

 Let $R$ be a discrete valuation ring with fraction field $K$ and
algebraically closed residue field $k$. Let $X$ and $Y$ be two
smooth, absolute connected, and projective curves over $K$, and
$Z=X\times Y$  their fiber product over $R$. Assume that $X$ and $Y$
have regular and semistable models $X_R$ and $Y_R$ with no self
intersections. Then $Z$ has a model $X_{R}\times _RY_{R}$ which is
singular at products of  two singular points on special fibers $X_k$
and $Y_k$. Blowing-up these singular points we obtain a regular
model $Z_R$ over $R$.

\subsubsection*{Covering charts}
The special fiber of $Z_R$ consists of proper transformations $\wt
{AB}$ of the products of components $A$ and $B$ of $X_R$ and $Y_R$
and exceptional divisors $E_{p, q}$ indexed by singular points $p$
and $q$ of $X_k$ and $Y_k$. To see this, we cover $X_R$ and $Y_R$
formally  near their singular points by local completions of the
open affine schemes of the form:
$$V=\Spec R[x_0, x_1]/(x_0x_1-\pi), \qquad
W=\Spec R[y_0, y_1]/(y_0y_1-\pi).$$ Then $Z_R$ is covered by blow-up
 at the singular point $(x_0, x_1, y_0, y_1)$ of
$$V\times _RW=\Spec R[x_0, x_1, y_0, y_1]/(x_0x_1-\pi, y_0y_1-\pi).$$
It is clear that $Z_R$ is covered by four  charts of spectra of
subrings of  the fraction field $K(x_0, y_0)$ of $V\times W$:

$$U_{x_0}=\Spec R[x_0, y_0/x_0,
y_1/x_0]/((y_0/x_0)(y_1/x_0)x_0^2-\pi),$$
$$U_{x_1}=\Spec R[x_1, y_0/x_1,
y_1/x_1]/((y_0/x_1)(y_1/x_1)x_1^2-\pi),$$
$$U_{y_0}=\Spec R[y_0,
x_0/y_0, x_1/y_0]/((x_0/y_0)(x_1/y_0)y_0^2-\pi),$$
$$U_{y_1}=\Spec R[y_1, x_0/y_1,
x_1/y_1]/((x_0/y_1)(x_1/y_1)y_1^2-\pi).$$

In terms of valuations on $\bar K$ normalized such that $\ord
(\pi)=1$, then $\alpha _i=\ord (x_i)$, $\beta _i=\ord (y_i)$ are
non-negative with sum
$$\alpha _0+\alpha _1=\beta _0+\beta _1=1$$ and $\bar R$-points in
these charts are defined by domains of $(\alpha _0, \beta _0)\in [0,
1]^2$:
$$U_{x_0}:\qquad \min (\alpha _0, 1-\alpha _0)\ge \beta_0,$$
$$U_{x_1}:\qquad \min (\beta _0, 1-\beta _0)\ge 1-\alpha _0,$$
$$U_{y_0}:\qquad \min (\beta _0, 1-\beta _0)\ge \alpha _0,$$
$$U_{y_1}:\qquad \min (\alpha _0, 1-\alpha _0)\ge 1-\beta _0.$$
These are exactly four domains in the unit square divided by two
diagonals.

Let $A_0, A_1, B_0, B_1$ be divisors in $V$ and $W$ defined by $x_1,
x_0, y_1, y_0$ respectively. Then the special fiber of $Z_R$ is a
union of five divisors
$$\wt{A_0B_0}, \quad \wt {A_0B_1}, \quad \wt {A_1B_0}, \quad
\wt{A_1B_1}, \quad E.$$ Here the first four terms are proper
transforms of the products of curves in $V\times_R W$ and $E$ is the
exceptional divisor. Each divisor is defined by an element in each
of the above charts. For example, $E$ is defined by equations $x_0,
x_1, y_0, y_1$ in the above four charts, and $\wt {A_0B_0}$ is
defined by $y_1/x_0, 1, x_1/y_0, 1$ respectively.

\begin{figure}[ht]

    \begin{picture}(200,165)(-100,0)
        \put(0,0){\includegraphics{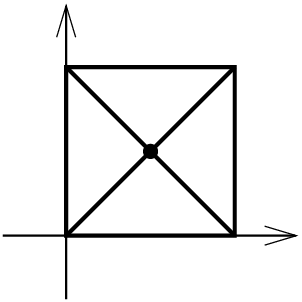}}

        \put(147,30){$\alpha_0$}
        \put(27, 150){$\beta_0$}
        \put(64, 42){$U_{x_0}$}
        \put(90, 70){$U_{x_1}$}
        \put(38, 70){$U_{y_0}$}
        \put(64, 95){$U_{y_1}$}
    \end{picture}
\caption{Reduction complex}

    \label{Fig: Reduction complex}

\end{figure}
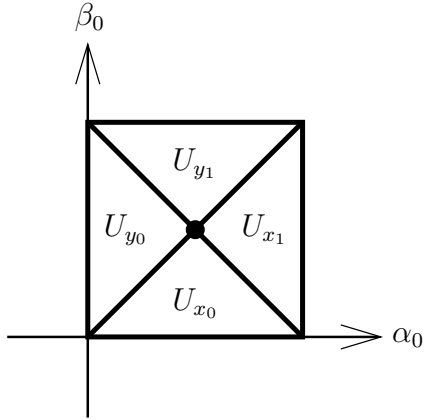
 {Figure 1 shows the reduction complex associated to $Z_R$ placed on
        $(\alpha_0, \beta_0)$-coordinate axes. The four corners and the center point of the square
        correspond to the four product components and the exceptional divisor of $Z_R$, respectively.
        The eight segments in the square correspond to the curves of intersection of the components in
        $Z_R$. The four $2$-cells labeled $U_{x_0}, U_{x_1}, U_{y_0}, U_{y_1}$ correspond to the four
        points of $Z_R$ where three components meet transversally.}

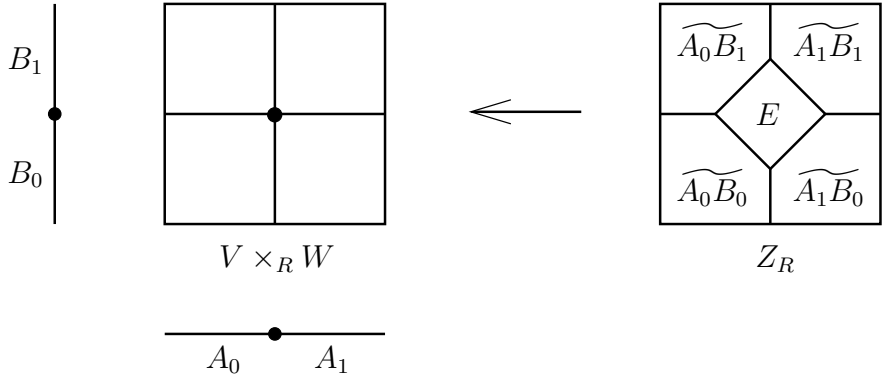
\begin{figure}[ht]
\begin{picture}(200,165)(-25,-25)
        \put(0,0){\includegraphics{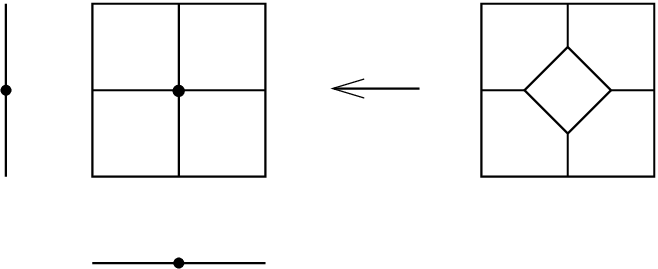}}

        \put(-15, 60){$B_0$}
        \put(-15, 103){$B_1$}
        \put(60, -10){$A_0$}
        \put(101, -10){$A_1$}
        \put(65, 28){$V \times_R W$}

        \put(268,28){$Z_R$}
        \put(238, 53){$\widetilde{A_0B_0}$}
        \put(282, 53){$\widetilde{A_1B_0}$}
        \put(238, 108){$\widetilde{A_0B_1}$}
        \put(282, 108){$\widetilde{A_1B_1}$}
        \put(268, 82){$E$}
    \end{picture}

    \caption{Special fibers of blow-ups }
    \label{Fig: Blow-up square}

\end{figure}

Figure 2 shows the configuration of special fibers before and after
blow-ups. On the left is a diagram representing the various product
components in the
        special fiber of $V \times_R W$ and how they project onto each factor. On the right is
        a diagram representing $Z_R$, which is the blow-up of $V \times_R W$. The components
        $\widetilde{A_iB_j}$ are strict transforms of product components in $V \times_R W$.
        The component shaped like a diamond in the diagram on the right collapses to the singular
        point at the center of the cross in the left diagram.

\subsubsection*{Intersections} Back to the global situation. The
following properties are easy to verified:
\begin{itemize}
\item  each component $\wt{AB}$ is obtained from $A\times
B$ by blowing at the singular points of $X_R\times Y_R$  on $A\times
B$, and has multiplicity one in divisor $(\pi)$;
\item two different  components $\wt {A_0B_0}$ and $\wt
{A_1B_1}$ intersect if and only if either $A_0=A_1$ and $B_0\cap
B_1\ne \emptyset $, or $B_0=B_1$ and $A_0\cap A_1\ne \emptyset$;
\item each exceptional divisor is isomorphic to $\BP ^1\times \BP^1$
and has multiplicity $2$ in $(\pi)$;
\item two component $\wt {AB}$ and $E_{p,q}$
insect if and only if $p\in A$ and $q\in B$; \item two different
exceptional divisors do not intersect.
\end{itemize}

In the following we want to compute the intersection numbers more
precisely. Assume that $p=A_0\cdot A_1$, $q=B_0\cdot B_1$. The
intersection can be described as follows: \begin{itemize}
\item $\wt {A_0B_0}\cdot \wt {A_0B_1}$
is given by the proper transformations $\wt {A_0 q}$ of $A_0\times
q$ in $\wt {A_0B_0}$ and $\wt {A_0B_1}$;
\item one may choose an isomorphism $E_{p,q}\simeq \BP
^1\times \BP ^1$ so that the following hold:
$$\wt {A_0B_0}\cdot E_{p,q}= \BP ^1\times 0, \qquad \wt
{A_1B_1}=\BP ^1\times \infty,$$
$$\wt {A_0B_1}\cdot E_{p, q}=0\times \BP ^1, \qquad \wt
{A_1B_0}=\infty \times \BP ^1.$$
\end{itemize}

We may also compute the self intersection of vertical divisors in
Chow group using the following equation: for any vertical divisor
$F$ in $Z_k$,
$$0=F\cdot (\pi)=F\cdot (\sum_{A, B} \wt {AB}+2\sum _{p, q}E_{p, q}),$$
where the sums are over components $A, B$ and singular points $p$
and $q$ of $X_{k}, Y_{k}$. It follows that
$$\wt {AB}^2=-\sum _{q}\wt {A q}-\sum_{p} \wt {pB}-2\sum_{p, q}
\text{exceptional divisors over $(p, q)$},$$
$$E_{p, q}^2=-\BP ^1\times 0-0\times \BP ^1.$$
Here the sums are over singular points $p, q$ in $A, B$.

In the following we want to compute the intersection numbers between
a curve $C$ and a surface $F$ included in the special fiber $\CZ_k$
of $\CZ$. Assume that $C$ is included in a surface $G$ then
$$C\cdot F=(C\cdot F_G)_G$$
where $F_G$ is the pull-back of $F$ in $G$ via the inclusion $G\lra
Z_R$, and the right hand is an  intersection  in $G$. Thus to study
intersection pairing it suffices to study the intersection pairing
of $Z$ with the subgroup $B(G)$ of divisors of $G$ generated by
$F_G=F\cdot G$ in $\NS(G)$, the N\'eron--Severi group of $G$.

\begin{lemma}
The intersection pairing on $B(G)$ is non-degenerate. \end{lemma}
\begin{proof}
If $G=\wt{AB}$, this group is generated by $\NS(A)$, $\NS (B)$ via
projections and the exceptional divisors. It is clear that the
intersection pairing on $B(G)$ is non-degenerate. If $G=\BP ^1\times
\BP ^1$, then $B(G)=\NS (G)$ and the intersection pairing is clearly
non-degenerate.
\end{proof}

By this lemma, we may replace $C$ by its projection $B(C)$ in
$B(G)$. As all $B(G)$'s are generated by intersections $F\cdot G$,
we need only describe the intersection of three surfaces in $Z_k$.
Let $F_1, F_2, F_3$ be three components. \begin{itemize} \item If
they are all distinct, then the intersection is non-zero only if
they have the following forms after an reordering
$$F_1=\wt {AB_0}, \qquad F_1=\wt{AB_1}, \qquad F_3=E_{p,
q}$$ where $p$ is a singular point on $A$ and $q=B_0\cdot B_1$. In
this case the intersection is $1$:
$$F_1\cdot F_2\cdot F_3=1.$$
\item
If $F_1=F_2\ne F_3$ then
$$F_1\cdot F_2\cdot F_3=i^*(F_1)^2, \qquad i: F_3\lra X.$$
Furthermore if $F_1=\wt {AB_0}$, $F_3=\wt {AB_1}$, then $i^*F_1=\wt
{Aq}$ and
$$F_1^2\cdot F_3=i^*(F_1)^2=-s(A):=-\text {number of singularity of
$X_{k}$ on $A$}.$$ \item If $F_1=\wt {AB}$, $F_3=E_{p, q}$ with
$p\in A$ and $q\in B$, then $i^*F_1$ is a $\BP ^1$ on $F_3\simeq \BP
^1\times \BP ^1$ of degree $(1, 0)$ or $(0,1)$. It follows that
$$F_1^2\cdot F_3=0.$$
\item If $F_1=E_{p, q}, F_3=\wt{AB}$ with $p\in A$, $q\in B$, then
$i^*F_1$ is one exceptional divisor on $F_3$ and then
$$F_1^2\cdot F_3=-1.$$
\item Finally if $F_1=F_2=F_3$ then
$$0=F^2\cdot (\pi)=F^2\cdot (\sum_{A, B} \wt {AB}+2\sum _{p, q}E_{pq}).$$
It follows that
$$\wt {AB}^3=2s(A) s(B), \qquad \wt E_{p, q}^3=2.$$
\end{itemize}

\subsubsection*{Relative dualising sheaf} Now assume that $X=Y$. Let
$\wt \Delta_R\subset Z_R$ be the Zariski closure of the diagonal in
$Z_R$. Then $\wt \Delta_R$ is the blowing-up of $X_R$ at its double
points in the special fiber. Let $i: \wt \Delta_R\lra Z_R$ and $f:
\wt \Delta_R\lra X_R$ be the induced morphisms and let $\omega $ be
the relative dualising sheaf on $X_R$.
\begin{lemma}
$$f^*\omega =i^*\CO _{Z_R}(-\wt \Delta_R).$$
\end{lemma}

\begin{proof}

Since the question is local, we may assume that $X_R$ is given by
$$\Spec R[x_0, x_1]/(x_0x_1-\pi)$$
then the relative dualising sheaf is given by the subsheaf of
$\Omega ^1_{X_R/R}\otimes K(X)$ generated by $dx_0/x_0=-dx_1/x_1$.
The scheme $Z_R$ is obtained by blowing up the singular point on
$X_R\times _R X_R$ and is  covered by $$U_{x_0},\quad U_{x_1},\quad
U_{y_0}, \quad U_{y_1}.$$ The subscheme $\wt \Delta_R$ is defined by
an ideal $I$ generated by  $y_0/x_0-1$, $y_1/x_1-1$ in these charts
and has coverings given by
$$V_{x_0}=\Spec R[x_0,
x_1/x_0]/((x_1/x_0)x_0^2-\pi),$$
$$V_{x_1}=\Spec R[x_1, x_0/x_1]/((x_0/x_1)x_1^2-\pi).$$
 As $dx_0$ and $dx_1$ are
the image of $y_0-x_0$ and $y_1-x_1$ on $\wt \Delta_R$, we see that
$I/I^2$  is generated by $dx_0/x_0=-dy_0/y_0$.
\end{proof}

\subsection{Base changes and reduction complex}
In this subsection, we describe the pull-back of vertical divisors
respect to base changes. The direct limit of vertical divisors can
be identified with piece-wise linear functions on the reduction
complex which is the product of metrized graphs.

 Let $S$ be a ramified extension of $R$ of
degree $n$ with fraction field $L$. Let $Z_S$ be the model of $Z_L$
obtained by the same way as $Z_R$. In the following we want to
describe the morphism $\CZ_S\lra \CZ_R$ in terms of charts. As this
question is local, we may assume that $X_{R}$ and $Y_{R}$ are given
by
$$X_{R}=\Spec R[x_0, x_1]/(x_0x_1-\pi), \qquad
Y_{R}=\Spec R[y_0, y_1]/(y_0y_1-\pi).$$ Then $Z_R$ is obtained by
blowing up at the singular point $(x_0, x_1, y_0, y_1)$ of
$X_R\times _R Y_R$ and is covered by four charts of spectra  of
subrings of  the fraction field $K(x_0, y_0)$ of $X_R\times Y_R$:
$$U_{x_0}, \quad U_{x_1}, \quad U_{y_0}, \quad U_{y_1}.$$

Let $t$ be a local parameter of $S$ such that $\pi=t^n$. For
integers $a, b\in [0, n-1]$, set
$$x_{0a}=x_0/t^a, \quad x_{1a}=x_1/t^{n-1-a},\quad
y_{0b}=y_0/t^b, \quad y_{1b}=y_1/t^{n-1-b}.$$ Then $X_{S}$ and $Y_S$
are   unions of the following spectra:
$$V_{a}=\Spec R[x_{0,a}, x_{1, a}]/(x_{0,a}\cdot x_{1, a}-t),
\qquad 0\le a\le n-1.$$
$$W_{b}=\Spec R[y_{0b}, y_{1b}]/(y_{0b}\cdot y_{1b}-t),
\qquad 0\le b\le n-1.$$ In terms of valuations on $\bar R$-points,
$\alpha =\ord _\pi (x_0)$, $\beta =\ord _\pi (y_0)$, $V_\alpha$ and
$W_\beta$ are  defined by inequalities
$$a/n\le \alpha \le (a+1)/n, \qquad b/n\le \beta \le (b+1)/n.$$

The special component of $X_S$ has $n$-singular points with one on
$V_a$ each and is the union of $n+1$-components $A_{n, a}$
($a=0\cdots n$) defined by $x_{1a}$ in $U_{a}$ if $a\le n-1$, by
$x_{0, a-1}$ in $U_{ a-1}$ if $a\ge 1$, and by $1$ on other
components. Similarly, we have components $B_{n, b}$ for $Y_S$. The
product $V_a\times W_b$ has the special fiber to be a union of
$$A_{n, a}\times B_{n, b}, \quad A_{n, a}\times B_{n, b+1},
\quad A_{n, a+1}\times B_{n, b}, \quad A_{n, a+1}\times B_{n,
b+1}.$$ The scheme $Z_S$ is covered by  the blowing up $U_{a, b}$ of
$V_a\times W_b$ at its singular point.
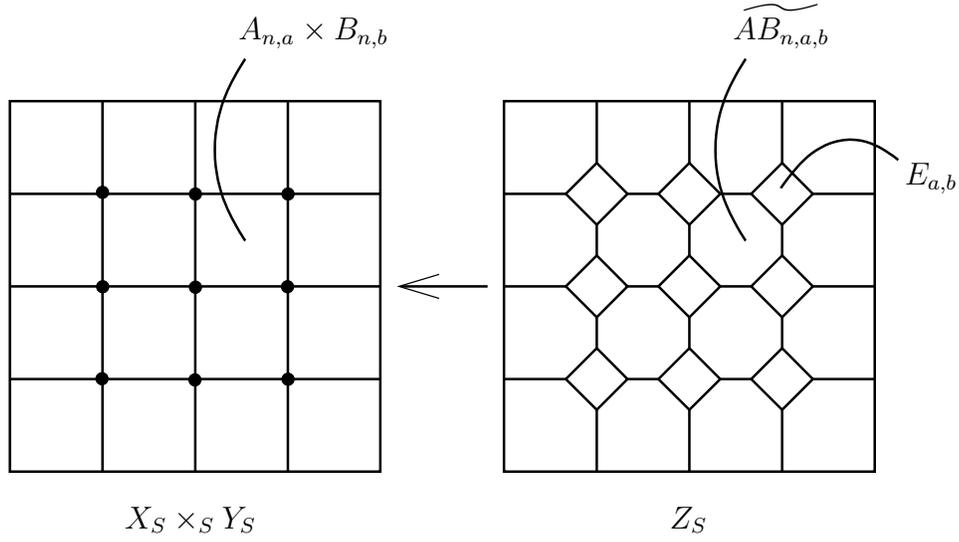
\begin{figure}[ht]

    \begin{picture}(200,245)(-100,-20)
        \put(-80,10){\includegraphics{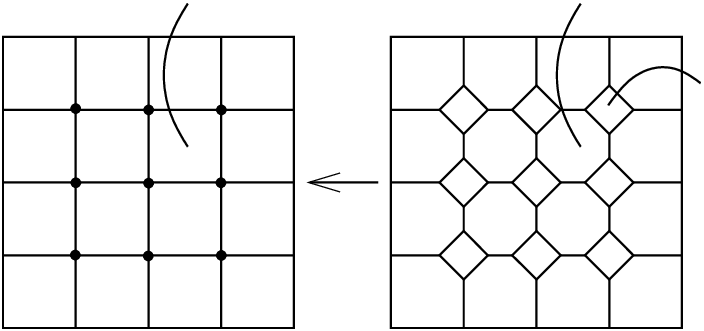}}

        \put(-35, -10){$X_S \times_SY_S$}
        \put(171, -10){$Z_S$}
        \put(8, 175){$A_{n,a} \times B_{n,b}$}
        \put(195, 175){$\widetilde{AB_{n,a,b}}$}
        \put(260, 120){$E_{a,b}$}
    \end{picture}

    \caption{Special fibers of base changes}
    \label{Fig: Base Change}

\end{figure}
Here we have diagrams of the schemes $X_S \times_S Y_S$ and the
blow-up
    along the singular points, denoted $Z_S$ (for the case $n=3$). We have labeled a general
    product component $A_{n,a} \times B_{n,b}$ as well as its strict transforms
    $\widetilde{AB_{n,a,b}}$. We have also labeled one of the exceptional divisors $E_{a,b}$ that
    arises from the blow-up

The scheme $U_{a, b}$ is covered by  four affine schemes with
equations:
$$U_{x_0, a, b}: \quad \left(\frac {y_{0b}}{x_{0a}}\right)
\left(\frac {y_{1b}}{x_{0a}}\right)(x_{0a})^2=t$$
$$U_{x_1, a, b}:\quad \left(\frac {y_{0b}}{x_{1a}} \right)
\left(\frac {y_{1b}}{x_{1a}}\right)(x_{1a})^2=t$$
$$U_{y_0, a, b}: \quad \left(\frac {x_{0a}}{y_{0b}}\right)
\left(\frac {x_{1a}}{y_{0b}}\right)(y_{0b})^2=t$$
$$U_{y_1, a, b}: \quad \left(\frac {x_{0a}}{y_{1b}}\right)
\left(\frac {x_{1a}}{y_{1b}}\right)(y_{1b})^2=t.$$ The divisor $(t)$
has five components over $V_a\times W_b$:
$$\wt {AB_{n, a, b}}, \quad \wt{AB_{n, a, b+1}},
\quad \wt {AB_{n, a+1,b}}, \quad \wt {AB_{n, a+1, b+1}}, \quad E_{a,
b}.$$

In terms of valuations $\alpha =\ord _\pi (x_0)$, $\beta =\ord _\pi
(y_0)$, $U_{a, b}$ are defined by following inequalities:
$$\min (\alpha _0-a/n, 1/n-(\alpha _0-a/n))\ge \beta_0-b/n$$
$$\min (\beta _0-b/n, 1/n-(\beta _0-b/n))\ge 1/n-(\alpha _0-a/n)$$
$$\min (\beta _0-b/n, 1/n-(\beta _0-b/n))\ge (\alpha _0-a/n)$$
$$\min (\alpha _0-a/n, 1/n-(\alpha _0-a/n))\ge 1/n-(\beta _0-b/n)$$
These are  parts divided by diagonals in the square $[a/n,
(a+1)/n]\times [b/n, (b+1)/n]$. Thus the morphism from $Z_S$ to
$Z_R$ is given by the inclusion of the parts in $[0, 1]^2$.

\begin{figure}[ht]

    \begin{picture}(200,225)(-75,-10)
        \put(0,0){\includegraphics{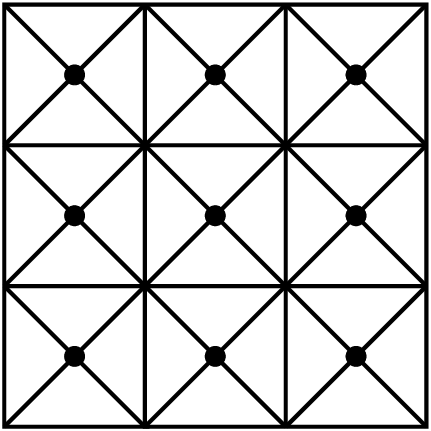}}

    \end{picture}

    \caption{Reduction complex of base changes}
    \label{Fig: Base Change Reduction Complex}

\end{figure}
Here is a diagram of the reduction complex associated to the special
fiber in
        Figure~\ref{Fig: Base Change}. The vertices of the complex correspond to the components
        of $Z_S$, with the nine fat points corresponding to the 9 exceptional divisors. (They have
        multiplicity $2$ in the special fiber.)

\subsubsection*{Pull-back of vertical divisors} In the following we
want to compute the pull-back of vertical divisors $Z_R$ in $Z_S$.
Let $\varphi: Z_S\lra Z_R$ denote the morphism. Let us index
divisors using set $\Lambda _n$  on $[0,1]^2$ of the form
$\left(\frac a{2n}, \frac b{2n}\right)$:
$$D_{a/n, b/n}:=\wt {AB_{n, a, b}}, \qquad a, b\in \BZ$$
$$D_{a/n, b/n}:=2E_{a-1/2, b-1/2}\qquad a, b \in \BZ+\frac 12$$
\begin{lemma}
$$\varphi ^*\wt {AB}=n\sum _{(a, b))\in \Lambda _n}
\max (1-a-b, 0)D_{a, b}.$$
$$\varphi ^*(E)
=n\sum _{(a, b)\in \Lambda}\min (a, 1-a, b, 1-b)D_{a, b}.$$
\end{lemma}

\begin{proof}
First we notice that   $\varphi^*\wt {A_0B_0}$  is defined as zeros
of $x_1/y_0=y_1/x_0$ on $U_{x_0}$ and $U_{y_0}$,  and $1$  on
$U_{x_1}$ and $U_{y_1}$. Thus it is defined by $1$ on $U_{x_i, a,
b}$ and $U_{y_i, a, b}$ if $a+b\ge n$. It follows that the
multiplicity of $\wt {AB_{n, a, b}}$ and $E_{a, b}$ are zero if
$a+b\ge n$.  Now we assume that $a+b<n $. Then $x_1/y_0=y_1/x_0$ on
$U_{x_0}$ has the following expressions in the charts $U_{x_0,a, b}$
and $U_{y_0, a, b}$:
$$\frac {y_1}{x_0}=\frac {y_{1b}}{x_{0a}}t^{n-1-a-b}
 =\left(\frac {y_{0b}}{x_{0a}}\right)^{n-a-b}
\left(\frac
{y_{1b}}{x_{0a}}\right)^{n-1-a-b}(x_{0a})^{2(n-1-a-b)}.$$
$$\frac {x_1}{y_0}=\frac
{x_{1a}}{y_{0b}}t^{n-1-a-b}=\left(\frac
{x_{0a}}{y_{0b}}\right)^{n-1-a-b} \left(\frac
{x_{1a}}{y_{0b}}\right)^{n-a-b}(y_{0b})^{2(n-1-a-b)}$$

Either one  of these formulae shows that the pull-back of $\wt
{A_0B_0}$ has multiplicity $n-a-b$ at $\wt{AB_{n, a, b}}$, and
$2(n-1-a-b)$ at $E_{a, b}$. This proves the first formula in Lemma.

For exceptional divisor, we may using the following decompositions
$$\div (\pi)=\sum_{i, j=0}^1 \wt {A_iB_j}+2E,$$
$$\div (t)=\sum _{(a, b)\in \Lambda }D_{a, b}.$$
The fact $\varphi^*\div (\pi)=n\div (t)$ implies that
$$\varphi ^*(E)
=n\sum _{(a, b)\in \Lambda}\min (a, 1-a, b, 1-b)D_{a, b}.$$
\end{proof}

\subsubsection*{Remarks}
Alternatively, we may compute pull-back of an exceptional divisor
directly by using charts:
$$x_0=\left(\frac {y_{0b}}{x_{0a}}\right)^{a} \left(\frac
{y_{1b}}{x_{0a}}\right)^{a}(x_{0a})^{2a+1}$$
$$x_1=\left(\frac {y_{0b}}{x_{1a}} \right)^{n-a-1}
\left(\frac {y_{1b}}{x_{1a}}\right)^{n-a-1}(x_{1a})^{2n-2a-1}$$
$$y_0=\left(\frac {x_{0a}}{y_{0b}}\right)^b
\left(\frac {x_{1a}}{y_{0b}}\right)^b(y_{0b})^{2b+1}$$
$$y_1=\left(\frac {x_{0a}}{y_{1b}}\right)^{n-b-1}
\left(\frac {x_{1a}}{y_{1b}}\right)^{n-b-1}(y_{1b})^{2n-b-1}.$$

\subsubsection*{Reduction complex} Let $\CC(X)$ and $\CC(Y)$ be the
reduction graphs of $X$ and $Y$ respectively with reduction
morphisms
$$r_X: \quad X(\bar K)\lra \CC (X), \qquad r_Y:\quad Y(\bar K)\lra \CC(Y).$$
Recall that $R(X)$ and $R(Y)$ are metrized graphs with edges of
lengths $1$ parameterized by irreducible components and singular
points in special fibers of $X_R$ and $Y_R$. The reduction map is
given as follows. An edge $E\simeq [0,1]$ corresponds to a singular
point near which $X_R$ has local structure
$$R[x_0, x_1]/(x_0x_1-\pi)$$
such that $0$ and $1$ correspond to $x_1=0$ and $x_0=0$
respectively. Then the reduction morphism is given by
$$(x_0, y_0)\lra \ord (x_0).$$
Here $\ord (x_0)$ is a valuation on $\bar K$ such that $\ord
(\pi)=1$. The reduction map of a point is a vertex  (resp. a smooth
point in a edge) if and only if the reduction modulo $\pi$ of this
point is in a corresponding smooth point in a component (resp.  a
singular point). After a base change $L/K$ of degree $n$, the dual
graph is unchanged if we change the lengths of edges to be $1/n$. In
other words, the irreducible components of $X_L$ and singular points
corresponding to rational points on $\CC (X)$ with denominator $n$
and intervals between them.

Let us define the reduction complex of $Z=X\times Y$ to be
$$\CC(Z_R):=\CC(X_R)\times \CC(Y_R)$$
with a triangulation by adding diagonals. We have induced reduction
map. The vertices in the complex correspond to irreducible
components in $Z_R$; the edges correspond to intersection of two
components; the triangle correspond to intersection of three
components.

The reduction complex of base change $[L:K]=n$ after a change of
size coincides with the same complex with an $n$-subdivision of
squares and then an triangulation on it. Thus we may define $\CC
(Z)$ the complex without triangulation.

Let $V(X_R)$ denote the group of divisors with real coefficients
supported in the special fiber. Let $\BR(\CC(Z))$ denote the space
of continuous real functions on $\CC (Z)$. Then we can define a map
$$V(X_R)\lra \BR (\CC (Z)), \quad F\mapsto f _{R, F}$$ with following
properties: write
$$F=\sum_C a_CC+\sum _E 2b_EE$$
where $C$ runs through all non-exceptional components of $Z_R$, and
$E$ all exceptional components, then
\begin{itemize}
\item $f _{R, F}$ is linear on all triangles.
\item $f_{R, F}(r(C))=a_C$,
\item $f _{R, F}(r(E))=b_E$.
\end{itemize}

Let $V(X)$ denote the direct limit of $V(X_S)$ via pull-back map in
the projective system $X_S$ defined by finite extensions of $R$ in
$\bar R$. The main result in the last subsection gives the
following:

\begin{lemma} The map $[S:R]^{-1}f_{S}$ induces a map
$$\phi: \quad V(X)=\lim V(X_S)\lra \BR (\CC(X)).$$
Moreover the image of this map are continuous function which are
linear on some $n$-triangulation.  So it is dense in the space of
continuous functions.
\end{lemma}

\subsection{Triple pairing}
In this section we are try to define a triple pairing for functions
in $F(\CC (Z))$. More precisely, let $f_1, f_2, f_3$ be three
continuous functions on $\CC (Z)$. Then for any positive integer
$n$, let us define piece wise linear functions $f_{i, n}$ such that
$f_{i,n}$ is linear on each triangle of the $n$-triangulation, and
has the same values as $f_i$ at vertices of triangles. Then
$f_{i,n}$ will correspond to vertical divisors $F_{i,n}$ in
$V(Z_{R_n})$ where $R_n$ is a ramified extension of degree $n$. Lets
define the triple pairing
$$(f_{1,n}, f_{2,n}, f_{3,n})=n^2 (F_{1,n}\cdot F_{2,n}\cdot F_{3,n})$$
where the right hand side is the intersection pairing on $Z_{R_n}$.
Notice that if $f_i=f_{i, 1}$, then $f_{i,n}=f_i$ and
$F_{i,n}=n^{-1}\varphi _n^*D_{i,1}$ by \S3.2, where $\varphi_n$ is
the projection $Z_{R_n}\lra Z_R$. It follows that
$$n^2(F_{1,n}\cdot F_{i,n}\cdot F_{i,n})=n^{-1}\varphi _n^*F_{1,1}\cdot
\varphi_n ^*F_{2,1}\cdot \varphi _n ^* F_{3,1}=F_{1,1}\cdot
F_{2,1}\cdot F_{3,1}.$$ Thus the above pairing does not depend on
the choice of $n$ if every $f_i=f_{i1}$. We want to examine when the
limit does exist and what expression we can get for this limit.

\begin{proposition}
Assume that  the functions $f_1, f_2, f_3$ on $\CC(Z)$ are smooth on
each square with bounded first and second derivatives.  Then the
intersection pairing on vertical divisors induces a trilinear
pairing
$$(f_1, f_2, f_3):=\int _{\CC
(Z)}(f_{1x}f_{2y}f_{3xy}+\text{permutations})dxdy,$$  where the
integrations are taken on the smooth part of the complex and
$f_{1x}$, $f_{2y}$, etc are partial derivatives for any directions
on edges of $R(X)$ and $R(Y)$.
\end{proposition}

\begin{proof}

Our first remark is that from the formulae given in \S3.1, the
computation can be taken as a sum of intersections  on squares. In
other words, we may assume that both $X_R$ and $Y_R$  have one
singular point. Then the complex $\CC (Z)$ can be identified with
the square $[0,1]^2$. By calculation in \S3.1 and \S3.2, we have the
following expression of divisors:
$$F_{i,n}=\sum _{(a, b)\in \Lambda _n}f_i(a, b) D_{a, b}.$$
Again the intersection can be taken on sum of small squares:
\begin{align*}&n^2(F_{1n}\cdot F_{2n}\cdot F_{3n})\\ =&n^2\sum _{a,
b=0}^{n-1}\prod _{i=1}^3(f_i(a, b)D_{a, b}+f_i(a+1/n, b)D_{a+1/n,
b}+f_i
(a, b+1/n)D_{a, b+1/n}\\
&+f_i(a+1/n, b+1/n)D_{a+1/n, b+1/n}+f_i (a+1/2n, b+1/2n)D_{a+1/2n,
b+1/2n})_{a, b}\end{align*} where the last product is the
intersection on the square starting at $(a, b)$. As the sum
$$D_{a, b}+D_{a+1/n, b}+D_{a, b+1/n}+D{a+1/n, b+1/n}+D_{a+1/2n,
b+1/2n}$$ has zero intersection with products, we subtract each
coefficient by $D_{a+1/2n}$. Thus the last product has the form
$$\prod _{i=1}^3(a_i D_{a, b}+b_iD_{a+1/n, b}+c_i D_{a, b+1/n}
+d_iD_{a+1/n, b+1/n})$$ with
$$a_i=f_i(a, b)-f_i(a+1/2n, b+1/2n),$$
$$ b_i=f_i(a+1/n, b)-f_i(a+1/2n, b+1/2n),$$
$$c_i=f_i(a, b+1/n)-f_i(a+1/2n, b+1/2n), $$
$$d_i=f_i(a+1/n, b+1/n)-f_i(a+1/2n, b+1/2n).$$

We use the following facts to compute this product among the
divisors in the sum:
\begin{itemize}
\item the product of three distinct element will be $0$;
\item the product of square of one divisor with another divisor is
$-1$, if they intersect, and $0$ otherwise;
\item the cube of any element is $2$.
\end{itemize}
Then we have \begin{align*} &\prod _{i=1}^3(a_i D_{a,
b}+b_iD_{a+1/n, b}+c_i D_{a, b+1/n}
+d_iD_{a+1/n, b+1/n})\\
=&2(a_1a_2a_3+b_1b_2b_3+c_1c_2c_3+d_1d_2d_3)\\
&-(a_1a_2+d_1d_2)(b_3+c_3)-(b_1b_2+c_1c_2)(a_3+d_3)+\text
{permutations}
\end{align*}

Write  Taylor expansions for $a_i, b_i, c_i, d_i$ at $a'=a+1/2n,
b'=b+1/2n$:
$$\alpha _i=\frac 1{2n}(f_{ix}(a', b')+f_{iy}(a', b')),\quad \beta _i=
\frac 1{8n^2}(\Delta f _i(a', b')+2f_{ixy}(a', b'))$$
$$\gamma  _i=\frac 1{2n}(f_{ix}(a', b')-f_{iy}(a', b')),\quad \delta _i=
\frac 1{8n^2}(\Delta f _i(a', b')-2f_{ixy}(a', b'))$$ Then
$$a_i=-\alpha _i+\beta _i+O(1/n^3),$$
$$b_i=\gamma _i
+\delta _i+O(1/n^3),$$
$$c_i=-\gamma _i
+\delta _i+O(1/n^3),$$
$$d_i=\alpha _i
+\beta _i+O(1/n^3).$$ It is clear that the product is an even
function in $\alpha _i$ and $\gamma _i$. It follows that their
appearance in the product have the even total degree. Also we have
neglected term $O(1/n^5)$. Thus we can write
\begin{align*} &\prod _{i=1}^3(a_i D_{a,
b}+b_iD_{a+1/n, b}+c_i D_{a, b+1/n}
+d_iD_{a+1/n, b+1/n})\\
=&4(\alpha _1\alpha _2\beta _3+ \gamma _1\gamma _2\delta _3)-4\alpha
_1\alpha _2\delta _3-4\gamma _1\gamma _2\beta _3+\text
{permutations}\\
=&4(\alpha _1\alpha _2-\gamma _1\gamma _2)(\beta _3-\delta _3)+\text
{permutations}
\end{align*}

By a direct computation, we see that
$$\alpha _1\alpha _2-\gamma_1\gamma _2=\frac 1{2n^2}
(f_{1x}(a', b')f_{2y}(a', b')+f_{1y}(a', b')f_{2x}(a', b'))$$
$$\beta _3-\delta _3=\frac 1{2n^2}f_{ixy}(a', b').$$
It follows that
\begin{align*} &\prod _{i=1}^3(a_i D_{a,
b}+b_iD_{a+1/n, b}+c_i D_{a, b+1/n}
+d_iD_{a+1/n, b+1/n})\\
=&\frac 1{n^4}f_{1x}(a', b')f_{2y}(a', b')f_{ixy}(a', b')+\text
{permutations}+O(1/n^5)
\end{align*}
Put everything together to obtain
$$n^2(F_{1n}\cdot F_{2n}\cdot F_{3n})=\frac 1{n^2}\sum _{a, b}^{n-1}
(f_{1x}(a', b')f_{2y}(a', b')f_{ixy}(a', b')+\text
{permutations})+O(1/n)$$ This is of course convergent to
$$\int _0^1\int _0^1(f_{1x}f_{2y}f_{3xy}+\text{permutations})dxdy.$$
Adding  all integrals over squares we obtain the identity in Lemma.
\end{proof}

\subsection{Intersection of  functions in diagonals}

Now we want to treat case where $f_i$ has some singularity. We
assume that all singularity lies on edges of some $n$-triangulation.

 As
the additive property stated in the last subsection, we need only
consider the diagonal in a square. More precisely, we consider
functions $f_i$ on the square $[0,1]^2$ with following properties:
\begin{itemize}
\item $f_i$ is continuous and smooth in two triangles divided by
diagonal $y=x$;
\item the first and second derivatives of $f_i$ are bounded on two
triangles; \item the restriction of the function on the diagonal is
smooth.
\end{itemize}

Let us write
$$\delta f_i=f_{ix}^+-f_{ix}^-=f_{iy}^--f_{iy}^+.$$
where we use super script $\pm$ to denote the limits of derivatives
in upper and down triangles on the diagonals. The second identity
follows from the fact that the restriction of $f$ on the diagonal is
smooth.

\begin{proposition}If $f_i$ has only singularity on
some union $\CD$ of diagonals in $n$-triangulation, the singular
contribution is given by
$$\int _\CD\left(-\frac 12 \delta f_1\delta f_2\delta f_3+f_{1x}f_{2y}\delta f_3
+\text{permutations}\right)dx.$$
\end{proposition}

\begin{proof}
The total contribution  in the diagonal is given by
\begin{align*}&n^2(F_{1n}\cdot F_{2n}\cdot F_{3n})_\CD\\ =&n^2\sum _{a,
0}^{n-1}\prod _{i=1}^3(f_i(a, a)D_{a, a}+f_i(a+1/n, a)D_{a+1/n,
a}+f_i
(a, a+1/n)D_{a, a+1/n}\\
&+f_i(a+1/n, a+1/n)D_{a+1/n, a+1/n}+f_i (a+1/2n, a+1/2n)D_{a+1/2n,
a+1/2n})_{a, a}.\end{align*}  Again we may replace the last product
by \begin{align*} &\prod _{i=1}^3(a_i D_{a, a}+b_iD_{a+1/n, a}+c_i
D_{a, a+1/n} +d_iD_{a+1/n, a+1/n})\\
=&2(a_1a_2a_3+b_1b_2b_3+c_1c_2c_3+d_1d_2d_3)\\
&-(a_1a_2+d_1d_2)(b_3+c_3)-(b_1b_2+c_1c_2)(a_3+d_3)+\text
{permutations}
\end{align*}
 with
$$a_i=f_i(a, a)-f_i(a+1/2n, a+1/2n),$$
$$ b_i=f_i(a+1/n, b)-f_i(a+1/2n, a+1/2n),$$
$$c_i=f_i(a, a+1/n)-f_i(a+1/2n, a+1/2n), $$
$$d_i=f_i(a+1/n, a+1/n)-f_i(a+1/2n, a+1/2n).$$

We have Taylor expansions for $a_i, b_i, c_i, d_i$ at $(a', a')$
with $a'=a+1/2n$:
$$a_i=-\frac 1{2n}(f_{ix}^-+f_{iy}^-)(a', a')+O(1/n^2),$$
$$b_i=\frac 1{2n}(f_{ix}^--f_{iy}^-)(a', a')+O(1/n^2),$$
$$c_i=\frac 1{2n}(-f_{ix}^++f_{iy}^+)(a',a')+O(1/n^2),$$
$$d_i=\frac 1{2n}(f_{ix}^++f_{iy}^+)(a', a')+O(1/n^2).$$
Write
$$f_{ix}=\frac 12 (f_{ix}^++f_{ix}^-), \quad f_{iy}=\frac 12 (f_{iy}^++f_{iy}^-),$$
then,
$$f_{ix}^\pm =f_{ix}\pm \frac 12 \delta f_i, \qquad
f_{iy}^\pm =f_{iy}\mp \frac 12\delta f_i.$$

Then we have the following expressions:
$$a_i=-\frac 1{2n}(f_{ix}+f_{iy})(a')+O(1/n^2),$$
$$b_i=\frac 1{2n}(f_{ix}-f_{iy}-\delta f_i)(a')+O(1/n^2),$$
$$c_i=\frac 1{2n}(-f_{ix}+f_{iy}-\delta f_i)(a')+O(1/n^2),$$
$$d_i=\frac 1{2n}(f_{ix}+f_{iy})(a')+O(1/n^2).$$

As $a_i+d_i=O(1/n^2)$, it follows that
\begin{align*} &\prod _{i=1}^3(a_i D_{a, a}+b_iD_{a+1/n, a}+c_i
D_{a, a+1/n} +d_iD_{a+1/n, a+1/n})\\
=&2(b_1b_2b_3+c_1c_2c_3)-2a_1a_2(b_3+c_3)+\text
{permutations}+O(1/n^4)\\
=&\frac {-1}{2n^3}\delta f_1\delta f_2\delta f_3(a') +\frac 1{n^3}
f_{1x}f_{2y}\delta f_3(a')+\text{permutation}+O(1/n^4)
\end{align*}

The total diagonal  intersection is
\begin{align*}&n^2(F_{1n}\cdot F_{2n}\cdot F_{3n})_\CD\\ =&
\frac 1n\sum_{a=0}^{n-1}(
 \frac {-1}{2}\delta f_1\delta f_2\delta
f_3(a') + f_{1x}f_{2y}\delta
f_3(a')+\text{permutation})+O(1/n).\end{align*} Taking limits of sum
over all , we get singular contribution:
$$\int _0^1(-\frac 12 \delta f_1\delta f_2\delta f_3+f_{1x}f_{2y}\delta f_3
+\text{permutations})dx.$$
\end{proof}

In the following we want to apply integration by parts to the smooth
formulae in Proposition 3.3.1 and 3.4.1:

\begin{theorem}
Assume that each $f_{ix}$ (resp. $f_{iy}$) is continuous as a
function of $y$ (resp. $x$) except at some diagonals. Then the
intersection is given by
\begin{align*} (f_1, f_2, f_3)=&
\int _{R(C)^2} \left(\Delta _x(f_1)(f_{2y}f_{3y})+\Delta
_x(f_2)(f_{3y}f_{1y})+\Delta _x(f_3)(f_{1y}f_{2y})
\right)dxdy\\
&+\frac 14\int _\CD \delta (f_1)\delta (f_2)\delta
(f_3)dx.\end{align*}
\end{theorem}
\begin{proof}
 Let us first group the smooth contribution as product of
derivatives of $x$:
$$(f_1, f_2, f_3)_\text{smooth}=
\int _{R(X)^2}[ f_{1x}(f_{2y}f_{3y})_x+f_{2x}(f_{3y}f_{1y})_x+
f_{3x}(f_{1y}f_{2y})_x]dxdy.$$ Now we apply integration by parts to
obtain $$(f_1, f_2, f_3)_\text{smooth}= \int _{R(X)^2} \Delta
_x'f_1(f_{2y}f_{3y})dxdy+\int
_\CD(f_{1x}^+f^+_{2y}f^+_{3y}-f^-_{1x}f^-_{2y}f_{3y}^-)dy+\cdots
$$
Here   $\CD$ is some union of diagonals in an $n$-triangulation, and
the coordinates in each square are  chosen such that $\CD$ is given
by $x=y$, $\Delta _x'f_i$ is the restriction of $\Delta _x f_i$ on
the complement of $\CD$, and $f_i^+$ and $f_i^-$ are restrictions of
$f_i$ in the upper and lower triangles respectively. Define the
derivatives of $f_i$ at the diagonal as the average of two
directions diagonals. Then we have the formulae:
$$f_{ix}^\pm =f_{ix}\pm \frac 12 \delta f_i, \qquad f_{iy}^\pm
=f_{iy}\mp \frac 12 \delta f_i.$$ The integrand in the diagonal on
$\CD$ has the following expression:
$$(f_{1x}+\frac 12 \delta f_1)(f_{2y}-\frac 12 \delta
f_2)(f_{3y}-\frac 12 \delta f_3)- (f_{1x}-\frac 12 \delta
f_1)(f_{2y}+\frac 12 \delta f_2)(f_{3y}+\frac 12 \delta f_3).$$ It
is clear that the above expression is odd in $\delta$; thus it has
an expression
$$\delta
(f_1)f_{2y}f_{3y}+\frac 14 \delta (f_1)\delta (f_2)\delta (f_3)
-f_{1x}f_{2y}\delta (f_3)-f_{1x}f_{3y}\delta (f_2).$$ As the
restriction of $\Delta _xf_i$ on $\CD$ is $\delta (f_i)$,
\begin{align*}
(f_1, f_2, f_3)_\text{smooth}= &\int _{R(X)^2} \Delta
_xf_1(f_{2y}f_{3y})dxdy+\cdots\\
&+\int _\CD\left(\frac 34 \delta (f_1)\delta (f_2)\delta
(f_3)-f_{1x}f_{2y}\delta (f_3)-\cdots\right)dy
\end{align*}
Combined with singular contribution, we have the following
expression for the total pairing: \begin{align*} (f_1, f_2, f_3)=&
\int _{R(X)^2} \left(\Delta _x(f_1)(f_{2y}f_{3y})+\Delta
_x(f_2)(f_{3y}f_{1y})+\Delta _x(f_3)(f_{1y}f_{2y})
\right)dxdy\\
&+\frac 14\int _\CD \delta (f_1)\delta (f_2)\delta
(f_3)dx.\end{align*}
\end{proof}

\subsection{Completing proof of main theorem}
In this subsection, we will complete the proof of Main Theorem
1.3.1. By Theorem 2.5.1 and Proposition 2.5.3, it remains to compute
the quantity in 2.5.1 in the local setting. More precisely, let
$X=Y$ be a curve of genus $g\ge 2$ on over a local field $K$. Let
$G$ be the admissible green's function on the reduction graph
constructed in our inventiones paper \cite{admissible}. Then we have
a metrized line bundle $\CO (\wh \Delta)$ with norm $\|\cdot\|$
given by
$$-\log \|1_\Delta\|=i(x, y)+G(R(x), R(y))$$ where $R: X(\bar K)\lra
R(X)$ is the reduction map. We want to show the following

\begin{proposition} The adelic metrized bundle $\CO (\wh \Delta)$ is
integrable and
\begin{align*}
&-\log \|1_\Delta\|\cdot (\wh \Delta^2-6\wh \Delta\cdot p_1^*\wh
e+6p_1^*\wh e\cdot p_2^*\wh e)\\
=&-\frac 14\delta (X)+\frac 14 \int _{R(X)}G(x,
x)((10g+2)d\mu_a-\delta _{K_X}).\end{align*}
\end{proposition}

To see that $\CO(\wh \Delta)$ is integrable, we let $\xi$ denote a
class of degree $1$ such that $(2g-2)\xi=\omega _X$ and put an
admissible metric on it. Then we have seen that  the class
$$\wh t:=\wh \Delta -p_1^*\wh\xi-p_2^* \wh\xi$$
is integrable. On the other hand the adjunction formula gives
$$\Delta ^*\xi=-\wh\omega _X-2\wh\xi=-2g\wh \xi.$$
Thus $\wh \xi$ is integrable and then $\wh \Delta$ is integrable.

In the following, let us give  precise models of $\Delta$ over $\CO
_K$ which converges to $\wh \Delta$. We consider integral models
$\CZ_{\CO_L}$  of $X_L\times X_L$ for finite Galois extension $L$ of
$K$. The special fiber $\CZ_w$ of $\CZ_{\CO _L}$ over a finite place
$w$ over a place $v$ of $K$ has components parameterized by some
$e(w)$-division points in the reduction complex $R(Z_w)$ where
$e(w)$ is the ramification index of $w$ over $v$. Let $G_w$ be the
restriction of the Green's function $G$ on these points. Then we get
a vertical divisor $V_w$ in $\CZ_w$ with rational coefficients. The
divisor $\wt \Delta +\sum_w V_w$ with Green's function at
archimedean place defines an arithmetic divisor $\wh \Delta _L$.

We claim that this divisor is the pull-back of some divisor on some
model $\CZ_{\CO_K}^L$ over $\CO_K$. Indeed, let $\CL$ be an ample
line bundle on $\CZ_{\CO_L}$ invariant under $\Gal (L/K)$; for
example, we may take
$$\CL=\CO (\sum \text{-Exceptional divisors})
\otimes \pi_1^*\omega^n\otimes \pi _2^*\omega ^n,$$ where $n$ is
some big positive number. In this way, we may write
$$\CZ_{\CO _L}=\Proj\oplus _{m\ge 0}\pi _*\CL ^m$$
where $\pi$ is the projection $\CZ_{\CO_L}\lra \Spec \CO_K$. It is
well known that the algebra $(\oplus _{m\ge 0}\pi _*\CL ^m)^{\Gal
(L/K)}$ of $\Gal (L/K)$ invariants is finitely generated and thus
defines a $\CO_K$-scheme:
$$\CZ_{\CO_K}^L=\Proj(\oplus _{m\ge 0}\pi _*\CL ^m)^{\Gal (L/K)}.$$
It is well known that the inclusion of rings defines a morphism:
$$\phi_{L/K}:\quad \CZ_{\CO_L}\lra \CZ_{\CO_K}^L.$$
In this way, we get a divisor
$$\wh \Delta ^L:=[L:K]^{-1}\phi _{L/K *}(\wh \Delta _L).$$
To get a metrized line bundle on $\CZ_{\CO_K}^L$, we may take
positive integer $t$ such that $t\wh \Delta$ has integral
coefficients, and then take a norm
$$\RN _{\phi_{L/K}}(\CO (t\wh\Delta _L))$$
which is an arithmetical model of $t\Delta$. By our definition the
integration of $-\log \|1_\Delta\|$ against curvatures of line
bundles is the limit of intersection of $V_i$ with arithmetic
divisors. Thus we may replace $-\log \|1_\Delta\|$ in the
proposition by Green's function $G$ on the reduction complex. We
will finish the proof of the Proposition by computing the triple
pairings one by one in the following three lemmas.

\begin{lemma}
$$(G, \wh \Delta, p_1^*\wh e)=0.$$
\end{lemma}

\begin{proof}
Using  formulae $\wh \Delta =\wt \Delta+G$ and $\wh e=\bar e+G_e$,
we have decomposition \begin{align*} (G, \wh \Delta , p_1^*\wh e)
=&(G, p_1^*\wh e)_{\wt \Delta}+(G, G, p_1^*\wh e)\\
=&(G, p_1^*\wh e)_{\wt \Delta } +(G, G)_{p_1^*e}+(G, G, p_1^*
G_e).\end{align*} Let us to compute each of these of term:
$$(G, p_1^*\wh e)_{\wt \Delta}=\int _{R(X)}G(x, x)d\mu.$$
\begin{align*}
(G, G)_{p_1^*e}=&-\int _{R(X)}G(e, y)\Delta _y G(e, y)dy\\
=&-\int _{R(X)}G(e, y)(\delta _e(y)-d\mu (y))=-G(e, e)
\end{align*}
\begin{align*}
(G, G, p_1^*G_e) =&\int \Delta _x p_1^*G_e (G_y)^2dxdy\\
=&\int_{R(X)^2} (\delta _e (x)-d\mu (x))(G_y(x, y)^2)dy\\
=&\int _{R(X)}G_y(e, y)^2dy-\int _{R(X)^2}G_y(x, y)^2dyd\mu (x)\\
=&\int _{R(X)}\Delta _yG_y(e, y)\cdot G_y(e, y)dy-\int
_{R(X)^2}\Delta _y
G_y(x, y)\cdot G_y(x, y)dyd\mu (x)\\
=&G(e, e)-\int _{R(X)^2}(\delta _x(y)-d\mu (y)) G_y(x, y)dyd\mu
(x)\\
=&G(e, e)-\int _{R(X)^2}G_y(x, x)d\mu (x).
\end{align*}
The lemma follows from the above three formulae.
\end{proof}

\begin{lemma}
$$(G, p_1^*\wh e, p_2^*\wh e)=0.$$
\end{lemma}

\begin{proof}
Using  the decomposition of cycles, we have the following
expression:
$$(G, p_1^*\wh e, p_2^*\wh e)
=(G, p_2^*\wh e)_{p_1^*e}+(G, p_1^*G_e)_{p_2^*e}+(G, p_1^*G_e,
p_2^*G_e).$$ We compute each term as follows:
$$(G, p_2^*\wh e)_{p_1^*e}=\int G(e, y)d\mu (y)=0$$
$$(G, p_1^*G_e)_{p_2^*e}=-\int G(x, e)\Delta _x G(x, e)=-G(e, e)$$
\begin{align*}
(G, p_1^*G_e, p_2^*G_e)=&\int \Delta _x G_e(x, e) \cdot (G_y(x, y)
G_{y}(e, y))dxdy\\
=&\int (\delta _e(x)-d\mu (x) \cdot (G_y(x, y)
G_{y}(e, y))dy\\
=&\int_{R(X)}G_y(e, y) G_{y}(e, y)dy- \int G_y(x, y) G_{y}(e,
y))d\mu (x)dy\\
=&\int \Delta _y G(e, y)\cdot G(e, y)dy=G(e, e)
\end{align*}
The lemma follows from the above three computations.
\end{proof}

\begin{lemma} Let $g$ be the genus of the curve, then
$$(G,\wh \Delta, \wh \Delta)=\frac 14 \delta(X)+\frac 14\int G(x,
x)(K_X-(10g+2)d\mu ).$$ \end{lemma}

\begin{proof} Using the formula, $\wh \Delta =\wt \Delta +G$,
 the left hand side can be
decomposed as follows:
\begin{align*}
&(G, \wh \delta  )_\Delta+(G, G)_{\Delta}+(G, G, G)\\ =&\int
_{\CC(X)}-G(x, x)c_1(\omega )-\int _{\CC (X)}G(x, x) \Delta G(x, x)
+ (G, G, G).\end{align*} The curvature $c(\omega )$ is $(2g-2)d\mu$
where $d\mu$ is the admissible metric. To compute Laplacian of $G(x,
x)$ we use the following formula
$$c+G_a(K_X, x)+G(x, x)=0$$
It follows that the curvature of $\Delta (G(x, x))$ is given by
$$-\Delta G(K_X, x)=(2g-2)d\mu -K_X.$$ Here $K_X$ is the canonical divisor
on $\CC(X)$. Thus we have the formula
$$(G, \wh \Delta, \wh \Delta)=\int _{\CC (X)}G(x, x)(K_X-4(g-1)d\mu)
+(G, G, G).$$

It remains to compute the triple pairing $(G, G, G)$. Notice that
$G_x$ (resp. $G_y$) are continuous in $y$ (resp. $x$) except on
diagonal. By the main formula in the last section, we have $$(G, G,
G)=\frac 14\int _D(\delta G)^3dx+3\int \Delta _x G G_y^2dxdy.$$ By
definition,
$$(\Delta _x G) dx=\delta _y (x)-d\mu.$$
It follows that $\delta (G)=1$ on the diagonal and thus the first
integral is $$\frac 14 \ell(R(X))=\frac 14\delta (X).$$ The second
integral is given by
$$3\int _{\CC (X)}G_y^2 (y, y)dy-3\int_{\CC(X)^2} G_y ^2 (x, y)d\mu (x)dy.$$
Recall that $G_y(y, y)$ is defined to be $$\frac  12 (G_y^+(y,
y)+G_y^-(y, y))= \frac  12 (G_y^+(y, y)+G_x^+(y, y))=\frac 12 G(y,
y)_y.$$ It follows that the above integral  is given by
\begin{align*}
&\frac 34\int G^2(y, y)_ydy-3\int _{\CC(X)^2}\Delta _y G(x, y)G(x,
y)d\mu (x)dy\\
=&\frac 34\int \Delta _yG(y, y)G(y, y)dy-3\int _{\CC(X)^2}\Delta _y
G(x, y)G(x, y)d\mu (x)dy\\
=&\frac 34\int G(y, y)((2g-2)d\mu -K_X)-3\int _{\CC(X)^2}G(x,
y)(\delta _x(y)-d\mu (y))d\mu (x)\\
=&\frac 34\int G(y, y)((2g-2)d\mu -K_X)-3\int _{\CC(X)}G(x, x)d\mu\\
=&\int _{\CC (X)}G(x, x)\left(\frac 32( g-3)d\mu-\frac 34K_X\right)
\end{align*}
The lemma follows from the above computations.
\end{proof}

\section{Integrations on metrized graph}
In this section,  we reformulate Conjectures 1.4.2 and 1.4.5 in
terms of metrized graphs. We will verify the conjectures in the
elementary graphs where every edge is included in at most one
circle. We will conclude the section by reducing the conjecture to
the case that the graph is $2$-edge connected in the sense that the
complement of any point is still connected.

\subsection{Some conjectures on metrized graphs}
In this subsection we want to reformulate Conjectures 1.4.2 and
1.4.5 in terms of metrized graphs. We will also give some trivial
formula which can be used to prove Theorem 1.3.5.

 Let $\Gamma $ be a connected metrized graph and
let $q$ be a   function on $\Gamma $ with a finite support. We
define the canonical divisor of $(\Gamma , q)$ by
$$K:=\sum _{x\in \Gamma }(v(x)+2q(x)-2)x.$$
The genus of the  metrized graph is defined to be
$$g=1+\frac 12 \deg K=\sum _x q(x)+b(\Gamma )$$
where $b(\Gamma)$ is the first Betti number of the (topological)
graph $\Gamma$ without metric.  We say that the pair $(\Gamma, q)$
is a {\em polarized metrized graph} if the following conditions hold
\begin{itemize}
\item $q$ is non-negative; \item  $K$ is effective.\end{itemize}
Notice that the reduction graph $R(X)$ of any semistable curve $X$
of genus $g$  over a discrete valuation ring is a polarized metrized
graph of genus $g$.

 Let $G(x, y)$
and $d\mu$ be the admissible green's function and metric associate
to the pair $(\Gamma , q)$. We are interested in the following
constants:
$$\varphi (\Gamma ):=-\frac 14\ell (\Gamma )+\frac 14\int _\Gamma G(x, x)((10g+2)d\mu -\delta
_K),$$
$$\lambda (\Gamma ):=\frac {g-1}{6(2g+1)}\varphi (\Gamma)+\frac 1{12}(\epsilon(\Gamma)+\ell (\Gamma)),$$
where $\ell (\Gamma)$ is the total length of $\Gamma$ and
$$\epsilon(\Gamma):=\int G(x, x)[(2g-2)d\mu +\delta _K].$$
When $\Gamma =R(X)$ is the reduction graph for a curve, then
notation of  invariants here coincides with the invariant defined in
the introduction except we use $\ell (\Gamma)$ for $\delta (X)$
there.

A point $p\in \Gamma$ is called {\em a smooth point} if it is not in
the support of $K$. For such a smooth point $p$, let $\Gamma_p$ be
the subgraph obtained from $\Gamma$ by removing $p$ and attached two
points $p_1, p_2$. More precisely, $\Gamma_p$ is a metrized graph
with a surjective map to $\Gamma$ which is injective and isometric
over $\Gamma\setminus \{p\}$ and two-to-one over $p$. The function
$q $ defines a function on $\Gamma_p$. We call $p$  of {\em type
$0$} if $\Gamma_p$ is connected. In this case   $\Gamma_p$ has genus
$g-1$. If $p$ is not of type $0$, then $\Gamma_p$ is a union of two
connected graphs of genus $i$ and $g-i$ for some $i\in (0, g/2]$. In
this case, we say that $p$ is of {\em type $i$}. For each number $i$
in the interval $[0, g/2]$ let $\Gamma_i$ be the subgraph of $\Gamma
$ of points of type $i$. Let $\ell _i(\Gamma)$ denote the length of
$\Gamma_i$. It is easy to see that there are only finitely many
$i\in [0, g/2]$ with non-zero $\ell _i(\Gamma)$.

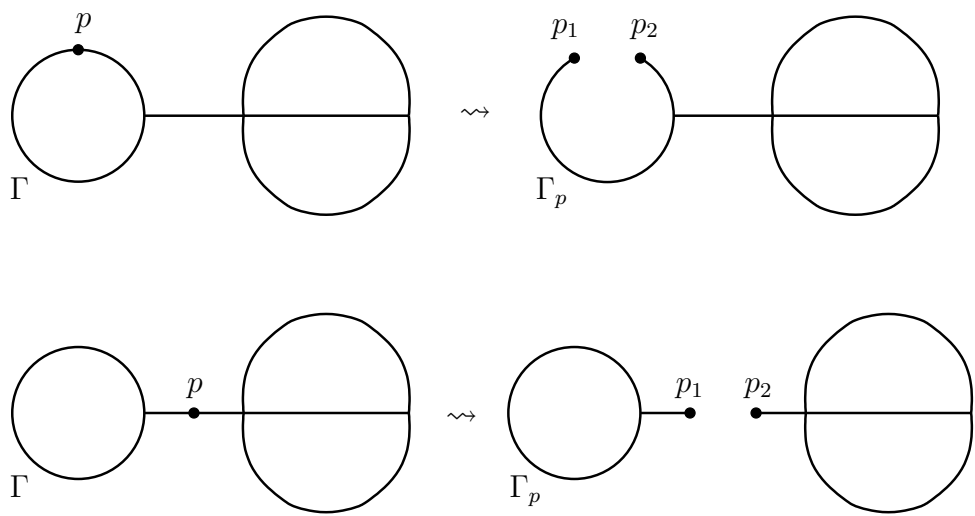
\begin{figure}[ht]

    \begin{picture}(200,205 )(-100,0)
        \put(-85,0){\includegraphics{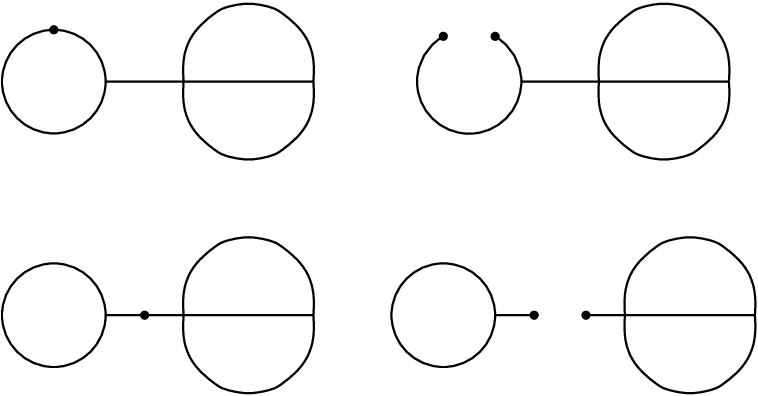}}
        \put(-85, 120){$\Gamma$}
        \put(-85,7){$\Gamma$}
        \put(114, 120){$\Gamma_p$}
        \put(104,7){$\Gamma_p$}

        \put(-60, 185){$p$}
        \put(120, 183){$p_1$}
        \put(150, 183){$p_2$}

        \put(-18, 47){$p$}
        \put(167, 47){$p_1$}
        \put(193, 47){$p_2$}

        \put(85, 150){$\rightsquigarrow$}
        \put(80, 35){$\rightsquigarrow$}

    \end{picture}

    \label{Fig: Point Types}
    \caption{Type of smooth points}

\end{figure}
Here is a diagram illustrating the definition of the type of a point
for a graph with $q=0$. In the left figures, $p$
        is a smooth point of $\Gamma$. In the top, $p$ is of type~0 because $\Gamma_p$ is connected.
        In the bottom, $p$ is of type~$1$ because the minimum genus of the two connected
        components of $\Gamma_p$ is $1$.

\begin{conjecture} There is positive function $c(g)$ of $g>1$ such
that
$$\varphi (\Gamma)\ge c(g)\ell _0(\Gamma)+\sum _{i\in (0, g/2]}
\frac {2i(g-i)}g\ell _i(\Gamma).$$
$$\lambda (\Gamma)\ge \frac g{8g+4}\ell _0(\Gamma)+\sum _{i\in (0, g/2]}
\frac{i(g-i)}{2g+1}\ell _i(\Gamma).$$
\end{conjecture}

\subsubsection*{Formulae for Green's functions and admissible
metrics}
 We need to have a formula for $G(x, x)$ in
terms of resistance $r(x, y)$. Recall that we always have a formula
like \begin{equation} r(x, y)=G(x, x)-2G(x, y)+G(y,
y).\end{equation} See formula (3.5.1) in \cite{admissible}. Double
integrations gives
\begin{equation} \tau (\Gamma):=\int G(x, x)d\mu(x)=\frac 12 \int
r(x, y)d\mu (x)d\mu (y).\end{equation} One integral with $d\mu (y)$
gives
\begin{equation}
G(x, x)=\int r(x, y)d\mu (y)-\frac 12 \int r(x, y)d\mu (x)d\mu
(y).\end{equation} Bring this to the definition of $\epsilon
(\Gamma)$ to obtain
\begin{equation} \epsilon(\Gamma)=\int r(x, y)\delta _K (x)d\mu
(y).\end{equation} The constants $\varphi (\Gamma)$ and $\lambda
(\Gamma)$ can be expressed in terms of $\ell (\Gamma)$, $\tau
(\Gamma)$ and $\epsilon (\Gamma)$:
\begin{equation}
\varphi (\Gamma)=3g\tau (\Gamma )-\frac 14(\epsilon (\Gamma )+\ell
(\Gamma))\end{equation} \begin{equation} \lambda (\Gamma )=\frac
{g(g-1)}{2(2g+1)}\tau (\Gamma )+\frac {g+1}{8(2g+1)}(\ell
(\Gamma)+\epsilon (\Gamma )).\end{equation}

 Recall form Lemma 3.7 in \cite{admissible} that $d\mu$ has an expression
\begin{equation}
d\mu =\frac 1{g}\left(\sum q(x)\delta _x+\sum \frac {dx_e}{\ell
_e+r_e}\right).\end{equation}

We will reduce Conjecture 4.1.1 to the case where $\Gamma $ is
2-edge connected. In this case, the conjecture is equivalent to the
following

\begin{conjecture}
Assume that $\Gamma$ is 2-edge connected.  Then the following two
inequalities hold: $$\frac {g-1}{g+1}(\ell (\Gamma )-4g\tau (\Gamma
)) \le \epsilon (\Gamma )\le 12 g\tau (\Gamma )-(1+c(g))\ell (\Gamma
),$$ here $c(g)$ is a positive number for each $g>1$.
\end{conjecture}

\subsection{Proof of Theorem 1.3.5}
In this subsection, we give a trivial bound for $\varphi (\Gamma )$
and use it to complete the proof of Theorem 1.3.5.

\begin{lemma}
$$-\frac {2g-1}4 \ell (\Gamma )\le \varphi (\Gamma )\le \frac {3g}2\ell (\Gamma ).$$
\end{lemma}

\begin{proof}
From formulae (4.1.5) and (4.1.2), we obtain  $$\varphi (\Gamma )\le
3g\tau (\Gamma )\le \frac {3g\ell (\Gamma )}2$$ where we use an
inequality $r(x, y)\le \ell (\Gamma )$ for any points $x, y\in
\Gamma$. Similarly, we can get a lower bound:
$$\varphi (\Gamma )
\ge -\frac 14(\epsilon (\Gamma )+\ell (\Gamma )) \ge -\frac
14(2g-2+1)\ell (\Gamma ).$$
\end{proof}

\subsubsection*{Proof of Theorem 1.3.5}
To prove Theorem 1.4.4, we need only to prove that the following
difference function is bounded, for all closed point $t\in T$:
$$f(t)=\frac 1{\deg t}\left((2g-2)\pair{\Delta _\xi(Y_t), \Delta _\xi
(Y_t)}-(2g+1)\pair{\omega _{Y/T}, \omega _{Y/T}}\right).$$ Replace
$T$ by a finite covering, we may assume that the family can be
extended into an semi-stable family $\CY\lra \CT$ of integral
schemes over $\CO_K$.  Let $\delta _\CT$ be the boundary divisor
induced from the morphism $\CT\lra \bar \CM_g$ and the boundary
divisor $\bar\CM_g\setminus \CM_g$. Then $\delta _\CT$ is supported
over finitely many closed fibers of $\CT\lra \Spec \CO_k$, say over
points in a finite subset $S$ of $\Spec \CO_k$. Now by Theorem
1.3.1, the function is given by
$$f(t)=\frac {2g-2}{\deg t}\sum _w\varphi _w(Y_t)$$
where the sum is over all places of $K(t)$. When $w$ is archimedean
over an archimedean place $v$ of $K$ , $\phi _w(Y_t)$ is a
continuous function on $t\in T_v(\BC)$ thus it is bounded by a
constant $C_v$ depends only on place $v$.

If $w$ is archimedean, then by Lemma 4.2.1, $\varphi _w(Y_t)$ is
bounded by a constant multiple of the length $\ell (\Gamma )$ of the
reduction graph of $Y_t$ at $w$. We notice that this length is equal
to the number of singular points on $Y_t$ over $w$ and can be
computed by divisor $\delta _{\CT}$:
$$\ell (\Gamma )=(\delta _{\CT}\cdot \bar t)_w$$
where the right hand side is the local intersection number of
$\delta _\CT$ and the Zariski closure $\bar t$ of $t$ over $w$. This
number is also bounded by a number $C_v$ as $\delta _\CT$ is a
vertical divisor. In summary we have shown that
$$|f(t)|\le \frac 1{\deg t}\sum _w C_v=\sum C_v$$
where $C_v$ are some constant which is zero at all but finitely many
places of $k$. Thus this is a finite number. This shows the
boundedness of $f(t)$.

\subsection{Additivity of constants}

In this section we want to reduce Conjecture 4.1.1 to the case where
$\Gamma$ is either a line segment or a $2$-edge connected in the
case that for any smooth point $p\in \Gamma $, the complement
$\Gamma _p$ is still connected. If $\Gamma _p$ is not connected,
then it is the union of two graphs $\Gamma _1$ and $\Gamma _2$ and
$\Gamma $ is a pointed sum  of $\Gamma _1$ and $\Gamma _2$.

\begin{lemma} Any metrized graph $\Gamma $ is a successive pointed sum of
graphs $\Gamma _i$ such that each $\Gamma _i$ is either $2$-edge
connected or an edge with  all inner points smooth.
\end{lemma}
\begin{proof}
Let $\Gamma _+$ be closure of the subgraph of points $p$ such that
$\Gamma _p$ is not connected. Then $\Gamma _+$ is a finite disjoint
union of trees, and the closed complement $\Gamma _0$ of $\Gamma _+$
in $\Gamma $ is a finite disjoint union of the maximal $2$-edge
connected subgraphs. The graph $\Gamma _+$ can be further decomposed
to edges with smooth inner points. We let $\Gamma _i$ be the
components of these $2$-edge connected points or edges with smooth
inner points.
\end{proof}

Assume that we have a decomposition of $\Gamma $ into a pointed sum
of connected subgraphs $\Gamma _i$ as in Lemma 4.3.1. For each  $i$
and each $A\in \Gamma_i $, let $\Gamma _{A}$ be the closure of the
connected component of $A$ in complement the $\Gamma _i\setminus
\{A\}$. Then for all but finitely many $A$, $\Gamma _A=A$. We have a
map $\pi _i: \Gamma \lra \Gamma _i$ with fiber $\Gamma _A$ over
$A\in \Gamma _i$. Let $q_i(A)$ be the genus of the polarized graph
$(\Gamma _A, q|_{\Gamma _A})$.

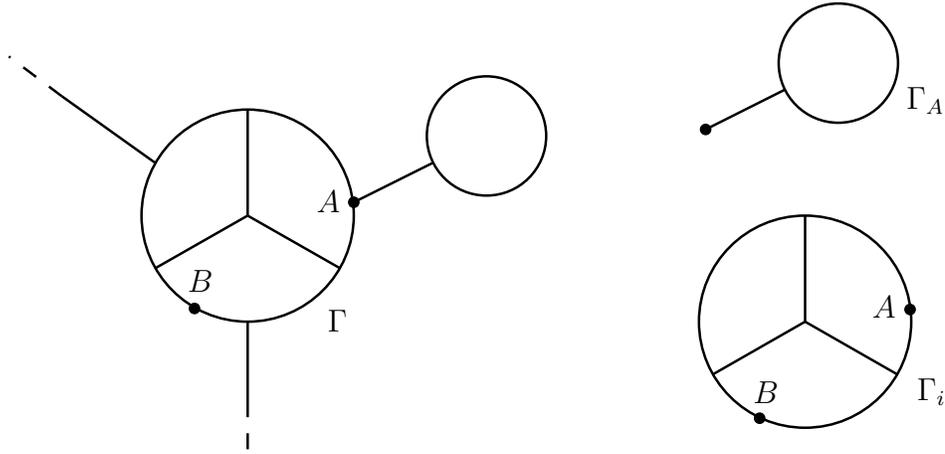
\begin{figure}[ht]

    \begin{picture}(200,185)(-100,0)
        \put(-100,0){\includegraphics{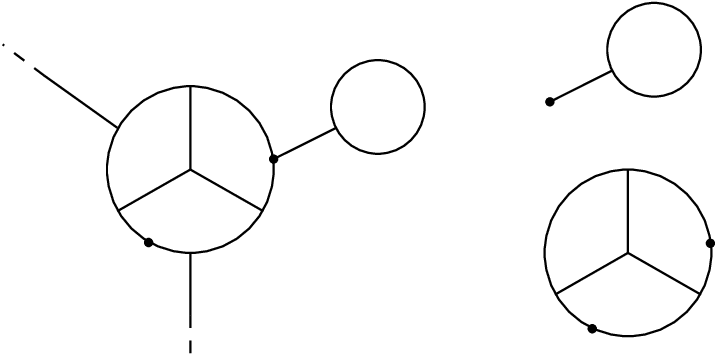}}

        \put(245,22){$\Gamma_i$}
        \put(22, 48){$\Gamma$}
        \put(18, 93){$A$}
        \put(-31, 63){$B$}
        \put(241, 132){$\Gamma_A$}
        \put(228, 53){$A$}
        \put(183, 21){$B$}
    \end{picture}

    \label{Fig: Components and fibers}
    \caption{Quotient graphs}

\end{figure}
Here is a figure of a graph $\Gamma$ that draws attention to one of
its $2$-edge connected
        components $\Gamma_i$. The point $A \in \Gamma_i$ gives rise to
        the graph $\Gamma_A$,
        which is the fiber over $A$ of the projection map
        $\pi_i: \Gamma \to \Gamma_i$. The point $B$
        satisfies $\pi_i^{-1}(B) = \{B\}$.

Our main result is as follows:

\begin{theorem} Each pair $(\Gamma _i, q_i)$ is a polarized  metrized graph
with  the same genus $g$ as $(\Gamma , q)$. Moreover all invariants
have the additivity:
$$\epsilon (\Gamma )=\sum _i\epsilon (\Gamma _i, q_i), \qquad \tau (\Gamma )=
\sum _i\tau (\Gamma _i, q_i).$$
$$\varphi (\Gamma )=\sum _i \varphi (\Gamma _i, q_i),\qquad
\lambda (\Gamma )=\sum _i\lambda (\Gamma _i, q_i).$$
\end{theorem}

\begin{proof} By definition, we need
to show that $q_i$ is non-negative and $K_{\Gamma _i}$ is effective.
By definition, the genus of $\Gamma _A$ with restriction genus
function $q(x)$ is given by $$q(\Gamma _A)=\sum _{x\in \Gamma
_A}q(x)+b(\Gamma _A)$$ where $b(\Gamma _A)$ is the first Betti
number of the topological space $\Gamma _A$. It is clear that
$q(\Gamma _A)\ge 0$.  We need to compute the degree of the canonical
divisor $K_i$ of $(\Gamma _i, q_i)$. Notice that the canonical
divisor $K_A$ of $(\Gamma _A, q|_{\Gamma _A})$ and $K$ on a point
$x\in \Gamma _A$ have the same multiplicities respectively:
$$2q(x)-2+v_{\Gamma _A}(x), \qquad 2q(x)-2+v(x).$$
These two numbers are equal except at $x=A$ where the difference is
$v_{\Gamma _i}(A)$. It follows that
$$\ord _A K_i=2q(A)-2+v_{\Gamma _i}(A)=\sum _{x\in \Gamma _A}\ord _x
K.$$ This implies that $K_i$ is effective and thus $(\Gamma _i,
q_i)$ is polarized. Take sum over $A$ to obtain that
$$2g(\Gamma _i, q_i)-2=2g(\Gamma )-2.$$
It follows that $g(\Gamma _i)=g(\Gamma )$.

For four identities, by (4.1.5) and (4.1.6), it suffices to prove
the first two. For any two points $x, y\in \Gamma$, the resistance
$r(x, y)$ can be computed using $\Gamma _i$:
$$r(x, y)=\sum _{i}r(\pi _i (x), \pi _i(y))$$
where right side is the resistance on $\Gamma _i$ which is the same
as the resistance on $\Gamma$. We have a decomposition
\begin{equation}\tau (\Gamma )=\sum_i \tau _i(\Gamma ),\qquad \epsilon
(\Gamma )=\sum \epsilon _i(\Gamma ),\end{equation} where
$$\tau _i(\Gamma )=\frac 12 \int r(\pi _i(x), \pi _i(y))
d\mu (x)d\mu (y), $$ $$\epsilon _i (\Gamma )=\int  r(\pi _i(x), \pi
_i(y))\delta _{K}(x)d\mu (y).$$ We may compute these last two
integrations over fibers of $\pi_i:\Gamma \lra \Gamma _i$:
$$\tau _i(\Gamma )=\frac 12 \int_{\Gamma _i^2} r(x, y)
d\mu _i(x)d\mu _i(y), $$ $$\epsilon _i (\Gamma )=\int  r(x,y)\delta
_{K, i}(x)d\mu _i(y),$$ where $d\mu _i(x)$ is the sum of smooth part
of $d\mu (x)$ supported on $\Gamma _i$ plus the Dirac measure $A$ in
$\Gamma _i$ with mass
$$\int _{\Gamma _A}d\mu (x)=\frac {q_i(A)}g.$$
Similarly, $\delta _{K, i}(x)$ is the Dirac measure with mass
$$\int _{\Gamma _A}\delta _K=2q_i(A)-2+v_{\Gamma _i}(x)=\deg_A K_i.$$
It follows that
$$\tau _i(\Gamma )=\tau (\Gamma _i, q_i), \qquad
\epsilon _i(\Gamma )=\epsilon (\Gamma _i, q_i).$$ The formulae
(4.3.1) thus finishes the proof.
\end{proof}

\subsection{Reduction and  elementary graphs}
In this section, we want to reduce Conjecture 4.1.1 to the case
where $G$ is $2$-edge connected. Then we prove the conjecture for
elementary graphs.

\begin{proposition} Let $D_1, \cdots D_m$ be the set of maximal
$2$-edge connected subgraphs of $\Gamma _i$. Then
$$\varphi(\Gamma)=\sum _{i\in (0, g/2]}\frac {2i(g-i)}g\ell_i(\Gamma)
+\sum _D\varphi (D, q_D).$$
$$\lambda (\Gamma)=\sum _{i\in (0, g/2]}\frac
{i(g-i)}{8g+4}\ell_i (\Gamma) +\sum _D\lambda (D, q_D).
$$
\end{proposition}
\begin{proof} By Lemma 4.3.1 and Theorem 4.3.2,  we need only prove the Proposition
when $\Gamma $ is an edge with smooth inner points. Let $i$ and
$g-i$ be values of genus function at two ends $a$ and $b$.  Then
$$K=(2i-1)a+(2g-2i-1)b, \qquad d\mu =\frac 1g(i\delta _a+(g-i)\delta
_b).$$ As $r(x, y)$ is the distance between $x$ and $y$, it follows
that
$$\tau (\Gamma )=\frac 12\int r(x, y)d\mu (x)d\mu (y)
=\frac {i(g-i)}{g^2}\ell (\Gamma ).$$
$$\epsilon (\Gamma )=\int r(x, y)\delta _K (x)d\mu (y)=
\left(4\frac{i(g-i)}g-1\right)\ell(\Gamma ).$$
The formulae in the Proposition follows from (4.1.5) and (4.1.6).
\end{proof}

\begin{corollary}
Conjecture 4.1.1 in general case follows from the case where $\Gamma
$ is $2$-edge connected.
\end{corollary}

 A graph is called {\em elementary} if every edge is included in at most
one circle. In the following, we give some explicit formulae for
$\varphi (\Gamma)$ and $\lambda(\Gamma)$ for elementary graphs and
then deduce Conjecture 4.1.1. For each circle $C$ in $\Gamma$, let
$V_C$ be the set of points on $C$ such that $q(x)>0$, and write
$C^0=C\setminus V_C$.
  Then $\Gamma\setminus C^0$ is a union of subgraphs
$\Gamma_A$ for $A\in V_C$. Let $g_A$ denote the genus of $\Gamma _A$
for the restriction of genus function $g_A$, and let $r_C(A, B)$
denote the resistance between two points $A$ and $B$ on the circle
$C$. We want to prove Conjecture 4.2.1 for elementary graph:

\begin{proposition}
$$\varphi(\Gamma)=\frac {g-1}{6g}\ell _0(\Gamma)
+\sum _{i\in (0, g/2]}\frac {2i(g-i)}g\ell_i(\Gamma) +\sum _{c\in
C}\sum_{A, B\in V_C} \frac {g_Ag_B}gr_C(A, B).$$
$$\lambda (\Gamma)=\frac g{8g+4}\ell _0(\Gamma)+\sum _{i\in (0, g/2]}\frac
{i(g-i)}{8g+4}\ell_i (\Gamma) +\sum _{c\in C}\sum_{A, B\in V_C}
\frac {g_Ag_B}{4g+2}r_C(A, B).
$$
\end{proposition}

\begin{proof} By Proposition 4.4.1, it suffices to prove Proposition for
case where $\Gamma $ is a circle. Let us compute the integrals
$\epsilon (\Gamma )$ and $\tau (\Gamma )$:
$$\tau(\Gamma )=\frac 12\int r(x, y)d\mu (x)d\mu (y), \qquad \epsilon (\Gamma )=\int
r(x, y)\delta _K(x)d\mu (y).$$ For $A, B\in \Gamma $ the resistance
$r(A, B)$ is given by $\ell (A, B)\ell ' (A, B)/\ell $ where
$\ell(A, B)$ and $\ell '(A, B)$ are the lengths of two segments of
in the complement of $A, B$ in $\Gamma $. The measures in the
integrals are given by
$$d\mu =\frac 1g(\sum _Aq(A)\delta _A+\ell ^{-1}dx),
\qquad \delta _K=\sum _A2q(A)\delta _A.$$

Let $\Gamma ^0$ be the complement of the support of $q$. Then we
have discrete contribution  when both $x$ and $y$ are not in $C^0$.
The  contributions in this case are given by
$$\tau _{A, B}(\Gamma )=\frac {q(A)q(B)}{2g^2}r(A, B),
\qquad \epsilon _{A, B}(\Gamma )=2\frac {q(A)q(B)}gr(A, B).$$

Next we consider the case where $x\notin C^0$, $y\in C^0$. We assume
that $x=A$. Let us choose coordinate $t$ on $C$ such that $t(A)=0$.
Then we have contributions:
$$ \tau_A^1(\Gamma):=\frac {q(A)}{2g} \int_{\Gamma ^0}
\frac {t(\ell _c-t)}{\ell _c}d\mu\\
=\frac {q(A)}{2g^2}\int _0^{\ell}\frac {t(\ell -t)}{\ell }\frac
{dt}{\ell } =\frac {q(A)}{6g^2}\ell ,$$
$$\epsilon_A^1(\Gamma):=2q(A) \int_{\Gamma ^0} \frac {t(\ell -t)}{\ell }d\mu =\frac
{q(A)}{3g}{\ell }.
$$

Now let us consider the case $x\in \Gamma ^0$, $y=A$. Then we have
contribution: \begin{align*} \tau_A^2(\Gamma):=&\frac 12\int
_{\Gamma ^0}\frac {t(x)(\ell -t(x))}{\ell }d\mu (x)\cdot \frac
{q(A)}g =\frac {g_A}{12g^2}\ell ,
\end{align*}
 $$\epsilon_A^2(\Gamma):=\int _{\Gamma ^0}\frac {t(\ell
-t)}{\ell }\delta _K\cdot \frac {q(A)}g=0.
$$

Finally, lets us consider the case where both $x$ and $y$ are in
$\Gamma ^0$. Then we have contribution: \begin{align*}
\tau^0(\Gamma):=&\frac 12\int _{\Gamma ^0}\int _{\Gamma ^0}\frac
{|t(x)-t(y)|(\ell -|t(x)-t(y)|)}{\ell
}d\mu (x)d\mu (y)\\
=&\frac 1{2g^2} \int _0^{\ell }\int _0^{\ell }\frac
{|t(x)-t(y)|(\ell _e-|t(x)-t(y)|)}{\ell }\frac {dxdy}{\ell ^2}
=\frac {\ell }{12g^2}.
\end{align*}

Thus a total contribution from a circle is
\begin{align*}
\tau(\Gamma)=&\sum _{A, B}\tau_{A,B}(\Gamma)
+\sum _A(\tau_A^1(\Gamma)+\tau_A^2(\Gamma))+\tau^0(\Gamma)\\
=&\sum _{A, B} \frac {q(A)q(B)}{2g^2}\cdot r(A, B)+\frac
{1}{6g^2}(\sum _A q(A))\ell +\frac {\ell }{12g^2},
\end{align*}
\begin{align*}
\epsilon(\Gamma)=&\sum _{A, B}\epsilon_{A,B}(\Gamma)+\sum
_A(\epsilon_A^1(\Gamma)
+\epsilon_A^2(\Gamma))+\epsilon^0(\Gamma)\\
=&\sum _{A, B} 2\frac {q(A)q(B)}g\cdot r(A, B)+\frac {1}{3g}(\sum _A
q(A))\ell.
\end{align*}
It is easy to verify that $\sum q(A)+1=g$. Thus we have formulae
$$
\tau_C(\Gamma) =\sum _{A, B} \frac {q(A)q(B)}{g^2}\cdot r(A,
B)+\frac {2g-1}{6g^2}\ell,
$$
$$
\epsilon_C(\Gamma) =\sum _{A, B} 2\frac {q(A)q(B)}g\cdot r(A,
B)+\frac {g-1}{3g}\ell.
$$

By formulae (4.1.5) and (4.1.6), we obtain the formulae in
proposition.
$$\varphi(\Gamma)=\left(\frac {g-1}{6g} +\sum_{A, B\in \Gamma } \frac
{q(A)q(B)}gr(A, B)\right)\ell _c,$$
$$ \lambda (\Gamma)
=\left(\frac g{4(2g+1)} +\frac 12\sum_{A, B\in V_C} \frac
{q(A)q(B)}{2g+1}r(A, B)\right)\ell _c.
$$
\end{proof}

\section{Triple product L-series and Tautological cycles}
In this section, we define a subgroup containing the Gross--Schoen
cycle  of homologous to zero cycles of codimension $2$ on  the
triple product $X^3$ of a curve $X$. The Beilinson--Bloch conjecture
relates the rank of this group and the order of vanishing of
$L$-series at $s=0$ associated to the motive $M$ defined as the
kernel
$$\wedge ^3H^2(X)\,(2)\lra H^1(X)(1).$$
We will list some formulae for  $L$-series and root numbers  in the
semistable case. At the end of this section, we want to rewrite
heights of $\Delta _\xi$ in terms of K\"unnemann's height pairing of
tautological cycles $X_1$ and $\CF (X_1)$ in the
Beauville--Fourier--Mukai theory. In particular, we can show that
the non-vanishing of height of $\Delta _\xi$ will implies the
non-vanishing of the Ceresa cycle $X-[-1]^*X$ in the Jacobian.

\subsection{Beilinson--Bloch's conjectures}
In this subsection, we define some groups of cycles homologous to
$0$ of codimension $2$ on a product of three curves and state the
Beilinson--Bloch's conjectures for corresponding motives.

Let $X_i$ ($i=1,2,3$) be three curves over a number field with three
fixed points $e_i$. We consider the triple product $Y=X_1\times
X_2\times X_3 $, the group $\Ch^2(Y)$ of cycles of dimension $1$ on
the $Y$, and the class map $$\Ch ^2(Y)\lra H^4(Y).$$ The kernel of
this map is called the group of cycles homologous to $0$ and is
denoted by $\Ch ^2(Y)^0$. We have the following Beilinson--Bloch's
conjecture \cite{beilinson1, beilinson2, bloch1}:

\begin{conjecture}[Beilinson--Bloch]
The rank of $\Ch^2(Y)^0$  is finite and is equal to the order of
vanishing of $L(H^3(Y), 2)$.
\end{conjecture}

By K\"unneth formula, we have a decomposition:
\begin{equation}
H^3(Y)=H^1(X_1)\otimes H^1(X_2)\otimes H^1(X_3)\oplus
\oplus_iH^1(X_i)(-1)^{\oplus 2}.\end{equation} Thus, the right hand
side is the product of $L$-series corresponding to the
decomposition. We would like to decompose the group $\Ch ^2(Y)^0$
into a sum of subgroups and formulate a conjecture for these
subgroups. For this, we need only find correspondence decomposition
of the identity correspondence which gives decomposition. In the
following we want to describe the group $\Ch ^2(Y)^0$ in terms of
projections and embeddings.

\begin{lemma} Let $\Ch^2(Y)_0$ be the subgroup of elements with trivial
projection onto $X_i\times X_j$ and $\Ch^1(X_i)^0$  be the group of
zero cycles on $X_i$ of degree $0$. Then
$$\Ch^2(Y)^0=\Ch^2(Y)_0\oplus \oplus _{i}(\Ch^1(X_i)^0)^{\oplus 2}.$$
Moreover, this decomposition is compatible with K\"unneth
decomposition in the sense that they are given by same
correspondences on $Y$.
\end{lemma}

\begin{proof}
Let $i, j, k$ be a reordering of $1,2,3$. For any factor $X_k$,  we
have an injection $\iota _k: X_k\lra Y$ by putting $e_i, e_j$ for
other factors; similarly we have an embedding $\iota _{i, j}:
X_i\times X_j\lra Y$ be the inclusion by putting component $e_k$ on
$X_k$. Then we have an inclusion $\Ch^0(X_k)\lra \Ch^2(Y)$ and $\Ch
^1(X_i\times X_j)\lra \Ch ^2(Y)$ by push-push forward. Let $\pi _k$
and $\pi _{i, j}$ denote the projections to $X_k$ and $X_{i, j}$.

For any cycle $Z$ on $Y$, let $Z_{i, j}$ and $Z_k$ denote push
forwards of $Z$ onto $X_i\times X_j$ and $X_k$ by $\pi _{i, j}$ and
$\pi_k$ respectively.  We define the following combinations:
$$Z^0=Z-\sum _{i,j} Z_{i, j}+\sum _kZ_k,$$
$$Z_{i, j}^0=Z_{i, j}-Z_i-Z_j.$$
Then we have a decomposition \begin{equation} Z=Z^0+\sum_{i, j}Z_{i,
j}^0+\sum _kZ_k.\end{equation}
 It can be proved
  that
 $Z^0$ has the trivial projection to $X_i\times
X_j$, and $Z_{i, j}$ has trivial projection on $X_i$ and $X_j$.
These imply that  $Z^0$ is cohomologically trivial, and $Z_{i, j}^0$
and $Z_k$ have cohomological classes in the following groups
respectively:
$$H^1(X_i)\otimes H^1(X_j)\otimes H^2(X_k),\qquad
H^2(X_i)\otimes H^2(X_j)\otimes H^0(X_k).$$

Assume now that $Z$ is homologically trivial. Then $Z_k=0$ and the
class $Z_{i, j}$ are cohomologically trivial with decomposition
$$Z_{i, j}=A\times X_j\times \{e_k\}+X_i\times B\times \{e_k\}$$
where $A$ and $B$ are divisors on $X_i$ and $X_j$ with degree $0$
respectively.
\end{proof}

The group $\Ch ^1(X_i)^0$ is nothing other than the Mordell--Weil
group of $\Jac (X_i)$. The Birch and Swinnerton--Dyer conjecture
gives
$$\ord _{s=1}L(H^1(X_i), s)=\rank \Ch ^1(X_i)^0.$$
Thus conjecture 5.1.1 is equivalent to the following:

\begin{conjecture} The rank of  $\Ch ^2(Y)_0$ is finite and is equal
to
$$\ord _{s=2}L(H^1(X_1)\otimes H^1(X_2)\otimes H^1(X_3), s).$$
\end{conjecture}

In the following we try to discuss the conjecture in the  spacial
case  $X_1=X_2=X_k=X$, where $X$ is a general curve of genus $g\ge
2$. In this case, we have more correspondences to decompose the
motive $H^1(X)^{\otimes 3}$. We will decide a submotive whose Chow
group containing the modified diagonal.

First of all, we notice that the modified diagonal is invariant
under the symmetric group $S_3$. Thus it corresponds to the
component of $H^1(X)^{\otimes 3}$ under the action of $S_3$. Notice
that the action of $S_3$ on this group is given by the following:
for $\alpha, \beta, \gamma \in H^1(X)$ then it defines an element
$\alpha (x)\wedge \beta (y)\wedge (z)$ in $H^1(X)^{\otimes 3}$. The
group $S_3$ acts by the permutations of $x, y, z$. Thus the
invariant under $S_3$ is exactly the subspace $\wedge ^3H^1(X)$ of
$H^1(X)^{\otimes 3}$. Thus the Beilinson-Bloch conjecture gives
$$\ord _{s=0}L(s, \wedge ^3H^1(X)(2))= \dim \Ch ^2(Y)_0^{S_3}.$$
Here $\Ch ^2(Y)_0^{S_3}$ is the group of cohomologically trivial
cycles with trivial projection under $\pi _{i, j}$ and invariant
under permutation. Both sides are nontrivial only if $g\ge 2$.

Using the alternative paring on $H^1(X)$, we can define a surjective
morphism
$$\wedge ^3H^1(X)(2)\lra H^1(X)(1),
\qquad a\wedge b\wedge c\mapsto a(b\cup c)+b(c\cup a)+c(a\cup b).$$
This morphism is defined by a correspondence on $X^3$ as follows:
$$X^2\lra (X^3)\times (X):\qquad (x, y)\mapsto (x, x, y)\times
(y).$$ Thus the kernel  $M$ is a motive fitted in a  splitting:
$$\wedge ^3H^1(X)(2)=M\oplus H^1(X)(1).$$

The corresponding decomposition is given by
$$\Ch^2(Y)_0^{S_3}=\Ch(M)\oplus \Pic ^0(X)(K)$$
where   $\Ch (M)$ is a subgroup of $\Ch ^2(Y^3)^0$ consists of
elements $z$ satisfying the following conditions:
\begin{enumerate}
\item $z$ is symmetric with respect to permutations on $X^3$;
\item the pushforward $p_{12*}z=0$ with respect to the projection
$$p_{12}:\quad X^3\lra X^2, \qquad (x, y, z)\mapsto (x, y).$$
\item let $i: X^2\lra X^3$ be the embedding defined by $(x, y)\lra
(x, x, y)$ and $p_2: X^2\lra X$ be the second projection. Then
$$p_{2*}i^*z=0.$$
\end{enumerate}

For any $\eta \in \Jac (X)(K)$, the corresponding element in
$\Ch^2(Y)_0^{S_3}$ is given  by
$$\alpha (\eta )=\sum _{i, j, k}\Delta _{i, j}^0\times \eta _k$$
where $\eta _k\in X_k$ is corresponding to $\eta$.

The Beilinson-Bloch conjecture gives the following
\begin{conjecture} The group $\Ch(M)$ has finite rank and
$$\ord _{s=0}L(s, M)=\dim \Ch (M).$$
\end{conjecture}

Let us check if the modified diagonal is in the above group:

\begin{lemma}
$$\Delta _\xi\in \Ch (M).$$
\end{lemma}

\begin{proof} Indeed, it is easy to show that
$$i^*\Delta _\xi =(2-2g(X))\xi _\Delta
-(2-2g)\xi\times \xi-2\xi _\Delta+2\xi \times \xi .$$ It is clear
that
$$p_{2*}i^*\Delta _\xi =(2-2g)\xi-(2-2g)\xi-2\xi+2\xi=0.$$
\end{proof}

 \subsection{L-series and  root numbers}
 In this section we want to compute $L$-series and the epsilon factor
 for $L(s, M)$ when the curve has semi-stable reduction. Our
 reference for definitions is Deligne \cite{deligne2}.
 For convenience, we will work on homology motive
 $H_1(X)=H^1(X)(1)$.
 Recall that  $M$ is the kernel of a canonical surjective morphism
 motives:
 $$\wedge ^3H_1(X)(-1)\lra H_1(X),$$
 It follows that the motive $M$ is of weight
 $-1$ with a non-degenerate alternative pairing
 $$M\otimes M\lra \BQ(1).$$
 It is conjectured that the $L$-series $L(s, M)$ should be entire
 and satisfies a functional equation
 $$L(s, M)=\pm f(M)^{-s}L(s, M)$$
 where $f(M)\ge 1$ is the conductor of $M$ (an integer divisible
 only by finite places ramify in $M$).

\subsubsection*{Local L-functors}
 By definition, the L-series is defined by an Euler product:
 $$L(s, M)=\prod _v L_v(s, M)$$
 where $v$ runs through the set of places of $K$, and $L_v(s, M)$
 is a  local $L$-factor of $M$ at $v$.  For $v$ an archimedean place,
 the local $L$-factor is determined by the Hodge weights. Notice
 that we have a decomposition
 $$H_1(X, \BC)=H^{-1, 0}(X, \BC)\oplus H^{0,-1}(X, \BC)$$
 of Hodge structure into two spaces of dimension $g$, and that $\BC(-1)$ has Hodge
 weight $(1, 1)$. As $M$ is the kernel of a surjective morphism of Hodge structure
 $$\wedge ^3H_1(X, \BC)(-1)\lra H_1(X, \BC),$$
 it follows
 that $M$ has Hodge numbers given by
 $$h^{1,-2}=h^{-2, 1}=
 \frac {g(g-1)(g-2)}6, \quad h^{0,-1}=h^{-1,0}
 =\frac {g(g-2)(g+1)}2.$$
 The $L$-factor then is given by
 \begin{equation}
 L_v(s, M)=\Gamma _\BC(s+2)^{h^{-2, 1}}\Gamma _\BC(s+1)^{h^{-1,
 0}}, \qquad \Gamma _\BC=2\cdot (2\pi )^{-s}\Gamma (s).
 \end{equation}

 For $v$ a finite place with
 inertia group $I_v$, residue field $\BF _{q_v}$,
 and geometric Frobenius $F_v$, the L-series is
 given by
 \begin{equation}
 L_v(s, M)=\det (1-q_v^{-s}F_v; M_\ell^{I_v})^{-1}
 \end{equation}
 where $M_\ell$ is the $\ell$-adic realization of $M$ at a prime
 $\ell\nmid q_v$.
 For $v$ unramified, the L-series can be computed simply by Weil
 numbers. For $v$ a ramified place, then $\Jac (X)$ has a
 semi-abelian reduction: the connected component $J$ of the
  Neron model of $\Jac (X)$ is an extension of an abelian variety
  $A$ by a torus
  $$0\lra T\lra J\lra A\lra 0.$$
  Here $A$ is the product of Jacobians of the irreducible components
  in the semistable reduction of $X$ and $T$ is a torus determined
  by homology group in the reduction graph of $X$. We have
  a filtration of $V:=H_1(X_{\bar K}, \BQ_\ell)$:
  $$H_1(\bar T, \BQ_\ell)\subset H_1(\bar J, \BQ_\ell)\subset
  H_1(\bar X, \BQ_\ell).$$
  This filtration is compatible with action of the decomposition
  group $D_v$. By Serre--Tate, we have an identity:
  $$H_1(\bar J, \BQ_\ell)=H_1(\bar X, \BQ_\ell)^{I_v},$$
  and $H_1(\bar T, \BQ_\ell)$ is the orthogonal complement of
  $H_1(\bar J, \BQ_\ell)$ with respect to the Weil pairing on
  $H_1(\bar X, \BQ_\ell)$. In particular, the action of $F_v$ on
  these space  are
  semiample with eigenvalues of absolute value $q^{-1}$, $q^{1/2}$,
  and $1$. Thus $F_v$ on $H_1(\bar X, \BQ_\ell)$ is semi-simple.

  By Grothendieck, the action of $I_v$ on $H_1(\bar X, \BQ_\ell)$ is
  given by
  $$\sigma x=x+t_\ell (\sigma) Nx, \qquad x\in H_1(\bar X, \BQ_\ell),
  \quad N\in \End
  (H_1(\bar X, \BQ_\ell)),$$
  where $t_\ell: I_v\lra \BQ_\ell$ is a nonzero homomorphism.
  We may decompose $V:=H^1(\bar X, \BQ_\ell)$ into an orthogonal sum
  of two dimensional spaces $V_i$ ($i=1, \cdots g$)
  invariant under $D_v$.
  The $\wedge^3 H_1(\bar X, \BQ_\ell)$ is then a direct sum of tensors
$$\wedge ^3H_1(\bar X, \BQ_\ell)
=\oplus_{n_1+n_2\cdots =3} \wedge ^{n_1}V_1\otimes \wedge
^{n_2}V_2\otimes \cdots.$$ The invariants of $I_v$ must have
decomposition:
$$\wedge^3 H_1(\bar X, \BQ_\ell)^{I_v}
=\oplus_{n_0+n_1+n_2\cdots =3} (\wedge ^{n_0}V_0)^{I_v}\otimes
(\wedge ^{n_1}V_1)^{I_v}\otimes \cdots.$$ Thus $M(1)$ has the
following orthogonal decomposition of $D_v$-modules:
  \begin{equation}
  M(1)=\sum _{i<j<k}V_i\otimes V_j\otimes V_k+\sum _i
V_i\otimes (\sum _{j\ne i}(\wedge ^2V_i))^0\end{equation}
 where superscript $0$
means kernel in the Weil pairing. The space $M^{I_v}$ has a
decomposition \begin{equation} M^{I_v}(1)=\sum
_{i<j<k}V_i^{I_v}\otimes V_j^{I_v}\otimes V_k^{I_v}+\sum _i
V_i^{I_v}\otimes (\sum _{j\ne i}\wedge ^2V_i)^0 .\end{equation}  In
this way, we have a precise description of the Galois action on $M$
and therefore a formula for $L$-factor.

\subsubsection*{Local root numbers}

  In the following we want to compute the root numbers of the
  functional equation. Recall that the root number $\epsilon$ is the
product of local root numbers $\epsilon_v$.

\begin{lemma}
For $v$ complex we have
$$\epsilon _v =i^{6h^{-2,1}+2h^{-1, 0}}=\begin{cases}
1,& g\equiv 0, 1 \pmod 4\\
-1,& g\equiv 2,3 \pmod 4\end{cases} $$
 For $v$ a real place
$$\epsilon _v=i^{4h^{-2,1}+2h^{-1, 0}}=\begin{cases}
1, & g\not\equiv 1 \pmod 4\\
-1, &g\equiv 1\pmod 4.\end{cases}
$$
\end{lemma}

\begin{lemma}
Let $v$ be a non-archimedean place. Let $\tau=\pm 1$ be the product
of $\alpha _i$. Then the root number is given by
\begin{align*}
\epsilon _v:=&(-1)^{e(e-1)(e-2)/6}\tau
^{(e-1)(e-2)/2}(-1)^{e(g-2)}\tau ^{(g-2)}\\
=&(-1)^{e(e-1)(e-2)/6+ge}\tau ^{(e-1)(e-2)/2+g}
\end{align*}
Here $e$ is the rank of the first homology group of the reduction
graph of $X$ at $v$, and $\tau$ is the determinant of $F_v$ acts on
the character group of ($e$-dimensional ) toric part of the
reduction of $\Jac (X)$.
\end{lemma}

\begin{proof}
If $v$ is finite unramified place, then $\epsilon _v=1$. It remains
to compute the root number at a ramified finite place. It is given
by
$$\epsilon _v=\frac {\det (-F_v|_{M_\ell })}{\det
(-F_v|_{M_\ell ^{I_v}})}.$$

Now we want to compute $\epsilon_v$ using  decompositions (5.2.3)
and (5.2.4). Notice that on each $V_i$, $-F_v$ has determinant
$-q^{-1}$, and on $V_i^{I_v}$, it has eigenvalues $q^{-1}$ for $i\le
e$. We assume that $V_i\ne V_i^{I_v}$ exactly for the first $e$
$V_i$'s. Let $F_v$ have eigenvalues $\alpha _i$ on $V_i/V_i^{I_v}$
which has absolute value $1$. The contribution to root number from
each term is given as follows: \begin{align*} V_i\otimes V_j\otimes
V_k: &\qquad 1&\\
V_i\otimes (\sum _{j\ne i}\wedge ^2V_i): &\qquad 1&,\\
V_i^{I_v}\otimes V_j^{I_v}\otimes V_k^{I_v}: &\qquad -\alpha
_i\alpha _j\alpha _k, &\qquad i<j<k\le e,\\
V_i^{I_v}\otimes V_j^{I_v}\otimes V_k^{I_v}: &\qquad \alpha
_i^2\alpha _j^2, &\qquad i<j\le e<k,\\
V_i^{I_v}\otimes V_j^{I_v}\otimes V_k^{I_v}: &\qquad \alpha _i^4,
&\qquad i\le e<j<k,\\
V_i^{I_v}\otimes (\sum _{j\ne i}\wedge ^2V_i): &\qquad (-\alpha
_i)^{g-2},&\qquad i\le e.
\end{align*}
\end{proof}

\subsection{Tautological classes in Jacobians}
In this subsection, we would like to study tautological algebraic
cycles in the Jacobian defined by Ceresa \cite {ceresa} and
Beauville \cite{beauville3}. We will use Fourier--Mukai transform of
Beauville (\cite{beauville1, beauville2}) and height pairing of
K\"unnemann (\cite{kunnemann2}).

Let $A$ be an abelian variety of dimension $g\ge 3$  over a global
field $k$ with a fixed symmetric and ample line bundle $\CL$. Let
$L$ be the operator on motive $h(A)$ induced by intersecting with
$c_1(\CL)$ which thus induces operator on Chow group and cohomology
group. For each integer $p$ in the interval $[0, (g+1)/2]$, it is
conjectured that the map
$$L^{g+1-2p}: \quad \Ch ^p(A)_\BQ^0\lra \Ch ^{g+1-p}(A)_\BQ^0$$
is an isomorphism of two vector spaces of finite dimensional. Let
$\Ch ^p(A)^{00}_\BQ$ denote the kernel of $L^{g+2-2p}$ which is
called {\em the group of primitive class of degree $p$}. By the same
way, we can define the primitive cohomology classes
$H^{2p-1}(A)^{00}$. Then the Beilinson--Bloch conjecture says that
$$\rank \Ch ^p(A)^{00}_\BQ=\ord _{s=0}L(H^{2p-1}(A)^{00}(p), s).$$
Moreover, K\"unnemann has constructed a height pairing on
$\Ch^*(X)_\BQ^0$:
$$\pair{\cdot, \cdot}:\qquad \Ch ^p(X)_\BQ^0\otimes  \Ch
^{g-p+1}(X)_\BQ^0\lra \BR.$$ The index conjecture of Gillet--Soul\'e
says
$$(-1)^p\pair{x, L^{g+1-2p}x}>0, \qquad 0\ne x\in \Ch
^p(X)^{00}_\BQ.$$

Using Mukai--Fourier transform, we may decompose the group $\Ch
^p(A)$ into a direct sum of eigen spaces under multiplications:
$$\Ch^p(A)_\BQ=\sum_s \Ch ^p_s(A)$$
where $s$ are integers and $\Ch ^p_s(A)$ is the subgroup of cycles
$x\in\Ch ^p(A)$ with the property
$$[k]^*x=k^{2p-s}x, \qquad \forall k\in \BZ,$$
where $[k]$ is the multiplication on $A$ by $k$. It has been
conjectured that $\Ch ^p_s(A)=0$ if $s\ne 0, 1$.  By the projection
formula
$$\pair{k^*x, y}=\pair{x, k_*y}$$
we see that $\Ch _s^p(A)^0$ are perpendicular to $\Ch _t^q(A)^0$
unless
$$p+q=g+1, \qquad s+t=2.$$

Let $X$ be a curve over a global field $k$ with Jacobian  $J$. For
an integer $n\in [0, g]$,  we can define morphism $f_n: \quad
X^n\lra J$ by sending $(x_1,\cdots, x_n)$ to the class of $\sum
(x_i-\xi)$. Notice that the image does not depend on the choice of
$\xi$. We view $X$ as a subvariety of $J$ via embedding $f_1$ and
define the theta divisor $\theta$  as the image of  $f_{g-1}$. We
use $\theta$ for the primitive decomposition and
 Fourier--Mukai transform:
$$\CF: \Ch ^*(J)_\BQ\lra \Ch ^*(J)_\BQ$$
$$x\mapsto \CF (x):=p_{2*}(p_1^*x\cdot e^\lambda)$$
where $\lambda$ is the Poincar\'e class:
$$\lambda =p_1^*\theta+p_2^*\theta -m^*\theta.$$
The decomposition into $s$-space can be made explicit as follows:
Define a decomposition $\CF=\sum \CF _s$  by
$$\CF _s(x)=p_{2*}\left(p_1^*x\cdot \frac {\lambda ^{2g-2p+s}}{(2g-2p+s)!}\right).$$
Then we have decomposition $x=\sum x_s$ with $$x_s=\CF^{-1}(\CF
_s(x))\in \Ch ^p_s(J)_\BQ,$$ where $\CF^{-1}$ is the inverse of
$\CF$ which has an expression:
$$\CF ^{-1}=(-1)^g[-1]^*\circ \CF.$$

 Following Beauville, we define the ring $\CR$ of tautological cycles of $\Ch ^*(J)$
 as the smallest $\BQ$-vector generated by  $X$ under the following operations:
 the intersection, the star operator, and the Fourier--Mukai
 transform. By Beauville, in the decomposition $\CR=\oplus _s\CR_s$,
 $\CR_s=0$ if $s<0$ and $\CR_0$ is generated by $\theta$. Thus $\CR_0$
 maps injectively into cohomology group. Thus cohomological trivial
 cycles have components $s>0$. The height intersection on these
 cycles factors through the first component:
 $$\pair {x, y}=\pair{x_1, y_1}.$$

The key to prove Theorem 1.5.5 is the following pull-back formula:
\begin{theorem} Consider the morphism $f_3: X^3\lra
J$. Then
$$f_3^*\CF (X)=-g\sum _ip_i^*\xi-\sum _{ij}p_{ij*}\delta _\xi+\Delta _\xi,$$
where $\delta _\xi$ is the class
$$\delta _\xi=p_1^*\xi+p_2^*\xi-\Delta\in \Ch ^1(X^2).$$
\end{theorem}

\begin{proof} By discussion above,
$$\CF (X)=p_{2*}\left(p_1^*X\cdot e^{\lambda}\right).$$ Consider the morphism
$$g: X^4\lra J\times J, \qquad (x_i)\mapsto (x_0-\xi, x_1+x_2+x_3-3\xi)$$
Then it is easy to see that
$$f_3^*\CF (X)=p_{123*}g^*e^\lambda=p_{123*}\exp g^*\lambda.$$
Let us compute the class $g^*\lambda$:
$$g^*\lambda =p_0^*\theta +p_{123}^*f_3^*\theta -f_4^*\theta.$$
We want to use the theorem of cube to decompose this bundle into a
sum of pull-backs of bundles of a face $X^2$ of $X^4$. More
conveniently, we may consider this bundle as pull-back of bundle on
$A^4$ of the following bundle: $$m_0^*\theta +m_{123}^*\theta
-m_{0123}^*\theta$$ where for a subset $I$ of $\{0,1,2,3,4\}$, $m_I$
is the sum of elements in $I$. By the theorem of cube, this bundle
has an expression
$$\sum _{ij}\CL_{ij}+\sum _i\CM_i$$
where $\CL_{ij}$ are line bundles on $J^2$ with trivial restriction
on $\{0\}\times J$ and $J\times \{0\}$ and $\CM_i$ are line bundles
on $J$. Now lets us restrict the bundle on $ij$-factors with $0$ on
other factors to obtain:
$$\CL_{0i}=\lambda, \quad \CL _{ij}=0, \quad \forall i, j>0.$$
Similarly, restrict on a single factor to get $\CM_i=0$. In summary,
we have shown that
$$g^*\lambda =\sum _{i=0}^3 f_{0i}^*\lambda$$
where $f_{0i}$ is the projection $X^3\lra A^2$. To compute the
bundle $f_{0i}\lambda$ we consider the embedding $X^2\lra A^2$. It
is easy to see that the restriction of $\lambda$ is given by
$\delta_\xi$. It follows that
$$g^*\lambda=\sum _i p_{0i}^*\delta _\xi.$$

Thus we have
$$f_3^*\CF (X)=p_{123*}\exp g^*\lambda=\sum _{ijk}\frac 1{i!j!k!}
p_{123*}(p_{01}^*\delta _\xi^i\cdot p_{02}^*\delta _\xi^j\cdot
p_{03}^*\delta _\xi^k).$$ The identity in Theorem follows from a
direct computation. \end{proof}

\subsubsection*{Proof of Theorem   1.5.5}
The first formula follows from Theorem 5.3.1. The second follows
form the identity
$$f_{3*}\Delta _\xi=[3]_*X-3[2]_*X+3X, \qquad X=\sum X_s.$$
For the third formula, we notice the star operator and intersection
operator respect to the $s$-graduation. Push the first formula in
the Theorem   to $J$ to obtain:
$$X^{*3}\cdot \CF (X_1)=[3]_*X-3[2]_*X+3X.$$
Decompose this into $s$-components to obtain:
$$\CF (X_1)\cdot \sum _{i+j+k=s-1}X_i*X_j*X_k =\left(3^{2+s}-3\cdot
2^{2+s}+3\right) X_s.$$ This proves the identity in the third
formula. The list of equivalence is clear by three identities and
the following expression for Ceresa cycle:
$$X-[-1]_*X=2\sum _{s \,\textrm{odd}}X_s.$$

\subsubsection*{Proof of Theorem 1.5.6}
By Theorem 1.5.5, $f_3^*\CF (X_1)=\Delta _\xi$. The first inequality
follows from the projection formula:
$$\pair{\Delta _\xi, \Delta _\xi}_{X^3}=
\pair{\CF (X_1), (f_{3*}\Delta _\xi)_1}_{X^3}.$$ Now we use the
identity
$$f_{3*}\Delta _\xi=[3]_*X-3[2]_*X+3X=6X_1+\cdots.$$
For the second inequality, we use
 another projection formula \begin{align*} \pair{\Delta _\xi, \Delta
_\xi}_{X^3} =&\pair {f_3^*\CF (X_1), f_3^*\CF (X_1)}_{X^3}
=\pair{\CF (X_1), f_{3*}f_3^*\CF (X_1)}_J\\
=&\pair{\CF (X_1), X^{*3}\cdot \CF (X_1)}.\end{align*} As the
intersection pairing depends only on the $s=1$ component, we may
replace $X^{*3}$ by
$$X^{*3}_0=\frac 6{(g-3)!}\theta ^{g-3}.$$
Here for a subvariety $Y$ of $X$, $Y^{*d}$ denote $d$-th star
product power of $Y$. This proves the identity in the Theorem. To
show that $\CF (X_1)$ is primitive, we use the following identity:
$$L\cdot L^{g-3}\CF (X_1)=\frac {(g-3)!}6\left(\theta \cdot
X^{*3}\CF (X_1)\right)_1 =\frac
{(g-3)!}6f_{3*}\left(f_3^*\theta\cdot \Delta _\xi\right).$$ Thus it
suffices to prove
$$f_3^*\theta\cdot \Delta _\xi=0.$$
By Theorem 5.3.1, $$f_3^*\theta =-f_3^*\CF (X_0) =g\sum i
p_i^*\xi+\sum _{ij}p_{ij}^*\delta_\xi.$$ It is easy to show all of
these terms have zero intersection with $\Delta _\xi$.

\end{document}